\documentclass[11pt,twoside]{article}%
\usepackage{amssymb}
\usepackage{amsmath}
\usepackage{amsfonts}
\usepackage{mathtools}
\usepackage[title,titletoc,header]{appendix}
\usepackage{geometry}
\usepackage{graphicx}
\usepackage{indentfirst}
\usepackage{mathrsfs}
\usepackage{nopageno}
\usepackage{setspace}
\usepackage[nottoc]{tocbibind}
\usepackage{hyperref}%
\newtheorem{theorem}{Theorem}[section]

\newtheorem{axiom}[theorem]{Axiom}

\newtheorem{conclusion}[theorem]{Conclusion}

\newtheorem{corollary}[theorem]{Corollary}

\newtheorem{definition}[theorem]{Definition}
\newtheorem{example}[theorem]{Example}

\newtheorem{lemma}[theorem]{Lemma}

\newtheorem{proposition}[theorem]{Proposition}
\newtheorem{remark}[theorem]{Remark}

\newenvironment{proof}[1][Proof]{\noindent\textbf{#1.}
\setlength{\parindent}{1ex} }{\hfill \rule{0.4em}{0.4em}}
\makeatletter
\@removefromreset{figure}{section}
\makeatother
\counterwithin{equation}{section}

\begin{document}

\title{Generalized Von Neumann Universe and Non-Well-Founded Sets}
\author{Eugene Zhang}
\maketitle

\begin{abstract}
\medskip In this paper, a generalized version of the von Neumann universe known as the total universe is proposed to formally introduce non-well-founded sets that include infinitons, semi-infinitons and quasi-infinitons in Russell's paradox. All three infinitons are part of infinitely generated sets that are generators of non-well-founded sets. Combining the well-founded sets with the non-well-founded sets, the total universe is a model of ZF minus the axiom of regularity and free of Russell's paradox. The axiom of regularity can not define the well-founded sets and is invalid in any system consistent with ZF set theory.\footnote{A summary of the paper is contained in Appendix A.} \bigskip


\end{abstract}
\tableofcontents
\listoffigures

\section{Preliminaries}

\subsection{Introduction}

The investigation of non-well-founded sets began with the work of Mirimanoff in 1917 \cite{Mirimanoff}. A number of axiomatic systems of non-well-founded sets such as AFA (by Aczel, Forti and Honsell \cite{Aczel}), SAFA (by Scott), FAFA (by Finsler), and BAFA (by Boffa), have been proposed thereafter. These systems introduce non-well-founded sets by replacing the axiom of regularity with separate anti-foundation axioms. The main problem with these systems is that they lack precise mathematical descriptions for non-well-founded sets. As a result, non-well-founded sets are not rigorously defined and the exact process to generate them is unclear.\smallskip

In this paper, we will present a model for precisely defining the non-well-founded sets based on the notion of limit of formulas and the enlarged von Neumann universe ($V$). First, we show that $V$ is incomplete because it does not have limit ordinal ranks, a fact that is of fundamental importance because it implies that non-well-founded sets necessarily exist and should be added to $V$ as the limit ordinal ranks. Furthermore, limits of finite structures and formulas along with an algebra to handle the limit operations are given to provide enough mathematical rigor for describing non-well-founded sets. Consequently, the expanded universe of sets known as the total universe (\ref{TotalHierarchy}) is shown to be a model of ZF minus the axiom of regularity and free of Russell's paradox. The axiom of regularity is invalid in any system consistent with ZF set theory and even fails to define the well-founded sets (section \ref{SectionConclusion}). 

\subsection{Problems in Von Neumann Universe}

The von Neumann universe (also known as the cumulative hierarchy) is well known as the class of hereditary well-founded sets and is defined as follows,
\begin{align}
	V_{0}\, & =\,\varnothing\mathfrak{;}\nonumber  
	\\
	V_{\alpha}\, & =\,\mathcal{P}(V_{\alpha-1})\text{,\qquad\qquad\ }\alpha\text{ is any successor ordinal;}\nonumber
	\\
	V_{\alpha}\, & =\underset{\beta<\alpha}{\bigcup}V_{\beta},\text{\qquad\qquad\quad \ }\alpha\text{ is any limit ordinal;} \nonumber  
	\\
	V \, & =\,\,\smashoperator{\bigcup_{\alpha\in \mathrm{Ord}}} V_{\alpha}. \label{CumulativeHierarchy}
\end{align}

The structure of any set $S$ can be represented by a \textbf{tree}, in which $S$ can be regarded as the \textbf{root} and all the objects in the transitive closure of $S$ form the \textbf{nodes} of the tree \cite{Kechris}. A \textbf{branch} (or path) of the tree is a sequence of nodes connected by \textquotedblleft$\in$\textquotedblright\ from the root to an end node known as a \textbf{terminal}. Clearly, the only terminal in $V\,$is $\varnothing$. A finite branch consists of a finite number of nodes, while an infinite branch contains an infinite number of nodes. \medskip

A \textbf{transfinite sequence} $\gamma_{\alpha}=\left\langle \gamma_{\xi}\colon\xi\leqslant\alpha\right\rangle $\ is a function with an ordinal domain where $\alpha$ is its \textbf{length }\cite{Jech}.\ A $\in$-\textbf{sequence} is a transfinite sequence $\gamma_{\alpha}$ that $\gamma_{0}\in\dots\in\gamma_{\xi}\in\gamma_{\xi+1}\in \dots\in\gamma_{\alpha}$. Obviously, any branch of $S$ in $V$ can be represented by a $\in$-sequence like $\varnothing=\gamma_{0}\in\dots\in\gamma_{\alpha} =S$. As a result, well-founded and non-well-founded sets can be defined upon $\in$-sequences as follows.

\begin{definition}
\label{DefWFAndNWFSet} \ Suppose $S$ is a set (with $\varnothing$ as the only terminal) and $\gamma_{\alpha}=\left\langle \gamma_{\xi}\colon\xi\leqslant\alpha\right\rangle $ is a $\in$-sequence in $S$. Then $S\,$\textit{is \textbf{well-founded} (WF) if any }$\gamma_{\alpha}$\textit{\ of }$S$ \textit{has }$\alpha<\omega.$ $S\,$\textit{is \textbf{non-well-founded} (NWF) if one }$\gamma_{\alpha}$\textit{ of }$S$ \textit{has }$\alpha \geqslant\omega.$ \textit{If all }$\gamma_{\alpha}$\textit{\ of }$S$ \textit{have }$\alpha\geqslant\omega$\textit{, }$S\ $\textit{is \textbf{totally non-well-founded} (TNWF).}
\end{definition}

From definition \ref{DefWFAndNWFSet}, it follows easily that $V\,$consists of only well-founded sets.

\begin{lemma}
\label{1} \ $V$ \textit{is well-founded and no set in }$V$ is \textit{non-well-founded.\footnote{In this paper, an existing theorem in set and model theory is listed as a proposition, while a (mainly) new result is proved as a theorem. Lemmas, corollaries, conclusions and axioms can have both new and existing results.}}
\end{lemma}

\begin{proof}
\ By definition \ref{DefWFAndNWFSet}, we only need to prove (by transfinite induction) that any $\in$-sequence\textit{\ }in\textit{\ }$V$ is of finite length. First, any $\in$-sequence\textit{\ }in\textit{\ }$V_{1}$ is of finite length for $V_{1}= \{\varnothing\}$. Suppose any $\in$-sequence\textit{\ }$Z_{\xi}\in V_{\beta}$ has length $\xi<\omega$ for $\beta<\alpha$. If $\alpha$ is a successor ordinal and $Z_{\xi}\in V_{\alpha-1}$ has $\xi<\omega$, then $Z_{\xi+1}$ is a $\in$-sequence in $V_{\alpha}$ and $\xi+1<\omega$. If $\alpha$ is a limit ordinal, for any $X\in V_{\alpha}=\underset{\beta<\alpha}{\displaystyle\bigcup}V_{\beta}$, there is a $\gamma<\alpha$ that $X\in V_{\gamma}$. Since any $\in$-sequence $Z_{\xi}\in V_{\gamma}$ has $\xi<\omega$, so is $X.$ Thus $V$ is WF.\medskip

If $X\in V$ is NWF, $X$ has a $\in$-sequence $Z_{\xi}$ with length $\xi\geqslant\omega$, contradicting $V$ being WF.\bigskip
\end{proof}

Rank in $V$ is defined as follows.

\begin{definition}
\label{DefCumulativeHierarchyRank}\ The rank of $X$ in $V$ is the least $\alpha$ that $X\in V_{\alpha+1}$ (or equivalently $X\subset V_{\alpha}$).\footnote{A survey on rank in set theory is given in \cite{Knoche}.}
\end{definition}

This definition of rank appears to be originated by Mirimanoff \cite{Mirimanoff}, developed by Bernays \cite{Bernays}, and given its current form by Tarski \cite{Tarski}. Nonetheless, it is erroneous for the following reasons. Rank in a universe of the sets is a function $R$ mapping each set to a unique ordinal number and satisfies the property of monotonicity, i.e. for any $Y\in X$, $R(Y)<R(X)$, and $R(\{X\})=R(X)+1$.\smallskip

First, it is reasonable to believe that a set with one infinite branch (rather than infinite splittings) has the rank of a limit ordinal. For example, suppose $I_{n}=\underset{n}{\underbrace{\{\dots\{\varnothing\}\dots\}}}.$ Then $I_{n}$ contains a $\in$-sequence of length $n$, $\varnothing\in\{\varnothing\}\in\dots\in I_n$. By the monotonicity of $R$, we have
\[
R(I_{n})\,=\,R(\{I_{n-1}\})\,=\,R(I_{n-1})+1\,=\,R(\varnothing)+n
\]  
As $n$ approaches infinity, $R(I_{\omega}) = \omega$.\footnote{$I_{\omega}$ is called an infiniton as in definition \ref{DefInfiniton}.} Therefore, the rank $\omega$ should belong to non-well-founded sets like $I_{\omega}$ that has an infinite branch rather than $\omega$ under the current definition of rank. On the other hand, a WF set with infinite splittings like $\omega$ should not have a rank of a limit ordinal because all of its branches are of finite length. More specifically, even if for any $n\in\omega$, $R(n)<R(\omega)$, we can not conclude that $R(\omega)=\omega$ because $\omega+1,\omega+2,\dots$ also qualify. \medskip

Furthermore, the evaluation of rank should be different for $V_{\alpha}$ with a successor and limit ordinal. But in definition \ref{DefCumulativeHierarchyRank}, the evaluation is the same because it enforces all sets to be contained only in layers of successor ordinals and never involves layers of limit ordinals ($\alpha+1$ is always a successor ordinal). As a result, definition \ref{DefCumulativeHierarchyRank} is fallacious. 

\subsection{Correct Rank in Von Neumann Universe}

From the previous discussion, we adopt a more natural and correct version of rank in $V$ as follows.

\begin{definition}
\label{DefModifiedCumulativeHierarchyRank}\ The \textbf{rank} of $X$ in $V$ is defined as the least $\alpha$ that $X\in V_{\alpha}$ and denoted as $R_{V}(X)$, i.e. $R_{V}(X)=\,\smashoperator{\inf\limits_{\alpha\in\mathrm{Ord}}}\,\,\,\{\alpha\colon X\in V_{\alpha}\}$.
\end{definition}

\begin{lemma}
\label{2} \ No set in $V$ has a rank of a limit ordinal.
\end{lemma}

\begin{proof}
\ Suppose $\alpha$ is a limit ordinal and $R_{V}(X)=\alpha$. Then there is a $\gamma<\alpha$ that $X\in V_{\gamma}$ or $R_{V}(X)<\alpha$, contradiction.\smallskip
\end{proof}

\begin{corollary}
\label{3} \ \textit{For any von Neumann ordinal }$\alpha\in V,\,R_{V}(\alpha)\,=\,\alpha+1$.
\end{corollary}

\begin{proof}
\ We prove by transfinite induction. First, since $\varnothing\in V_{1},\,R_{V}(\mathbf{0})=1$. If $\alpha$ is a successor ordinal, suppose $R_{V}(\alpha)=\alpha+1$, i.e. $\alpha\in V_{\alpha+1}$ and $\alpha\notin V_{\alpha}$. Then $\alpha+1= \alpha\cup\{\alpha\}\in\mathcal{P(}V_{\alpha+1})=V_{\alpha+2}$. Also, $\alpha+1\notin V_{\alpha+1}$ for otherwise $\alpha+1 \subset V_{\alpha}$, which means $\alpha\in V_{\alpha}$, contradiction. So $R_{V}(\alpha+1)\,=\,\alpha+2.$\medskip

If $\alpha$ is a limit ordinal, for any $\gamma<\alpha,\,\gamma\in V_{\gamma+1}\subset V_{\alpha}$, i.e. $\alpha\subset V_{\alpha}$ and $\alpha\in V_{\alpha+1}.$ If $\alpha\in V_{\alpha}=\displaystyle\bigcup\limits_{\beta<\alpha}V_{\beta}$, there is a $\gamma<\alpha,\,\alpha\in V_{\gamma}$, contradiction. Thus $R_{V}(\alpha)\,=\,\alpha+1$.\smallskip
\end{proof}

\begin{remark}
\ Corollary \ref{3} shows that the rank of any ordinal in $V$ is a successor ordinal, which is consistent with the fact that $V$ contains only well-founded sets and no set in $V$ has a rank of a limit ordinal.
\end{remark}

Next, we will introduce the notion of the unpacking operator that is important for the rest discussion in this paper.

\subsection{Unpacking Operator and Nullity}\label{SectionUnpackingOperator}

\begin{definition}
\label{DefUnpackOperator}\ Suppose $G=\{a_{1},a_{2},\dots\}.$ The \textbf{unpacking operator} $\ast G$ of $G$ is defined as $\{\ast G\}=G,$ i.e. $\ast G=a_{1},a_{2},\dots$.
\end{definition}

\begin{remark}
\ Intuitively, the unpacking operator can be considered as removing the curly brackets of a set, and $\ast G$ as the collection of $a_{i}$ without the curly brackets.
\end{remark}

\begin{example}
\ Let $S=\{a_{1},a_{2},\dots,b_{1},b_{2},\dotsc\}$, $G_{1}=\{a_{1},a_{2},\dotsc\}$ and $G_{2}=\{b_{1},b_{2},\dotsc\}.$ Then $S=\{\ast G_{1},\ast G_{2}\}.$
\end{example}

The unpacking of the empty set $\ast\varnothing$ is of particular importance because it represents \textquotedblleft nothing\textquotedblright\ or \textquotedblleft nullity\textquotedblright, a philosophical term that denotes the general state of void or nonexistence. The empty set $\varnothing$ is not \textquotedblleft nothing\textquotedblright\ because it is a set with nothing inside it. This can be understood by viewing a set as a bag --- an empty bag is still a bag. Unpacking the empty set, nonetheless, removes the empty bag, and thus there is nothing left, or only nullity exists. Since $\varnothing=\{\ast\varnothing\}$ and $\varnothing\subset S,\,\ast\varnothing$ is a member of every set. So we have the following axiom.

\begin{axiom}
\label{DefNothing} \ $\ast\varnothing$ is known as \textbf{nullity}. Suppose $S$ is any set. Then
\end{axiom}

\begin{enumerate}
\item $\varnothing\:\Longleftrightarrow\,\left(\forall X\in\varnothing\right) \left(X=\ast\varnothing\right)$

\item $\forall S\left(\ast\varnothing\in S\right)$

\item $\forall X(X\in\ast\varnothing\,\,\Longrightarrow\, X=\ast\varnothing)$

\item $\forall S\left(\left(S,\ast\varnothing\right) \,=\,\left(\ast\varnothing,S\right) \,=\,\ast\varnothing\right)$
\end{enumerate}

\begin{remark}
\ The definition of $\varnothing$ is changed from containing nothing to containing nullity as its only member. In general, if nothing satisfies a sentence in the axiom of comprehension, then the solution is nullity.
\end{remark}

\begin{remark}
\ $\ast\varnothing$ is a special object that permeates in every set but is not involved in the general set operations. Since every set contains $\ast\varnothing$, most conclusions in set theory remain unchanged.
\end{remark}

\begin{lemma}
\label{41}\qquad$\{S,\ast\varnothing\}\,=\,\{S\}$
\end{lemma}

\begin{proof}
\ \ By axiom \ref{DefNothing},  for any $X$

\setstretch{1.4}
\hfill%
\begin{tabular}
[t]{l}%
$X\in\{S\}\,\,\Longleftrightarrow\,X=S\,\vee X=\ast\varnothing\,\,\Longleftrightarrow\, X\in\{S,\ast\varnothing\}$
\end{tabular}
\hfill%
\begin{tabular}
[t]{r}
\end{tabular}
\setstretch{1}
\end{proof}\bigskip

Lemma \ref{41} means that $\ast\varnothing$ can be omitted in any set.

\begin{corollary}
\qquad$S\times\varnothing\,=\,\varnothing\times S\,=\,\varnothing$
\end{corollary}

\begin{proof}
\ \ For any $\left(x,y\right)\in S\times\varnothing$, by axiom \ref{DefNothing}(iv), $\left(x,y\right)\,=\,\left(x, \ast\varnothing\right) \,=\,\ast\varnothing$. So $S\times\varnothing\,=\,\varnothing$. The second part is proved similarly.\bigskip
\end{proof}

Here is another example that uses the unpacking operator. 

\begin{example}
\ Suppose $\mathfrak{M} = (M,\dots)$ is a $\mathscr{L}$-structure and $A\subset M$. $X\subset M^n$ is $A$-definable if there is a $\mathscr{L}$-formula $\phi(x_1,\dots,x_n, y_1,\dots,y_m)$ and parameters $a_1,\dots,a_m\in A$ such that $X = \{(x_1,\dots,x_n)\colon \phi(x_1,\dots,x_n, a_1,\dots,a_m)\}$. If no parameters are needed, $X$ is called $\varnothing$-definable and $X = \{(x_1,\dots,x_n)\colon \phi(x_1,\dots,x_n)\}$. Through the unpacking operator, the two cases can be merged into one, i.e. $X = \{(x_1,\dots,x_n)\colon \phi(x_1,\dots,x_n, *D)\}$ where $D=\{a_1,\dots,a_m\}\subset A$.
\end{example}

\subsection{Membership Dimension}

\begin{definition}
\ The \textbf{membership dimension }of $S$ in $V_{}$ is the measure defined by a recursive function from $S$ to a cardinal number, i.e. $D\colon S\rightarrow\mathcal{K}$ and
\begin{equation}
D(S)\,=\,\sup\{D(X)\colon X\in S\}\,+\,1\label{DefMembershipDimension}
\end{equation}
Where $D(\varnothing)=0$.
\end{definition}

\begin{remark}
\ Membership dimension is based on cardinal numbers rather than ordinal ones and  measures the maximum number of curly brackets in a set (over $\varnothing$).
\end{remark}

\begin{example}
\ Suppose $\mathbf{n}$ is a finite \textit{von Neumann ordinal.}
\end{example}

\begin{enumerate}
\item $D(\mathbf{2})\,=\,D(\{\varnothing,\{\varnothing\}\})=2$ \ ($D(\{\varnothing\})=1\,$ and $D(\varnothing)=0$).

\item $D(\mathbf{n})\,=\,n.$

\item $D(\omega)\,=\,\aleph_{0}$ \ ($D({\omega})>n$ for any $n\in\omega$).
\end{enumerate}

The notion of membership dimension gives a necessary condition on when a set can be a member of itself.

\begin{theorem}
\label{47} \ Suppose $S_{n}$ are sets of finite membership dimension.
\end{theorem}

\begin{enumerate}
\item $S_{1}\in S_{2}\,\,\Longrightarrow\,D(S_{1})<D(S_{2})$

\item $S_{1}\in S_{2}\,\wedge\,\cdots\,\wedge\,S_{n-1}\in S_{n}\,\,\Longrightarrow\,D(S_{1})<D(S_{n})$

\item $\urcorner\left(S_{1}\in S_{2}\,\wedge\,\cdots\,\wedge\,S_{n-1}\in S_{n}\,\wedge\,S_{n}\in S_{1}\right)$
\end{enumerate}

\begin{proof}
\ \ (i) \ By (\ref{DefMembershipDimension}), $D(S_{2})\geqslant D(S_{1})+1>D(S_{1})$.\medskip

(ii) By (i) and induction.\medskip

(iii) \ If it is true, then by (i) and (ii), $D(S_{n})<D(S_{1})$ and $D(S_{1})<D(S_{n}),$ which is contradiction. $n=1$ reduces to the case that there is no $S_{1}$ that $S_{1}\in S_{1}$.\smallskip
\end{proof}

\begin{corollary}
\ That a set is a member of itself or contains a vicious cycle happens only if it has the infinite membership dimension.
\end{corollary}

\begin{proof}
\ By theorem \ref{47}(iii). $n=1$ reduces to the case that $S_{1}\in S_{1}$ means $D(S_{1})=\aleph_{0}.$\medskip

Note that the converse is not true. For example, $D(\omega)=\aleph_{0}$, but $\omega\notin\omega$ for $\omega$ is WF.\bigskip
\end{proof}

Membership dimension allows non-well-founded sets like infinitons to be defined intuitively as follows.\footnote{In later sections, we will give a rigorous treatment of these sets based on the notion of limit for finite structures.} An infiniton is a set that contains itself as the only member, i.e.
\begin{equation}
I=\underset{\aleph_{0}}{\,\underbrace{\{\dots\{\varnothing\}\dots\}}\,}=\underset{\aleph_{0}+1}{\,\underbrace{\{\{\dots\{\varnothing\}\dots\}\}}\,}=\{\underset{\aleph_{0}}{\,\underbrace{\{\dots\{\varnothing\}\dots\}}\,}\}  =\,\{I\}\label{Infiniteton}
\end{equation}

Generally, a set that is a member of itself is known as a semi-infiniton that takes on the following form. Suppose $G\,=\,\{a_{1},a_{2},\dots\}$ and $a_{k}\in V.$ Then
\[
Z\,=\,\{a_{1},a_{2},\dots,Z\}\,=\,\{\ast G,Z\}\footnote{$\ast G$ is the unpacking operator as in definition \ref{DefUnpackOperator}.}
\]

By replacing $Z$ with itself infinite times, we have
\begin{equation}
Z\,=\,\underset{\aleph_{0}}{\underbrace{\{\ast G,\{\ast G,\dots\{\ast G,\varnothing\}\dots\}}}\label{SemiInfiniton0}
\end{equation}

Then $Z$ is the solution to $S=\{\ast G,S\}$ for
\[
Z\,=\,\underset{\aleph_{0}+1}{\underbrace{\{\ast G,\{\ast G,\dots\{\ast G,\varnothing\}\dots\}}}\,=\,\{\ast G,\underset{\aleph_{0}}{\underbrace{\{\ast G,\dots\{\ast G,\varnothing\}\dots\}}}\}\,=\,\{\ast G,Z\}
\]

An infiniton is a special case of a semi-infiniton because $Z=\{\ast G,Z\}$ reduces to $Z=\{Z\}$ if $G=\varnothing$ (lemma \ref{41}).\smallskip

A set that contains a vicious cycle is called a quasi-infiniton and is illustrated as follows. Suppose
\[
S_{1}=\{\ast G_{1},S\},\,S_{2}=\{\ast G_{2},S_{1}\}, \,\dots, \,S_{n-1}=\{\ast G_{n-1},S_{n-2}\},\,S=\{\ast G_{n},S_{n-1}\}
\]

Then $S\in S_{1},\,\dots,\,S_{n-2}\in S_{n-1},\,S_{n-1}\in S$ form a vicious cycle. Let
\begin{equation}
Q\,=\,\underset{\aleph_{0}}{\underbrace{\{\overset{n}{\overbrace{\ast G_{n},\{\ast G_{n-1},\dots,\{\ast G_{1}}},\dots\{\overset{n}{\overbrace{\ast G_{n},\{\ast G_{n-1},\dots,\{\ast G_{1}}},\varnothing\}\dots\}}} \label{QuasiInfiniton0}
\end{equation}

Then $Q$ is the solution to $\,S\,=\,\{\ast G_{n},\{\ast G_{n-1},\dots,\{\ast G_{1},S\}\dots\}$ for
\begin{align*}
Q \,& =\,\underset{\aleph_{0}+n}{\underbrace{\{\overset{n}{\overbrace{\ast G_{n},\{\ast G_{n-1},\dots,\{\ast G_{1}}},\dots\{\overset{n}{\overbrace{\ast G_{n},\{\ast G_{n-1},\dots,\{\ast G_{1}}},\varnothing\}\dots\}}}
\\
& =\,\{\overset{n}{\overbrace{\ast G_{n},\{\ast G_{n-1},\dots,\{\ast G_{1}}},\underset{\aleph_{0}} {\underbrace {\{\overset{n}{\overbrace{\ast G_{n},\{\ast G_{n-1},\dots,\{\ast G_{1}}},\dots\{\overset{n}{\overbrace{\ast G_{n}, \{\ast G_{n-1},\dots,\{\ast G_{1}}},\varnothing\}\dots\}}}\dots\}
\\
& =\,\{{{\ast G_{n},\{\ast G_{n-1},\dots,\{\ast G_{1}}},Q\}\dots\}
\end{align*}

Obviously, generators of a quasi-infiniton form a finite cycle, and a quasi-infiniton (\ref{QuasiInfiniton0}) is reduced to a semi-infiniton (\ref{SemiInfiniton0}) if all\ $G_{k}$ $\left(1\leqslant k\leqslant n\right)$ are identical.\smallskip

\section{\label{SectionOfLimitFormula}Limit of Structures and Formulas}

In this section, we will investigate limits of (finitely generated) structures and formulas which provide a rigorous mathematical foundation for non-well-founded sets. The limit of finite structures is an infinite structure that can be described by an infinitely long formula of $\mathscr{L}_{\omega_{1},\omega_{1}}$ involving countably many conjunctions, disjunctions and quantifiers. An infinitely long formula involving countably many quantifiers may be undecidable \cite{Karp} but is always decidable in $\mathscr{L}_{\omega_{1},\omega}$ with only a finite number of quantifiers \cite{Scott}.

\subsection{Limit in General Language}

We begin with a brief review of some background knowledge in model theory. A (finitary) \textbf{formula} is a finite well-formed sequence of symbols from a given alphabet that is part of a formal language. A \textbf{sentence} is a formula that contains no free variables. A \textbf{theory} is a set of sentences in a first-order language $\mathscr{L}$ that is closed under logical implication. A \textbf{model} of a theory $T$ is a \textbf{structure} (a set along with relations, functions and constants) that satisfies the sentences of $T$. A \textbf{consistent} \textbf{theory} $T$ is a theory in which there is no sentence $\varphi$ that $T\vdash\varphi$ and $T\vdash\urcorner\varphi$. A \textbf{complete} \textbf{theory} $T$ is a theory in which for any sentence $\varphi,$ either $\varphi\in T$ or $\urcorner\varphi\in T$.\medskip

Furthermore, let $\mathfrak{M}=(M,\dots)$ be a $\mathscr{L}$-structure in a consistent theory $T$, $\bar{x}=(x_{1}, \dots,x_{n})$ be a $n$-tuple of variables, $\bar{a}=(a_{1},\dots,a_{n})$ ($a_{1},\dots,a_{n}\in M$), and $X\subset M$. The \textbf{complete type of $\bar{a}$ over $X$} with respect to $\mathfrak{M}$ is a maximal consistent set of formulas $\phi(\bar{x}, \bar{a})$ of $\mathscr{L}$ satisfied in $X^n$ (for all $n$), denoted as $\operatorname{tp}_M(\bar{a}/X)$ where $\bar{a}$ and the elements of $X$ are the parameters of the complete type. The \textbf{complete type over $X$} with respect to $\mathfrak{M}$ is a maximal consistent set of formulas $\phi(\bar{x})$ of $\mathscr{L}$ satisfied in $X^n$ (for all $n$). A $\mathbf{n}$\textbf{-type} is a consistent set of formulas with $n$ free variables and a subset of a complete type. A \textbf{type} (or \textbf{partial type}) is a subset of a complete type and can be either a complete type or a $n$-type. A type $\Gamma$ is called an \textbf{isolated type} in $T$ if, for any $\gamma\in\Gamma$, there is a complete formula $\varphi$ that $T\vDash\varphi\rightarrow\gamma$ (\cite{Chang-Keisler} and \cite{Hodges}). \medskip

The theory of $\mathfrak{M}$ (denoted as Th$(\mathfrak{M})$) is the set of sentences satisfied by $\mathfrak{M}$. In a complete theory $T$, a formula $\varphi$ is called \textbf{complete} in $T$ if for every formula $\phi,\,T\vDash\varphi\rightarrow \phi$ or $T\vDash\varphi\rightarrow\urcorner\phi$. An \textbf{atomic theory} $T$ is a theory in which every formula $\gamma$ that is consistent with $T$ can be derived from a complete formula $\varphi$ in $T$, i.e. $T\vDash\varphi\rightarrow\gamma$. A $\mathscr{L}$-structure $\mathfrak{A}$ is an \textbf{atomic structure} if every $n$-tuple $\bar{a}$ in $\mathfrak{A}$ satisfies a complete formula in Th$(\mathfrak{A})$. Obviously, every type in an atomic theory is isolated. A $\mathscr{L}$-structure $\mathfrak{M}=(M,\dots)$ is called $\mathbf{\kappa}$-\textbf{saturated} if for all subsets $A\subset M$ of cardinality less than $\kappa,\, \mathfrak{M}$ realizes all complete types over $A$. $\mathfrak{M}$ is called \textbf{countably saturated} if it is $\aleph_{0}$-\textbf{saturated}. \medskip

Suppose $\mathfrak{N}=(N,\dots)$ is another $\mathscr{L}$-structure in $T$. $\mathfrak{M}$ and $\mathfrak{N}$ are \textbf{isomorphic} $(\mathfrak{M}\cong\mathfrak{N})$ iff there is a $1$-$1$ function $f$ mapping $M$ onto $N$ and satisfying the following properties: (i) For each relation symbol $R$ of $\mathscr{L}$, $R^{\mathfrak{M}}(\bar{a})$ iff $R^{\mathfrak{N}}(f(\bar{a}))$; (ii) For each function $G$ of $\mathscr{L}$, $f(G^{\mathfrak{M}}(\bar{a}))=G^{\mathfrak{N}}(f(\bar{a}))$; (iii) For each constant $c$ of $\mathscr{L}$, $f(c^{\mathfrak{M}})=c^{\mathfrak{N}}$. $\mathfrak{M}$ and $\mathfrak{N}$ are \textbf{elementarily equivalent} $(\mathfrak{M}\equiv\mathfrak{N})$ iff for each $\mathscr{L}$-sentence $\phi$, $\mathfrak{M}\vDash \phi\Leftrightarrow\mathfrak{N}\vDash \phi$. $\mathfrak{M}$ is \textbf{$\omega$-homogeneous} if for any two tuples $\bar{a}=(a_{1},\dots,a_{n})$ and $\bar{b}=(b_{1},\dots,b_{n})$ in $\mathfrak{M}$ that $(\mathfrak{M},\bar{a})\equiv (\mathfrak{M},\bar{b})$, and any $c\in M$ there exists $d\in M$ such that $(\mathfrak{M},\bar{a},c)\equiv (\mathfrak{M},\bar{b},d)$. An atomic or saturated structure is always $\omega$-homogeneous. There is another homogeneity known as \textbf{ultrahomogeneous} in which any isomorphism between its two finite substructures in a $\mathscr{L}$-structure $\mathfrak{M}$ can be extended to an automorphism of $\mathfrak{M}$. An ultrahomogeneous structure is always $\omega$-homogeneous but not vice versa. Hence any conclusion holds for an ultrahomogeneous structure also holds for a $\omega$-homogeneous structure  (\cite{Chang-Keisler} and \cite{Hodges}). \medskip

A $\aleph_{0}$-\textbf{categorical theory }in $\mathscr{L}$ has exactly one countable structure up to isomorphism. A $\aleph_{0}$-\textbf{categorical structure} is a countable structure whose theory is $\aleph_{0}$-categorical. We list the following theorems without proof. Proposition \ref{25} states that a $\aleph_{0}$-{categorical structure} is made up of only finitely many countable atomic structures. It is central to $\aleph_{0}$-categorical theories and is essentially due to Engeler, Ryll-Nardzewski, Svenonius and Vaught.


\begin{proposition}
	\label{58} \ Any atomic structure is $\omega$-homogeneous.
\end{proposition}

\begin{proposition}
	\label{70} \ Two countable homogeneous structures that realize the same types are isomorphic.
\end{proposition}

\begin{proposition}
\label{25}\cite[Theorem 2.3.13]{Chang-Keisler} \ Let $T$ be a complete theory. Then the following are equivalent:
\end{proposition}

\begin{enumerate}
\item $T$ \textit{is }$\aleph_{0}$\textit{-categorical.}


\item \textit{For each $n<\omega$, each type $\Gamma(x_{1},\dots,x_{n})$ of $T$ has a complete formula.}

\item \textit{For each $n<\omega,\:T$ has only finitely many types in $x_{1},\dots,x_{n}$.}

\item \textit{For each $n<\omega$, there are only finitely many formulas $\varphi(x_{1},\dots,x_{n})$ up to equivalence with respect to $T$.}

\item \textit{All structures of $T$ are countably atomic and saturated.}
\end{enumerate}

In the following discussion of this paper, without further specification, $\mathscr{L}$ means a first-order language, $\bar{\mathscr{L}}$ the language of set theory, and $\mathscr{L}_{\omega_{1},\omega}$ an infinitary language of $\mathscr{L}$. $\mathscr{L}^{\prime}$ indicates an expansion of $\mathscr{L}$ with new relations, functions or constants. $T$ is a $\aleph_{0}$-categorical theory. A (small) Greek letter such as $\phi$ or $\varphi$ represents a formula in a language, and a Greek letter with index like $\phi_n$ or $\varphi_n$ represents a sequence of formulas. A (capital) letter in fraktur font like $\mathfrak{M}$ means a structure (model) and a fraktur letter with index like $\mathfrak{M}_n$ indicates a sequence of structures. A theorem involving \textit{homogeneous} means that it holds for both \textit{ultrahomogeneous} and \textit{$\omega$-homogeneous}.\medskip 

Now we discuss the limit of structures and formulas for a $\aleph_{0}$-categorical theory. First, we define the neighborhood of $\omega$ based on the cofinite topology.

\begin{definition}
	\label{DefNeighborhood} \ The\textbf{\ cofinite topology} on $\omega$ is defined as $\,\mathfrak{T}\,=\,\{y\subset \omega\colon y=\varnothing\,\vee\,\omega-y$ is finite$\,\}$. A \textbf{neighborhood of} $\mathbf{\omega}\, ( \omega$-neighborhood$)\,\mathfrak{H}$ is a member of $\mathfrak{T}$, i.e. $\mathfrak{H}\in\mathfrak{T}.$
\end{definition}

\begin{lemma}
	\label{7} \ $\mathfrak{H}$ is a neighborhood of $\omega$ if and only if $\exists N\in\omega$ such that $\forall n>N,n\in\mathfrak{H}.$
\end{lemma}

\begin{proof}
	\ Suppose $\mathfrak{H}$ is a neighborhood of $\omega$ and for any $N\in\omega,$ there is $n>N,\,n\notin\mathfrak{H}.$ Then $\omega-\mathfrak{H}$ is not finite, contradicting definition \ref{DefNeighborhood}. On the other hand, if there is a $N\in\omega$ such that for any $n>N,n\in\mathfrak{H},$ then $\omega-\mathfrak{H}$ is finite and $\mathfrak{H}\in\mathfrak{T}$.\medskip
\end{proof}

\begin{definition}
	\label{DefHomoSeq} \ Suppose $\mathscr{L}$ is a first-order language and $T$ is a $\aleph_{0}$-categorical theory of $\mathscr{L}$. Let $\phi_{n}$ be types in $T$ and $\,\mathfrak{M}_{n}$ be  $\mathscr{L}$-structures that $\,\mathfrak{M}_{n}\vDash\phi_{n}$. If there exists a $\omega$-neighborhood $\mathfrak{H}$ that for any $k,n\in\mathfrak{H}\:(k>n)$, $\mathfrak{M}_{k}\vDash\phi_{n}$, then $\{(\mathfrak{M}_{n}, \phi_n)\colon\mathfrak{M}_{n}\vDash\phi_{n}\land n<\omega\}$ is known as a \textbf{homogeneous sequence of structures defined by $\phi_n$} in $T$.
\end{definition}

\begin{theorem}
	\label{71}\ Suppose $T$ is a $\aleph_{0}$-categorical theory of $\mathscr{L}$ and $\{(\mathfrak{M}_{n},\phi_n)\colon \mathfrak{M}_{n} \vDash\phi_{n}\land n<\omega\}$ a homogeneous sequence of structures in $T$. Then there is a unique formula $\phi$ in $\mathscr{L}_{\omega_{1},\omega}$ (up to equivalence) and a unique countable atomic structure $\mathfrak{M}$ (up to isomorphism) for $\{(\mathfrak{M}_{n},\phi_n)\}$ such that $\mathfrak{M}\vDash\phi$. Also, there is a $N<\omega$ such that $\,\phi\Leftrightarrow\:\smashoperator{\bigwedge_{N<n<\omega}}\phi_{n}$ which is a complete formula of $\mathfrak{M}$.
\end{theorem} 

\begin{proof}
	\ By lemma \ref{7}, $\exists N<\omega$ such that $\forall k>n>N,\,\mathfrak{M}_{k}\vDash\phi_{n}$. Let $\Sigma=\{\phi_n\colon n>N\}$. Since any finite subset of $\Sigma$ has a model, by the compactness theorem, $\Sigma$ has a model. Since arbitrary large number of $\phi_{n}$ can be realized by $\mathfrak{M}_{k}$, $\Sigma$ has an infinite model. So $\Sigma$ is consistent and a subset of a complete type of $T$. Let $\mathfrak{M}$ be a countable structure that satisfies $\Sigma$. Then $\mathfrak{M}\vDash\:\smashoperator{\bigwedge_{N<n<\omega}}\phi_{n}\Leftrightarrow\phi$. By proposition \ref{25}, $\mathfrak{M}$ is a countable atomic structure. Since any $\phi_{n}$ can be derived from $\phi$, $\phi$ is a complete formula of $\mathfrak{M}$.  \smallskip
	
	Suppose $\mathfrak{N}$ is another countable atomic structure satisfying $\Sigma$. By proposition \ref{58}, both $\mathfrak{M}$ and $\mathfrak{N}$ are $\omega$-homogeneous. And by proposition \ref{70}, $\mathfrak{M}\cong\mathfrak{N}$. So $\mathfrak{M}$ is unique. 
\end{proof}

\begin{corollary}
	\label{16}\ Suppose in a sequence of structures $\{(\mathfrak{M}_{n},\phi_n)\colon\ \mathfrak{M}_{n} \vDash\phi_{n}\land n<\omega\}$ of a $\aleph_{0}$-categorical theory $T$, there are finitely many homogeneous subsequences of structures $\{(\mathfrak{M}_{n_i}, \phi_{n_i}) \colon \allowbreak\mathfrak{M}_{n_i} \vDash\phi_{n_i}\land n_i<\omega\}$. Then there is a unique formula $\widetilde{\phi_i}$ in $\mathscr{L}_{\omega_{1},\omega}$ (up to equivalence) and a unique countable atomic structure $\widetilde{\mathfrak{M}_i}$ (up to isomorphism) for each $\{(\mathfrak{M}_{n_i},\phi_{n_i})\}$ such that $\widetilde{\mathfrak{M}_i}\vDash\widetilde{\phi_i}$. In addition, there is a $N_i<\omega$ such that $\,\widetilde{\phi_i}\Leftrightarrow\:\smashoperator{\bigwedge_{N_i<n_i<\omega}} \phi_{n_i}$ which is a complete formula of $\widetilde{\mathfrak{M}_i}$.
\end{corollary} 

\begin{proof}
	Let $\Sigma_i=\{\phi_{n_i}\colon n_i>N_i\}$. By the proof in theorem \ref{71}, $\Sigma_i$ is consistent and belongs to a type of $T$. So by proposition \ref{25}, it is satisfied by a unique countable atomic structure $\widetilde{\mathfrak{M}_i}$ defined by a complete formula $\widetilde{\phi_i}$. \bigskip
\end{proof}

From definition \ref{DefHomoSeq}, theorem \ref{71} and corollary \ref{16}, we have the following definitions.

\begin{definition}
\label{DefLimit} \ Suppose $\{(\mathfrak{M}_{n},\phi_n)\colon \mathfrak{M}_{n} \vDash\phi_{n}\land n<\omega\}$ is a homogeneous sequence of structures in a $\aleph_{0}$-categorical theory. The unique countable atomic structure $\mathfrak{M}$ (up to isomorphism) in $\{(\mathfrak{M}_{n},\phi_n)\}$ is known as the \textbf{limit} of $\mathfrak{M}_n$ and is denoted as $\underset{n\rightarrow\omega}{\lim}\mathfrak{M}_{n}=\mathfrak{M}$. The unique formula $\phi$ (up to equivalence) in $\mathscr{L}_{\omega_{1},\omega}$ is known as the \textbf{limit} of $\phi_{n}$ and is denoted as $\underset{n\rightarrow\omega}{\lim}\phi_{n}=\phi$. In both cases, we also say that the limit of $\phi_{n}$ or the limit of $\mathfrak{M}_{n}$ is unique.
\end{definition}


\begin{definition}
	\label{DefSubLimit} \ Suppose in a sequence of structures $\{(\mathfrak{M}_{n},\phi_n)\colon \mathfrak{M}_{n} \vDash\phi_{n}\land n<\omega\}$ in a $\aleph_{0}$-categorical theory, there are finitely many homogeneous subsequences of structures $\{(\mathfrak{M}_{n_i}, \phi_{n_i})\colon \allowbreak \mathfrak{M}_{n_i} \vDash\phi_{n_i}\land n_i<\omega\}$. Then each $\underset{i\rightarrow\omega}{\lim}\,\mathfrak{M}_{n_i}$ is known as a \textbf{sublimit} of $\mathfrak{M}_{n}$, and  each $\underset{i\rightarrow\omega}{\lim}\,\phi_{n_i}$ is known as a \textbf{sublimit} of $\phi_{n}$. If some sublimits of $\mathfrak{M}_{n}/\phi_{n}$ are different, we say $\underset{n\rightarrow\omega}{\lim}\mathfrak{M}_{n}/ \underset{n\rightarrow\omega}{\lim}\phi_{n}$ exist (but not unique).
\end{definition}

In the rest discussion, we will not distinguish \textquotedblleft$=$\textquotedblright \ and \textquotedblleft$\Leftrightarrow$\textquotedblright \ for formulas. So we have

\begin{lemma}
$\left(\underset{n\rightarrow\omega}{\lim}\phi_{n}\,=\,\underset{n\rightarrow\omega}{\lim}\,\varphi_{n}\right)  \,\Longleftrightarrow\,\left(\underset{n\rightarrow\omega}{\lim}\phi_{n}\,\Longleftrightarrow\,\underset{n\rightarrow\omega}{\lim}\,\varphi_{n}\right)$
\end{lemma}

\begin{corollary}
\label{18}
\end{corollary}

\begin{enumerate}
	\item $(\exists N<\omega)(\forall n>N)(\phi_n=\phi)\,\Longrightarrow\,\underset{n\rightarrow\omega}{\lim}\phi_n=\phi$

	\item $(\exists N<\omega)(\forall n>N)(\mathfrak{M}_n=\mathfrak{M})\,\Longrightarrow\,\underset{n\rightarrow\omega}{\lim}\mathfrak{M}_n=\mathfrak{M}$
\end{enumerate}

\begin{corollary}
	\label{23} \ Suppose $\underset {n\rightarrow\omega}{\lim}\phi_{n}$ and $\underset{n\rightarrow\omega}{\lim}\varphi_{n}$ are unique. Then
\end{corollary}

\begin{enumerate}	
	\item $\underset{n\rightarrow\omega}{\lim}\phi_{n-p}\,=\,\underset{n\rightarrow\omega}{\lim}\phi_{n}\,\, \left(p<\omega\right)$
	
	\item $\left(\forall n\in\mathfrak{H}\right)\left(\phi_{n}\,=\,\varphi_{n}\right)\,\Longrightarrow\,\left(\underset {n\rightarrow\omega}{\lim}\phi_{n}\,=\,\underset{n\rightarrow\omega}{\lim}\,\varphi_{n}\right)$
\end{enumerate}

\begin{proof}
	\ By theorem \ref{71} and definition \ref{DefLimit}.
\end{proof}

\begin{corollary}
	\label{42} \ Suppose $\underset {n\rightarrow\omega}{\lim}\mathfrak{M}_{n}$ and $\underset{n\rightarrow\omega}{\lim}\mathfrak{N}_{n}$ are unique. Then
\end{corollary}

\begin{enumerate}
	\item $\underset{n\rightarrow\omega}{\lim}\,\mathfrak{M}_{n-p}\,=\,\underset{n\rightarrow\omega}{\lim}\,\mathfrak{M}_{n}\,\, \left(p<\omega\right)$
	
	\item $\left(\forall n\in\mathfrak{H}\right)\left(\mathfrak{M}_{n}\,=\,\mathfrak{N}_{n}\right) \,\Longrightarrow\,\left(  \underset{n\rightarrow\omega}{\lim}\,\mathfrak{M}_{n}\,=\,\underset{n\rightarrow\omega}{\lim}\mathfrak{N}_{n}\right)$
\end{enumerate}

\begin{corollary}
	\label{17} \ Suppose $\mathfrak{M}_{n}\vDash\phi_{n}$ and for any $n\in\omega,\,\vDash\phi_{n+1} \rightarrow\phi_{n}$. Then 
	\[
	\underset{n\rightarrow\omega}{\lim}\phi_{n}\, =\,\underset{n\rightarrow\omega}{\lim}\,\underset{m\leqslant n} {\displaystyle\bigwedge}\phi_{m}\,=\underset{n<\omega} {\displaystyle\bigwedge}\phi_{n}
	\]
\end{corollary}

\begin{proof}
	\ Let $\mathfrak{H} = \omega$. By lemma \ref{7}, for any $1\leqslant n<k<\omega,\,\vDash\phi_{k} \rightarrow\phi_{n}$. Since $\mathfrak{M}_{k}\vDash\phi_{n}$ for any $k>n\geqslant 1$, $\{(\mathfrak{M}_{n},\phi_n)\}$ is a homogeneous sequence. Thus by theorem \ref{71},  $\underset{n\rightarrow\omega}{\lim}\phi_{n}=\underset{n<\omega} {\displaystyle\bigwedge}\phi_{n}$. Since for any $1\leqslant m<n,\:\vDash\phi_{n} \rightarrow\phi_{m}$, $\underset{m\leqslant n}{\displaystyle\bigwedge}\phi_{m}\Leftrightarrow \phi_n$. So by theorem \ref{71} again, $\underset{n\rightarrow\omega}{\lim}\,\underset{m\leqslant n}{\displaystyle\bigwedge}\phi_{m}\, =\underset{n<\omega} {\displaystyle\bigwedge}\phi_{n}.$\bigskip
\end{proof}

In addition, the following axiom holds for the limit operations.

\begin{axiom}
\label{15} \ Suppose $\phi_{n}$ and $\varphi_{n}$ are consistent, $\underset{n\rightarrow\omega}{\lim}\phi_{n}$ and $\underset{n\rightarrow\omega}{\lim}\,\varphi_{n}$ are unique in a $\aleph_{0}$-categorical theory. Then
\end{axiom}

\begin{enumerate}
\item \textit{$\underset{n\rightarrow\omega}{\lim}\left(\phi_{n}\wedge\varphi_{n}\right)$ is unique and}
\[
\underset{n\rightarrow\omega}{\lim}\left(\phi_{n}\wedge\varphi_{n}\right)\,=\,\underset{n\rightarrow\omega}{\lim}\phi_{n} \wedge\underset{n\rightarrow\omega}{\lim}\,\varphi_{n}
\]

\item \textit{$\underset{n\rightarrow\omega}{\lim} \urcorner\phi_{n}$ is unique and}
\[
\underset{n\rightarrow\omega}{\lim}\urcorner\phi_{n}\,=\,\urcorner\,\underset{n\rightarrow\omega}{\lim}\phi_{n}
\]

\item \textit{$\underset{n\rightarrow\omega}{\lim}\exists x \,\phi_{n}$ is unique $(x$ is a variable in $\phi_{n})$, and}
\[
\underset{n\rightarrow\omega}{\lim}\exists x\,\phi_{n}\,=\,\exists x\,\underset{n\rightarrow\omega}{\lim}\phi_{n}
\]
\end{enumerate}

\begin{corollary}
\label{12} \ Suppose $\phi_{n}$ and $\varphi_{n}$ are consistent, $\underset{n\rightarrow\omega}{\lim}\phi_{n}$ and $\underset{n\rightarrow\omega}{\lim}\,\varphi_{n}$ are unique in a $\aleph_{0}$-categorical theory. Then
\end{corollary}

\begin{enumerate}
\item $\underset{n\rightarrow\omega}{\lim}\left(\phi_{n}\vee\varphi_{n}\right) \,=\,\underset{n\rightarrow\omega}{\lim} \phi_{n}\vee\underset{n\rightarrow\omega}{\lim}\,\varphi_{n}$

\item $\underset{n\rightarrow\omega}{\lim}\left(\phi_{n}\Longrightarrow\varphi_{n}\right)\,=\,\left(\underset {n\rightarrow\omega}{\lim}\phi_{n}\,\Longrightarrow\,\underset{n\rightarrow\omega}{\lim}\,\varphi_{n}\right)$

\item $\underset{n\rightarrow\omega}{\lim}\left(\phi_{n}\Longleftrightarrow\varphi_{n}\right)\,=\,\left(\underset {n\rightarrow\omega}{\lim}\phi_{n}\,\Longleftrightarrow\,\underset{n\rightarrow\omega}{\lim}\,\varphi_{n}\right)$

\item $\underset{n\rightarrow\omega}{\lim}\,\forall x\,\phi_{n}\,=\,\forall x\,\underset{n\rightarrow\omega}{\lim}\phi_{n}$
\end{enumerate}

\begin{proof}
\ (i) \ By axiom \ref{15}
\begin{align*}
\underset{n\rightarrow\omega}{\lim}\left(\phi_{n}\vee\varphi_{n}\right) \,\, & =\,\,\underset{n\rightarrow\omega}{\lim} \urcorner\left(\urcorner\phi_{n}\,\wedge\,\urcorner\varphi_{n}\right) 
\\
& =\,\,\urcorner\left(\urcorner\,\underset{n\rightarrow\omega}{\lim}\phi_{n}\,\wedge\,\urcorner\, \underset{n\rightarrow\omega}{\lim}\,\varphi_{n}\right) 
\\
& =\,\,\underset{n\rightarrow\omega}{\lim}\phi_{n}\,\vee\,\underset{n\rightarrow\omega}{\lim}\,\varphi_{n}
\end{align*}

(ii) and (iii) follow from (i).\medskip

(iv) \ By axiom \ref{15}
\begin{align*}
\underset{n\rightarrow\omega}{\lim}\,\forall x\,\phi_{n} \,& =\,\,\urcorner\,\underset{n\rightarrow\omega}{\lim}\,\exists x\urcorner\phi_{n}
\\
& =\,\,\urcorner\,\exists x\urcorner\,\underset{n\rightarrow\omega}{\lim}\phi_{n}
\\
& =\,\,\forall x\,\underset{n\rightarrow\omega}{\lim}\phi_{n}
\end{align*}\vspace{-12pt}
\end{proof}

\begin{corollary}
\label{46}\ Suppose $\mathfrak{M}_{n}\vDash\phi_{n}$ and for any $n\in\omega,\,\vDash\phi_{n-1} \rightarrow\phi_{n}$. Then 
\[
	\underset{n\rightarrow\omega}{\lim}\,\underset{m\leqslant n}{\displaystyle\bigvee}\phi_{m}\,=\underset{n<\omega} {\displaystyle\bigvee}\phi_{n}
\]
\end{corollary}

\begin{proof}
Let $\mathfrak{H} = \omega$. Then for any $n\in\mathfrak{H},\,\vDash\,\urcorner\phi_{n} \rightarrow\,\urcorner\phi_{n-1}$. So by corollary \ref{17}
\[
\underset{n\rightarrow\omega}{\lim}\urcorner\phi_n\,=\,\underset{n\rightarrow\omega}{\lim}\,\underset{m\leqslant n}{\displaystyle\bigwedge}\urcorner\phi_{m}\, =\underset{n<\omega} {\displaystyle\bigwedge}\urcorner\phi_{n}
\]
Thus by axiom \ref{15}
\begin{align*}
\underset{n\rightarrow\omega}{\lim}\,\underset{m\leqslant n}{\displaystyle\bigvee}\phi_{m}\,& =\,\,\urcorner\,\underset{n\rightarrow\omega}{\lim}\,\underset{m\leqslant n} {\displaystyle\bigwedge}\urcorner\phi_{m}
\\
& =\,\,\urcorner\,\underset{n<\omega} {\displaystyle\bigwedge}\urcorner\phi_{n}
\\
& =\,\,\underset{n<\omega}{\displaystyle\bigvee}\phi_{n}
\end{align*}\vspace{-6pt}
\end{proof}

%

\begin{proposition}
\ Suppose $\phi_{n}\,=\,\exists x\underset{m\leqslant n}{\displaystyle\bigwedge}\left(x>m\right)$. Then $\underset{n\rightarrow\omega}{\lim}\phi_{n}\,=\,\exists x\underset{m<\omega}{\displaystyle\bigwedge}\left(x>m\right)$.
\end{proposition}

\begin{proof}
\ Let $\varphi_n=(x>n)$. Then $\forall n\in\omega,\:\vDash\varphi_{n+1}\rightarrow\varphi_{n}$. So it follows by corollary \ref{17} and axiom \ref{15}. This confirms that there is an arbitrary large number in nonstandard number theory.\bigskip
\end{proof}

Now we give the limit of formula for the $\epsilon-N$ formula.

\begin{corollary}
\ Suppose $a_{\omega}=\left\langle a_{m}\colon m<\omega\right\rangle $ is a sequence in a separable space and
\[
\varphi_{n}\,=\,\exists x\,\exists N_{n}\,\forall m\left(m>N_{n}\, \Longrightarrow|a_{m}-x|<1/n\right)
\]
Then the limit of formula for $\underset{m\rightarrow\omega}{\lim}a_{m}=x$ is:
\[
\underset{n\rightarrow\omega}{\lim}\,\underset{p\leqslant n}{\displaystyle\bigwedge}\varphi_{p}\,=\,\exists x\underset{n<\omega}{\displaystyle\bigwedge}\exists N_{n}\,\forall m\left(m>N_{n}\,\Longrightarrow|a_{m}-x|<1/n\right)
\]
\end{corollary}

\begin{proof}
\ For any $n\in\omega$ and any $m>N_{n}$, since $|a_{m}-x|<1/\left(n+1\right)\Longrightarrow|a_{m}-x|<1/n$, $\vDash\varphi_{n+1}\rightarrow\varphi_{n}$ for any $n<\omega$. So it follows by corollary \ref{17} and axiom \ref{15}.\bigskip
\end{proof}

Note that a $\aleph_{0}$-categorical theory (or atomic theory) is absolutely necessary in the above definitions of limit and axiom \ref{15} because of the following example.\footnote{This example is suggested by Martin Goldstern.} Suppose $I_{0}=G_{0}$ and $I_{n+1}=\{I_{n}\}$ where $G_{0}\neq\{G_{0}\}$. Let $\chi=\forall x\left(x\neq\{x\}\right)$ and $\varphi_{n}=\phi_{n} \wedge\chi$ where $\phi_{n}$ is given in theorem \ref{38}(iii). Since $I_{n}\vDash\phi_{n}$ and $\underset{n\rightarrow\omega}{\lim}I_{n}$ is unique, by theorem \ref{30}(i), $I_{\omega}=\{I_{\omega}\}$. So $\underset {n\rightarrow\omega}{\lim}\,\varphi_{n}$ is unique but it must also satisfy $\chi$, which is a contradiction. This can be avoided by the fact that Th$(I_{n})$ is $\aleph_{0}$-categorical (theorem \ref{38}(ii)). So by theorem \ref{30}(iv), $I_{\omega}$ is atomic and there is a complete formula $\varphi_{\omega}$ for $\operatorname{Th}(I_{\omega})$ such that $\chi\notin\operatorname{Th}(I_{\omega})$ for $\varphi_{\omega}\Rightarrow\urcorner\chi$.\smallskip

\subsection{Limit in Language of Set Theory}

Next, we will study the limit of structures in the language of set theory $\bar{\mathscr{L}}=\{\in\}$.

\begin{axiom}
\label{19} \ Suppose $\mathfrak{M}_{n}$ are $\bar{\mathscr{L}}$-structures in a $\aleph_{0}$-categorical theory and $\underset{n\rightarrow\omega}{\lim}\,\mathfrak{M}_{n}$ is unique.\footnote{In the rest of discussion of this section, we assume  $\underset{n\rightarrow\omega}{\lim}\,\mathfrak{M}_{n}$ and $\underset{n\rightarrow\omega}{\lim}\,\mathfrak{N}_{n}$ are unique unless further specified.} Then
\[
\underset{n\rightarrow\omega}{\lim}\left(\mathfrak{A}\in\mathfrak{M}_{n}\right)\,\Longleftrightarrow\,\bigl(\mathfrak{A}\in\underset {n\rightarrow\omega}{\lim}\,\mathfrak{M}_{n}\bigr)
\]
\end{axiom}


\begin{lemma}
\label{13}\qquad$\underset{n\rightarrow\omega}{\lim}\,\exists\mathfrak{A}\left(\mathfrak{M}_{n}=\mathfrak{A}\right) \,\Longleftrightarrow\,\exists\mathfrak{A}\, \bigl(\underset {n\rightarrow\omega}{\lim}\,\mathfrak{M}_{n}=\mathfrak{A}\bigr)$
\end{lemma}

\begin{proof}
\ By axioms \ref{15}, \ref{19} and corollary \ref{12}
\begin{align*}
\underset{n\rightarrow\omega}{\lim}\,\exists\mathfrak{A}\left(\mathfrak{M}_{n}=\mathfrak{A}\right) \,\,&\Longleftrightarrow\,\,\underset{n\rightarrow\omega} {\lim}\,\exists\mathfrak{A}\,\forall\mathfrak{B}\left(\mathfrak{B}\in\mathfrak{M}_{n}\Longleftrightarrow \mathfrak{B}\in\mathfrak{A}\right)
\\
\,\,&\Longleftrightarrow\,\,\exists\mathfrak{A}\,\forall\mathfrak{B}\,\bigl(\mathfrak{B}\in\underset{n\rightarrow\omega}{\lim}\,\mathfrak{M}_{n}\Longleftrightarrow \mathfrak{B}\in\mathfrak{A}\bigr)  
\\
\,\,&\Longleftrightarrow\,\,\exists\mathfrak{A}\,\bigl(\underset{n\rightarrow\omega}{\lim}\,\mathfrak{M}_{n}=\mathfrak{A}\bigr)
\end{align*}\vspace{-9pt}
\end{proof}

\begin{corollary}
\label{29}\ Suppose $\mathfrak G$, $\mathfrak G_1,\dots,\mathfrak G_l$ are structures without relation and function symbols.\footnote{The symbol $\exists!$ means there exists only one.}
\end{corollary}

\begin{enumerate}
\item $\underset{n\rightarrow\omega}{\lim}\left(\mathfrak G\cup\{\mathfrak{M}_{n}\}\right)\,=\, \mathfrak G\cup\{\underset{n\rightarrow\omega}{\lim}\,\mathfrak{M}_{n}\}$

\item $\underset{n\rightarrow\omega}{\lim}\,\{\ast \mathfrak G_{1},\{\ast \mathfrak G_{2},\dots\{\ast \mathfrak G_{l},\mathfrak{M}_{n}\}\dots\}\,=\,\{\ast \mathfrak G_{1},\{\ast \mathfrak G_{2},\dots\{\ast \mathfrak G_{l},\underset{n\rightarrow\omega}{\lim}\, \mathfrak{M}_{n}\}\dots\}$
\end{enumerate}

\begin{proof}
\ (i) \ By lemma \ref{13}
\begin{align*}
\forall\mathfrak{B}\,\bigl(\mathfrak{B}\in\underset{n\rightarrow\omega}{\lim}(\mathfrak G\cup\{\mathfrak{M}_{n}\})\bigr) \,\,&\Longleftrightarrow\,\,\underset{n\rightarrow\omega}{\lim}\,\forall\mathfrak{B}\left(\mathfrak{B}\in\mathfrak G\cup\{\mathfrak{M}_{n}\}\right)
\\
\,\,&\Longleftrightarrow\,\,\underset{n\rightarrow\omega}{\lim}\,\forall\mathfrak{B}\left(\mathfrak{B}\in\mathfrak G\,\vee\, \mathfrak{B}\in\{\mathfrak{M}_{n}\}\right)
\\
\,\,&\Longleftrightarrow\,\,\underset{n\rightarrow\omega}{\lim}\,\exists!\,\mathfrak{B}\left(\mathfrak{B}\,=\,\mathfrak{M}_{n}\right)
\\
\,\,&\Longleftrightarrow\,\,\exists!\,\mathfrak{B}\,\bigl(\mathfrak{B}\,=\underset{n\rightarrow\omega}{\lim}\,\mathfrak{M}_{n}\bigr)
\\
\,\,&\Longleftrightarrow\,\,\forall\mathfrak{B}\,\bigl(\mathfrak{B}\in\{\underset{n\rightarrow\omega}{\lim}\,\mathfrak{M}_{n}\}\bigr)
\\
\,\,&\Longleftrightarrow\,\,\forall\mathfrak{B}\,\bigl(\mathfrak{B}\in\mathfrak G\cup\{\underset{n\rightarrow\omega}{\lim}\,\mathfrak{M}_{n}\}\bigr)
\end{align*}

\ (ii) \ By (i) and induction.
\end{proof}

\begin{corollary}
\label{24}
\end{corollary}

\begin{enumerate}
\item $\underset{n\rightarrow\omega}{\lim}\,\left(\forall\mathfrak{B}\in\mathfrak{M}_{n}\right)\phi_{n}\,\Longleftrightarrow\,\bigl(\forall \mathfrak{B}\in\underset{n\rightarrow\omega}{\lim}\,\mathfrak{M}_{n}\bigr)\underset{n\rightarrow\omega}{\lim}\phi_{n}$

\item $\underset{n\rightarrow\omega}{\lim}\,\left(\exists\mathfrak{B}\in\mathfrak{M}_{n}\right)\phi_{n}\,\Longleftrightarrow\,\bigl(\exists \mathfrak{B}\in\underset{n\rightarrow\omega}{\lim}\,\mathfrak{M}_{n}\bigr)\underset{n\rightarrow\omega}{\lim}\phi_{n}$
\end{enumerate}

\begin{proof}
\ (i) \ By corollary \ref{12}
\begin{align*}
\underset{n\rightarrow\omega}{\lim}\left(\forall\mathfrak{B}\in\mathfrak{M}_{n}\right)\phi_{n} \,\,&\Longleftrightarrow\,\,\underset{n\rightarrow\omega}{\lim}\,\forall\mathfrak{B}\left(\mathfrak{B}\in\mathfrak{M}_{n}\,\Longrightarrow\, \phi_{n}\right)
\\
\,\,&\Longleftrightarrow\,\,\forall\mathfrak{B}\,\bigl(\mathfrak{B}\in\underset{n\rightarrow\omega}{\lim}\,\mathfrak{M}_{n}\,\Longrightarrow \,\underset{n\rightarrow\omega}{\lim}\phi_{n}\bigr)
\\
\,\,&\Longleftrightarrow\,\,\bigl(\forall\mathfrak{B}\in\underset{n\rightarrow\omega}{\lim}\,\mathfrak{M}_{n}\bigr)  \underset{n\rightarrow\omega}{\lim}\phi_{n}
\end{align*}
\ (ii) is proved similar to (i).
\end{proof}

\begin{corollary}
	\label{14}\qquad
\end{corollary}

\begin{enumerate}
	\item $\underset{n\rightarrow\omega}{\lim}\left(\mathfrak{M}_{n}\in\mathfrak{N}_{n}\right)\,\Longleftrightarrow\,\bigl(\underset {n\rightarrow\omega}{\lim}\,\mathfrak{M}_{n}\in\underset{n\rightarrow\omega}{\lim}\mathfrak{N}_{n}\bigr)$
	
	\item $\underset{n\rightarrow\omega}{\lim}\left(\mathfrak{M}_{n}=\mathfrak{N}_{n}\right)\,\Longleftrightarrow\,\bigl(\underset {n\rightarrow\omega}{\lim}\,\mathfrak{M}_{n}=\underset{n\rightarrow\omega}{\lim}\mathfrak{N}_{n}\bigr)$
\end{enumerate}

\begin{proof}
\ (i) \ By axioms \ref{15}, \ref{19} and corollary \ref{12}
\begin{align*}
	\underset{n\rightarrow\omega}{\lim}\left(\mathfrak{M}_{n}\in\mathfrak{N}_{n}\right) \,\,&\Longleftrightarrow\,\,\underset{n\rightarrow\omega}{\lim}\,\exists\mathfrak{A}\left(\mathfrak{M}_{n}=\mathfrak{A}\,\wedge\,\mathfrak{A} \in\mathfrak{N}_{n}\right)
	\\
	\,\,&\Longleftrightarrow\,\,\underset{n\rightarrow\omega}{\lim}\,\exists\mathfrak{A}\left(\forall\mathfrak{B}\left(\mathfrak{B}\in\mathfrak{M}_{n}\Longleftrightarrow\mathfrak{B}\in\mathfrak{A}\right)\wedge\,\mathfrak{A} \in\mathfrak{N}_{n}\right)
	\\
	\,\,&\Longleftrightarrow\,\,\exists\mathfrak{A}\,\bigl(\forall\mathfrak{B}\,\bigl(\mathfrak{B}\in\underset{n\rightarrow\omega}{\lim}\,\mathfrak{M}_{n}\Longleftrightarrow\mathfrak{B}\in\mathfrak{A}\bigr)\wedge\, \mathfrak{A}\in\underset{n\rightarrow\omega}{\lim}\mathfrak{N}_{n}\bigr)
	\\
	\,\,&\Longleftrightarrow\,\,\bigl(\underset{n\rightarrow\omega}{\lim}\,\mathfrak{M}_{n}\in\underset{n\rightarrow\omega}{\lim}\mathfrak{N}_{n} \bigr)
\end{align*}

(ii) is proved similar to (i).
\end{proof}

\begin{corollary}
\label{67}\qquad $\underset{n\rightarrow\omega}{\lim}\left(\exists\,\mathfrak{M}_{n}\in\mathfrak{N}_{n}\right)  \phi_{n}\,\Longleftrightarrow\,\bigl(\exists\,\underset{n\rightarrow\omega}{\lim}\,\mathfrak{M}_{n}\in\underset{n\rightarrow\omega}{\lim} \mathfrak{N}_{n}\bigr)\underset{n\rightarrow\omega}{\lim}\phi_{n}$
\end{corollary}

\begin{proof}
\ By axioms \ref{15}, \ref{19} and corollary \ref{12}
\begin{align*}
\underset{n\rightarrow\omega}{\lim}\left(\exists\,\mathfrak{M}_{n}\in\mathfrak{N}_{n}\right)\phi_{n}
\,\,&\Longleftrightarrow\,\,\underset{n\rightarrow\omega}{\lim}\,\exists\mathfrak{A}\left(\mathfrak{M}_{n}=\mathfrak{A}\,\wedge\,\mathfrak{A} \in\mathfrak{N}_{n}\,\wedge\,\phi_{n}\right) 
\\ \,\,&\Longleftrightarrow\,\,\underset{n\rightarrow\omega}{\lim}\,\exists\mathfrak{A}\left(\forall\mathfrak{B}\left(\mathfrak{B}\in\mathfrak{M}_{n}\Longleftrightarrow\mathfrak{B}\in\mathfrak{A}\right)\wedge \mathfrak{A}\in\mathfrak{N}_{n}\wedge\phi_{n}\right)
\\
\,\,&\Longleftrightarrow\,\,\exists\mathfrak{A}\,\bigl(\forall\mathfrak{B}\,\bigl(\mathfrak{B}\in\underset{n\rightarrow\omega}{\lim}\mathfrak{M}_{n}\Longleftrightarrow\mathfrak{B}\in\mathfrak{A}\bigr)\wedge \,\mathfrak{A}\in\underset{n\rightarrow\omega}{\lim}\mathfrak{N}_{n}\,\wedge\underset{n\rightarrow\omega}{\lim}\phi_{n}\bigr)
\\
\,\,&\Longleftrightarrow\,\,\exists\mathfrak{A}\,\bigl(\underset{n\rightarrow\omega}{\lim}\,\mathfrak{M}_{n}=\mathfrak{A}\,\wedge\, \mathfrak{A} \in\underset{n\rightarrow\omega}{\lim}\mathfrak{N}_{n}\,\wedge\,\underset{n\rightarrow\omega}{\lim}\phi_{n}\bigr)
\\
\,\,&\Longleftrightarrow\,\,\bigl(\exists\,\underset{n\rightarrow\omega}{\lim}\,\mathfrak{M}_{n}\in\underset{n\rightarrow\omega}{\lim} \mathfrak{N}_{n}\bigr)\underset{n\rightarrow\omega}{\lim}\phi_{n}
\end{align*}\vspace{-12pt}
\end{proof}

\begin{corollary}
\label{20}
\[
(\exists N<\omega)\,(\forall n>N)\left(\mathfrak{M}_n\in\mathfrak{N}\right)\,\Longrightarrow\,\bigl(\underset{n\rightarrow\omega}{\lim}\,\mathfrak{M}_{n}\in\mathfrak{N}\bigr)
\]
\end{corollary}

\begin{proof}
\ Clearly
\[
(\forall n>N)\left(\mathfrak{M}_n\in\mathfrak{N}\right)\,\Longrightarrow\,(\forall n>N)\,\exists\mathfrak{A}\left(\mathfrak{M}_n=\mathfrak{A}\,\wedge\,\mathfrak{A}\in\mathfrak{N}\right)
\]
By corollary \ref{67}
\begin{align*}
\underset{n\rightarrow\omega}{\lim}\,\exists\mathfrak{A}\left(\mathfrak{M}_n=\mathfrak{A}\,\wedge\,\mathfrak{A}\in\mathfrak{N}\right)
\,&\Longleftrightarrow\,\exists\mathfrak{A}\,\bigl(\underset{n\rightarrow\omega}{\lim}\,\mathfrak{M}_n=\mathfrak{A}\,\wedge\,\mathfrak{A}\in\mathfrak{N}\bigr)
\\
\,\,&\Longrightarrow\,\bigl(\underset{n\rightarrow\omega}{\lim}\,\mathfrak{M}_{n}\in\mathfrak{N}\bigr)
\end{align*}
So it follows by corollary \ref{18}.\bigskip
\end{proof}

Lastly, we show that the theory of dense linear order without endpoints can be obtained through the limit of formulas.

\begin{corollary}
\ Suppose $T$ is the theory of DLO without endpoints and $\,\mathfrak{M}_{n}\,=\bigl\langle\,\,\smashoperator{\bigcup_{1\leqslant j<n}} \,\bigl(\mathbb{Z}+j/n\bigr),\leqslant,+,\cdot,0,1\bigr\rangle$. Then $\underset{n\rightarrow\omega}{\lim}\, \mathfrak{M}_{n}=\mathbb{Q}$ and $T=\,$Th$(\mathbb{Q})$.
\end{corollary}

\begin{proof}
\ Suppose $\varphi_{n},\phi_{n},\delta_{n}$ are sentences specifying the properties of linear ordering, a dense subset and set without endpoints for $\mathfrak{M}_{n}$. Then
\begin{align*}
\varphi_{n} & \,\Longleftrightarrow\, \bigl(\forall x,y,z\in\mathfrak{M}_{n}\bigr) \bigl(x\leqslant x\wedge\left(x\leqslant y\wedge y\leqslant x\,\Longrightarrow x=y\right) \wedge \left(x\leqslant y\wedge y\leqslant z\,\Longrightarrow x\leqslant z\right)\bigr) 
\\
\phi_{n} & \,\Longleftrightarrow\, \bigl(\forall x,y\in\mathfrak{M}_{n}\bigr) \bigl(x<y\,\Longrightarrow\bigl(\exists z\in\mathfrak{M}_{l} \bigr)(l>n\,\wedge\,x<z<y)\bigr) 
\\
\delta_{n} & \,\Longleftrightarrow\, \bigl(\forall x\in\mathfrak{M}_{n}\bigr)\bigl(\left(\exists y\in\mathfrak{M}_{n}\right)(y<x) \wedge\left(\exists y\in\mathfrak{M}_{n}\right)(x<y)\bigr)
\end{align*}

For any $x,y\in\mathfrak{M}_{n}\left(x<y\right)$, set $N_{n}=2n$. Then $\forall k>N_{n},\,\exists z\in\mathfrak{M}_{k}$ that $x<z<y$, i.e. $\mathfrak{M}_{k}\vDash\phi_{n}$. Since $\mathbb{Z}\vDash\varphi_{n}\wedge\delta_{n}$, $\mathfrak{M}_{k}\vDash\varphi_{n}\,\wedge\,\phi_{n}\,\wedge\,\delta_{n}$. So $\underset{n\rightarrow\omega}{\lim}\, \mathfrak{M}_{n}$ is unique. By corollary \ref{24}
\begin{align*}
\underset{n\rightarrow\omega}{\lim}\varphi_{n} & \,\Longleftrightarrow\,\bigl(\forall x,y,z\in\underset{n\rightarrow\omega}{\lim}\, \mathfrak{M}_{n}\bigr)\bigl(x\leqslant x\wedge(x\leqslant y\wedge y\leqslant x\,\Longrightarrow x=y)  \wedge(x\leqslant y\wedge y\leqslant z\,\Longrightarrow x\leqslant z)\bigr) 
\\
\underset{n\rightarrow\omega}{\lim}\phi_{n} & \,\Longleftrightarrow\,\bigl(\forall x,y\in\underset{n\rightarrow\omega}{\lim}\, \mathfrak{M}_{n}\bigr)\bigl(x<y\,\Longrightarrow\bigl(\exists z\in\underset{n\rightarrow\omega}{\lim}\, \mathfrak{M}_{n}\bigr)(x<z<y)\bigr) 
\\
\underset{n\rightarrow\omega}{\lim}\delta_{n} & \,\Longleftrightarrow\,\bigl(\forall x\in\underset{n\rightarrow\omega}{\lim}\, \mathfrak{M}_{n}\bigr)\bigl(\bigl(\exists y\in\underset{n\rightarrow\omega}{\lim}\,\mathfrak{M}_{n}\bigr)(y<x)  \wedge\bigl(\exists y\in\underset{n\rightarrow\omega}{\lim}\,\mathfrak{M}_{n}\bigr)(x<y)\bigr)
\end{align*}

Since $T$ is $\aleph_{0}$-categorical and $\underset{n\rightarrow\omega}{\lim}\,\varphi_{n}, \underset {n\rightarrow\omega}{\lim}\phi_{n}$, $\underset{n\rightarrow\omega}{\lim}\delta_{n}$ are axioms of Th$(\mathbb{Q})$, $\underset{n\rightarrow\omega}{\lim}\,\mathfrak{M}_{n}=\mathbb{Q}$ and $T=\,$Th$(\mathbb{Q})$.
\end{proof}\smallskip

\section{Non-Well-Founded Sets}\label{SectionNonWellFoundedSets}

\subsection{Introduction}

In this section, we will study infinitely generated sets and three of their types known as infinitons, semi-infinitons and quasi-infinitons. First, let's review existing theories of non-well-founded sets.\smallskip

The investigation of non-well-founded sets was initiated by Mirimanoff in 1917 \cite{Mirimanoff}, in which he formulated the distinction between well-founded and non-well-founded sets. A number of axiomatic systems of non-well-founded sets have been proposed thereafter. Since the Zermelo--Fraenkel set theory bans $\in$-sequences of infinite length by the axiom of regularity\footnote{This is actually fallacious because the axiom of regularity can also hold for many non-well-founded sets. See section \ref{SectionConclusion}.} (also known as the axiom of foundation), most of these systems incorporate non-well-founded sets by replacing the axiom of regularity with distinct anti-foundation axioms and are essentially models of ZF minus the axiom of regularity. A notable exception is New Foundations by Quine \cite{Quine} that allows non-well-founded sets without a specific axiom and avoids Russell's paradox by permitting only stratified formulas.\smallskip

There are mainly four anti-foundation axioms by far --- AFA (by Aczel, Forti and Honsell \cite{Aczel}), SAFA (by Scott), FAFA (by Finsler), and BAFA (by Boffa). Each of them defines a different notion of equality for non-well-founded sets. For example, AFA bases hypersets (including non-well-founded sets) on accessible pointed graphs (APG) that two hypersets are equal if and only if they can be pictured by the same APG. In the universe of AFA, a Quine atom (called infiniton in this paper) is shown to be existent and unique. The anti-foundation axioms of AFA, SAFA, FAFA and BAFA specify an increasing sequence of universes over the von Neumann universe, i.e. $V\subset A\subset S\subset F\subset B$. In the universe of BAFA which is the largest of the four, the Quine atoms form a proper class.\smallskip

The main problem of the above axiomatic systems, however, is the lack of precise mathematical descriptions for non-well-founded sets. For instance, a non-well-founded set such as a Quine atom in AFA is an APG that can be unfolded into an infinite tree. As we learn later, a tree with an infinite branch is a countable structure that must be handled by the limit of finite structures and formulas. Consequently, AFA only describes countable structures intuitively and does not provide enough mathematical rigor for depicting their structures and operations.\smallskip

Furthermore, there have been efforts to introduce non-well-founded sets by enlarging the von Neumann universe (through removing the axiom of regularity). For examples, in \cite{Sharlow}, $V$ is modified through the iterative conception of a set that includes some non-well-founded sets; in \cite{Viale}, $V$ is expanded through the process of bisimulation. However, these attempts do not identify exactly why $V$ needs be enlarged, as well as precisely how the non-well-founded sets are generated.\smallskip

In this paper, we will present a new way to generate non-well-founded sets by enlarging the von Neumann universe along with the precise reason why $V$ needs be enlarged as well as the exact process to generate these sets. First, we can see that $V$ is incomplete because it does not have the limit ordinal ranks (lemma \ref{2}). This fact is of fundamental importance because it implies that non-well-founded sets are necessarily existent and should take on the limit ordinal ranks in a complete universe of sets.\footnote{More precisely, only generators of non-well-founded sets known as infinitely generated sets take on the rank of limit ordinals in the total universe. However, (many) non-well-founded sets can have the rank of successor ordinals.} Then non-well-founded sets are added to $V$ as infinitely generated sets with limit ordinal ranks to form the total universe (\ref{TotalHierarchy}). Furthermore, limits of finite structures and formulas discussed in the previous section can provide rigorous analysis for three types of infinitely generated sets as infinitons, semi-infinitons and quasi-infinitons that appear in Russell's paradox. Consequently, the total universe is a model of ZF minus the axiom of regularity and free of Russell's paradox.

\subsection{Infinitely Generated Set}

An infinitely generated set is a generator of non-well-founded sets and contains only one infinite branch. It is the limit of well-found sets known as finitely generated sets. As discussed in a previous section, an infinitely generated set\ is an infinite structure that is (generally) described by an infinitely long formula of $\mathscr{L}_{\omega_{1},\omega}$ \cite{Scott}. First, we introduce the notion of a finitely generated set.

\begin{definition}
\ Suppose $\bar{\mathscr{L}}=\{\in\}$ is the language of set theory and $G_{k}\in V_{\omega} \,\left(0\leqslant k\leqslant n\right)$. A \textbf{finitely generated set} is a finite $\bar{\mathscr{L}}$-structure that is defined as:
\begin{equation}
H_{n}(G_{n},\dots,G_{0})\,=\,\{\ast G_{n},\{\ast G_{n-1},\dots\{\ast G_{1},G_{0}\}\dots\}\footnote{$\ast G$ is the unpacking operator as in definition \ref{DefUnpackOperator}.}\label{FiniteGeneratedSet}
\end{equation}
Where $G_{k} \left(1\leqslant k\leqslant n\right)$ are \textbf{principal} \textbf{generators} and $G_{0}$ is a \textbf{base generator} of $H_{n}$. $H_n$ can also be defined recursively as:
\[
H_{n}(G_{n},\dots,G_{0})\,=\,\{\ast G_{n},H_{n-1}(G_{n-1},\dots,G_{0})\}
\]

\end{definition}

An infinitely generated set is defined as the limit (definition \ref{DefLimit}) of finitely generated sets in $V_{\omega}$. In a later section, we will extend it to well-founded sets of higher ranks in the
total universe.

\begin{definition}
\label{DefInfGeneratedSet} \ Suppose $H_{n}$ is defined in (\ref{FiniteGeneratedSet}) and $\mathcal{G}=\{G_{n}\colon G_{n}\in V_{\omega},\,n<\omega\}$. An \textbf{infinitely generated set (IGS)} (at $\omega$) is defined as:
\begin{equation}
H_{\omega}(\mathcal{G})\,=\,\underset{n\rightarrow\omega}{\lim}\,H_{n}(G_{n},\dots,G_{0})\label{InfGeneratedSet}
\end{equation}
Where $G_{n}\,(n\geqslant1)$\ are \textbf{principal} \textbf{generators} and $G_{0}$ is a \textbf{base generator} of $H_{\omega}$. The language of set theory is expanded to $\bar{\mathscr{L}}^{\prime}=\{\in, H_{\omega}\}$.
\end{definition}

\begin{definition}
\label{DefHeightFunction}\ Let $h\colon V_{\omega}\to \mathrm{Ord}$ and $h(X)=\sup\{R_V(Y)\colon Y\in X\wedge X\in V_{\omega}\}$ where $h(\varnothing)=0$. Then $h$ is known as the \textbf{height function} in $V_{\omega}$. Clearly, $h(H_n)=R_V(G_0)+n$ $(n<\omega)$.\footnote{$h(X)$ measures the maximum number of curly brackets of $X$ in $V_{\omega}$ and is equal to the membership dimension of $X$ (definition \ref{DefMembershipDimension}).}
\end{definition}

From above definitions, we can see that each IGS has only one infinite branch. A non-well-founded set with multiple infinite branches can be formed from IGS through power set operations. Hence IGS are generators of the non-well-founded sets.\medskip

From definition \ref{DefLimit}, it is clear that the limit of finitely generated sets exists if $H_n$ is a homogeneous sequence in a $\aleph_{0}$-categorical theory. As a result, it is essential to find out conditions for Th$(H_{n})$ to be $\aleph_{0}$-categorical and homogeneous. First, let's review more background knowledge in model theory.\medskip

Let $\mathfrak{A}=(A,\dots)$ and $\mathfrak{B}=(B,\dots)$ be models of $\mathscr{L}$, $\bar{a}=(a_{1},\dots,a_{n})$ ($a_{1},\dots,a_{n}\in A$) and $\bar{b}=(b_{1},\dots,b_{n})$ ($b_{1},\dots,b_{n}\in B$). A \textbf{partial isomorphism} $I\colon \mathfrak{A}\cong_p\mathfrak{B}$ between $\mathfrak{A}$ and $\mathfrak{B}$ is a relation $I$ on $\bar{a}$ and $\bar{b}$ satisfying the following three properties: (i) $\varnothing\,I\,\varnothing$; (ii) If $\,\bar{a}\,I\,\bar{b}$, then $(\mathfrak{A},\bar{a})$ and $(\mathfrak{B},\bar{b})$ satisfy the same atomic sentences of $\mathscr{L}^{\prime}$; (iii) If $\,\bar{a}\,I\,\bar{b}$, then for any $c\in A$, there exists $d\in B$ such that $(\bar{a},c)\,I\,(\bar{b},d)$, and vice versa. Condition (iii) is known as the \textbf{back and forth condition}. $\mathfrak{A}$ is $\omega$-homogeneous if and only if there is a partial isomorphism from $\mathfrak{A}$ to $\mathfrak{A}$. Any two countable partial isomorphic models are isomorphic. Thus any two models satisfying the back and forth condition are isomorphic \cite{Chang-Keisler}.\medskip

A first-order theory $T$ has \textbf{quantifier elimination} if, for every formula $\varphi(x_{1},\dots,x_{n})$ there is a quantifier-free formula $\phi(x_{1},\dots,x_{n})$ such that $T\vdash\forall x_{1}\dots\forall x_{n}(\varphi(\bar{x})\leftrightarrow \phi(\bar{x}))$. The \textbf{skeleton} (or \textbf{age}) of a countable $\mathscr{L}$-structure $\mathfrak{M}$ is the class of all finite $\mathscr{L}$-structures, each of which is isomorphic to a substructure of $\mathfrak{M}$. An \textbf{amalgamation} class (or \textbf{Fraisse class}) is a class of finite $\mathscr{L}$-structures with the hereditary, the joint embedding and the amalgamation properties \cite{Hodges}. We summarize the above in the following propositions without proof.

\begin{proposition}
\label{21}\cite[Corollary 3.1.3]{Macpherson} \ Let $\mathfrak{M}$ be a countable $\mathscr{L}$-structure which is homogeneous in a finite relational language. Then Th$(\mathfrak{M})$ is $\aleph_{0}$-categorical.
\end{proposition}

\begin{proposition}
\label{22}\cite[Proposition 3.1.6]{Macpherson} \ Let $\mathfrak{M}$ be a $\aleph_{0}$-categorical structure in a relational language. Then $\mathfrak{M}$ is ultrahomogeneous if and only if Th$(\mathfrak{M})$ has quantifier elimination.
\end{proposition}

\begin{proposition}
\label{9}\cite[Theorem 7.1.2]{Hodges} \ A countable class of $\mathscr{L}$-structures is an amalgamation class if and only if it is the skeleton of a
countable ultrahomogeneous $\mathscr{L}$-structure $\mathfrak{M}$. The amalgamation class is unique and is the Fraisse limit\textbf{\ }of $\mathfrak{M}$.
\end{proposition}

Now we apply the above results to study the limit of $H_{n}$ in the expanded language of set theory $\bar{\mathscr{L}}^{\prime}=\{\in, H_{\omega}\}$. First, we need the notion of an amalgamation class in $V_{\omega}$.

\begin{definition}
\label{DefAmalgamation} \ An \textbf{amalgamation class} $\mathcal{K}$ of $V_{\omega}$ is a collection of finitely generated sets in $V_{\omega}$ that satisfies the following properties.
\end{definition}

\begin{enumerate}
\item \textit{\textbf{(Heredity)} \ If $H\in\mathcal{K}$, then any $J$ that is isomorphic to a finitely generated subset of $H$ is in $\mathcal{K}$.}

\item \textit{\textbf{(Joint embedding)} \ If $H_1,H_2\in\mathcal{K}$, then there is a $J\in\mathcal{K}$ and embeddings $f\colon H_1\to J$ and $g\colon H_2\to J$.}

\item \textit{\textbf{(Amalgamation)} \ If $H_1,H_2,J_1\in\mathcal{K}$ and embeddings $f_0\colon H_1\to H_2$ and $f_1\colon H_1\to J_1$, then there is a $J_2\in\mathcal{K}$ and embeddings $g_0\colon H_2\to J_2$ and $g_1\colon J_1\to J_2$ with $g_0\circ f_0=g_1\circ f_1$.}
\end{enumerate}

\begin{lemma}
\label{11} \ Suppose $H_{n}$ is defined in (\ref{FiniteGeneratedSet}) and $\mathcal{G}=\{G_{n}\colon G_{n}\in V_{\omega},\,n<\omega\}$. If $\{H_n\colon n<\omega\}$ is an amalgamation class, then Th$(H_{n})$ is $\aleph_{0}$-categorical.
\end{lemma}

\begin{proof}
Since $H_{n}$ is the only (isomorphic) copy of finite $\bar{\mathscr{L}}$-structures (finitely generated sets) of rank $n$, its skeleton is $\{H_n\colon n<\omega\}$. By proposition \ref{9}, $\{H_n\}$ is ultrahomogeneous. So by proposition \ref{21}, Th$(H_{n})$ is $\aleph_{0}$-categorical.
\end{proof}

\begin{corollary}
\label{10} \ Suppose $H_{n}$ is an amalgamation class.
\end{corollary}

\begin{enumerate}
\item \textit{If $H_{n}$ is a homogeneous sequence, then $\underset{n\rightarrow\omega}{\lim}H_{n}$ is unique.}

\item \textit{If $H_{n}$ consists only of finitely many homogeneous subsequences, then there are finitely many sublimits for $\underset{n\rightarrow\omega}{\lim}H_{n}$.}  
\end{enumerate}

\begin{proof}
By lemma \ref{11}, definition \ref{DefLimit} and \ref{DefSubLimit}.\bigskip
\end{proof}

An infinitely generated set is a set as well as a countable structure. Therefore, if two IGS are equal, they must satisfy the axiom of extensionality in set theory. In addition, they are isomorphic structures and satisfy the back and forth condition in model theory. So two IGS are equal if and only if their generators at each level are identical. As a result, the axiom of extensionality for infinitely generated sets in $\bar{\mathscr{L}}^{\prime}$ is as follows.

\begin{axiom}
\label{50}\ Suppose $H_{\omega}(\mathcal{G}_{1})$ and $H_{\omega}(\mathcal{G}_{2})$ are two IGS (with a unique limit respectively) where $\,\mathcal{G}_{1}\,=\, \{G_{n}^{1}\colon G_{n}^{1}\in V_{\omega},\,n<\omega\}$ and $\,\mathcal{G}_{2}\,=\,\{G_{n}^{2}\colon G_{n}^{2}\in V_{\omega},\,n<\omega\}$. Then
\[
\left(\forall n<\omega\right)\left(G_{n}^{1}=G_{n}^{2}\right)\,\Longrightarrow\,H_{\omega}(\mathcal{G}_{1}) = H_{\omega}(\mathcal{G}_{2})
\]
\end{axiom}

The following concept is significant for the rest discussion.

\begin{definition}
\label{DefOmegaInvariant} \ Suppose $H_{\omega+\gamma}=\{\ast G_{\omega+\gamma},\,H_{\omega+\gamma-1}\}$ where $G_{\omega+\gamma}\in V_{\omega}\left(\gamma\geqslant1\right)$. If for any $\alpha>\omega$, there is a (successor ordinal) $\beta>\alpha$ that $H_{\beta}=H_{\omega}$, then $H_{\omega}$ is called $\mathbf{\omega}$\textbf{-invariant}.
\end{definition}

\begin{remark}
\label{52}\ Generally, an IGS does not have an immediate member, i.e. there is no $z\in H_{\omega}$. The significance of $\omega$-invariance is that a $\omega$-invariant set always has an immediate member. For example, if $H_{\omega}$ is $\omega$-invariant, i.e. there is a $\beta>\omega$ that $H_{\beta}=H_{\omega}$, then $H_{\beta-1}\in H_{\beta}=H_{\omega}$.
\end{remark}

Next, we will study three types of infinitely generated sets that are fundamental in Russell's paradox.

\subsection{Infiniton}\label{SectionInfiniton}

An infiniton is a set that contains itself as the only member, i.e. $I=\{I\}$, a fact that will be proved rigorously next.

\begin{theorem}
\label{38} \ Suppose for each $n<\omega,$ $I_{n}$ is a finitely generated set and $I_{n}=\{I_{n-1}\}$ with $I_{0}\,=\,G_{0}\in V_{\omega}$, $\mathcal{I}_n=\langle \{I_n\},\in,h,G_0\rangle$, $\mathfrak{I}_n=\langle\{I_j\colon j\leqslant n\}, \in,h,G_0\rangle$ and $\mathfrak{I}=\bigcup_{n<\omega}\mathfrak{I}_n$.\footnote{$h$ is the height function in definition \ref{DefHeightFunction}.}  Then
\end{theorem}

\begin{enumerate}
\item $I_{n}\,=\underset{n}{\,\,\underbrace{\{\dots\{G_{0}\}\dots\}}}$.

\item \textit{Th$(\mathfrak{I})$ is $\aleph_{0}$-categorical and has quantifier elimination.}

\item \textit{$\mathfrak{I}$ has only one complete type.}

\item $\underset{n\rightarrow\omega}{\lim}\mathcal{I}_{n}$ \textit{and} $\underset{n\rightarrow\omega}{\lim}\mathfrak{I}_{n}$ \textit{is unique.}
\end{enumerate}

\begin{proof}
\ (i) follows easily by replacing $I_{n}$ recursively $n$ times.\medskip

(ii) \ For any $i<j<\omega$, suppose $f\colon I_{i}\to I_{j}$ is an isomorphism. Clearly, $f\subset g$ where $g\colon I_{n}\to I_{n+j-i}$ (for any $n<\omega$) is an automorphism on $\mathfrak{I}$. So $\mathfrak{I}$ is ultrahomogeneous. By proposition \ref{21} and \ref{22}, Th$(\mathfrak{I})$ is $\aleph_{0}$-categorical and has quantifier elimination.\medskip

(iii)\ Suppose $B=\underset{l}{\,\,\underbrace{\{\dots\{G_{0}\}\dots\}}}$ $(l<\omega)$ and
\begin{equation}
\varphi_n(x)\,\Longleftrightarrow\,\exists!\,y_{n-1}\,\cdots\,\exists!\,y_1\, \exists!\,y_0 \,\bigl(y_0\in y_1\,\wedge\dots\wedge\,y_{n-1}\in x\,\wedge\, y_0=B\bigr)\label{InfinitonOneType}
\end{equation}

Then the validity of $\varphi_n(x)$ means that there is a unique $\in$-sequence of length $n$ in $x$. If $l=0$, $B=G_0$ and $\varphi_n(x)\Leftrightarrow h(x)=R_V(G_0)+n$, indicating that the terminal of the unique $\in$-sequence is $G_0$. Thus $(\varphi_n(x))$ is a $1$-type of $\mathfrak{I}$ for $\mathcal{I}_n\vDash\varphi_n[I_n]$. Next suppose
\begin{equation}
\phi_n(x_1,\dots,x_n)\,\Longleftrightarrow\quad\smashoperator{\bigwedge_{1\leqslant i\leqslant n-1}}\,\,\bigl(x_{i}\in x_{i+1}\wedge h(x_{i+1})=h(x_i)+1\bigr)\label{InfinitonNType}
\end{equation}

Then $(\phi_n(x_1,\dots,x_n))$ is a $n$-type of $\mathfrak{I}$ for $\mathfrak{I}_n\vDash\phi_n[I_1,\dots,I_n]$. Clearly 
\[
\bigl\{\, \varphi_1(x_1),\dots, \varphi_n(x_1),\dots, \phi_2(x_1,x_2),\dots,\phi_n(x_1,x_2, \dots,x_n)\dots\bigr\}
\]
generates a maximum consistent set of formulas involving $G_0$ that is the only complete type of $\mathfrak{I}$.\medskip

(iv)\ By (iii), for any $n<\omega$, $\mathcal{I}_n \vDash \varphi_n[I_n]$, and for any $k>n,\,\mathcal{I}_k\vDash \varphi_n[I_k]$. So $(\mathcal{I}_{n},\varphi_{n})$ is a homogeneous sequence. By (ii) and definition \ref{DefLimit}, $\underset{n\rightarrow\omega}{\lim}\mathcal{I}_{n}$ is unique. Likewise, since for any $k>n$, $\mathfrak{I}_k\vDash \phi_n[I_{k-n+1},\dots,I_k]$, $(\mathfrak{I}_{n},\phi_n)$ is a homogeneous sequence. Thus $\underset{n\rightarrow\omega}{\lim}\mathfrak{I}_{n}$ is unique.
\end{proof}

\begin{definition}
\label{DefInfiniton} \ In theorem \ref{38}, $\underset{n\rightarrow\omega}{\lim}\mathcal{I}_{n}$ is known as the \textbf{infiniton} generated by $G_{0}$ and is denoted as:
\[
\underset{n\rightarrow\omega}{\lim}\mathcal{I}_{n}=\underset{n\rightarrow\omega}{\lim}I_{n}\,=\,I_{\omega}\,=\,\underset{\aleph_{0}}{\underbrace{\{\dots\{G_{0}\}\dots\}}} \,=\,\{G_{0}\}_{\mathcal{I}}
\]
Where $G_{0}$ is a \textbf{base generator} of $I_{\omega}$. $\bar{\mathscr{L}}^{\prime}=\{\in, H_{\omega}, I_{\omega}\}$ after $I_{\omega}$ is added as a constant.\footnote{In the rest discussion, we will no longer distinguish $\underset{n\rightarrow\omega}{\lim}\mathcal{I}_{n}$ and $\underset{n\rightarrow\omega}{\lim}I_{n}$.}
\end{definition}

\begin{definition}
\label{DefSetInfinitons}\ $S|_{\mathcal{I}}\,=\,\{\{G_{0}\}_{\mathcal{I}}\colon G_{0}\in S\}$ is known as the \textbf{set of the infinitons} from $S$.
\end{definition}

\begin{theorem}
\label{30}\ Suppose $I_n$ is the same as in theorem \ref{38} and $\mathfrak{I}^+=\langle\{I_n\colon n\leqslant\omega\},\in,h, G_0\rangle$. Then
\end{theorem}

\begin{enumerate}
\item $I_{\omega}\,=\,\{I_{\omega}\}$.

\item \textit{$I_{\omega}$ is $\omega$-invariant.}

\item \textit{A type of $I_{\omega}$ is that $I_{\omega}$ has a unique $\in$-sequence of length $\omega$ and is the member of itself.}

\item $\underset{n\rightarrow\omega}{\lim}\mathfrak{I}_{n}=\mathfrak{I}^+$ and $\mathfrak{I}^+$ \textit{is atomic}.
\end{enumerate}

\begin{proof}
\ (i) \ By theorem \ref{38}, $I_{\omega}$ is unique. So by corollary \ref{42} and \ref{29}
\[
I_{\omega}\,=\underset{n\rightarrow\omega}{\lim}I_{n}\,=\underset{n\rightarrow\omega}{\lim}\{I_{n-1}\}\,=\, \{\underset{n\rightarrow\omega}{\lim}I_{n-1}\}\,=\,\{I_{\omega}\}
\]

(ii) \ Obviously, for any $\alpha$ that $\omega<\alpha<\omega2$
\[
I_{\alpha}\,=\,\{I_{\alpha-1}\}\,=\dots=\,\underset{\alpha-\omega}{\underbrace{\{\dots\{I_{\omega}\}\dots\}}}\,=\, I_{\omega}
\]

Then by transfinite induction, for any $\alpha,\,I_{\alpha}\,=\,I_{\omega}$. So (ii) follows by definition \ref{DefOmegaInvariant}.\medskip

(iii) \ Let $\left\langle I_{j}\colon j\leqslant n\right\rangle$ be the unique $\in$-sequence of length $n$ in $I_{n}$. Then by (\ref{InfinitonOneType})
\[
\varphi_{n}(I_n)\,\Longleftrightarrow\,\exists!\,I_{n-1}\,\cdots\,\exists!\,I_{0}\,\bigl(\,\,\,\smashoperator{\bigwedge_{1\leqslant j\leqslant n-1}}\,(I_{j-1} \in I_{j})\wedge I_0=G_0\bigr)\wedge\,\exists!\,I_{n-1}(I_{n-1}\in I_{n})
\]

Since $\underset{n\rightarrow\omega}{\lim}I_{n}\,=\,I_{\omega}$, by axiom \ref{15}, corollary \ref{17} and \ref{67}
\[
\varphi_{\omega}\,\Longleftrightarrow\underset{n\rightarrow\omega}{\lim}\,\varphi_{n}\,\Longleftrightarrow\underset{n<\omega}{\displaystyle\bigwedge}\exists! \,I_{n}\,\exists!\,I_{n-1}\left(I_{n-1}\in I_{n}\right)\,\wedge\,\exists!\,I_0\,(I_0=G_0)\,\wedge\,\exists!\,I_{\omega}\left(I_{\omega}\in I_{\omega}\right)
\]

By theorem \ref{71}, $I_{\omega}\vDash\varphi_{\omega}$ where $\varphi_{\omega}$ defines $\left\langle I_{n}\colon n\leqslant\omega\right\rangle$, the unique $\in$-sequence of length $\omega$ in
$I_{\omega}$. \medskip

(iv) \ Fix $k<\omega$, for any $n>k$, $I_k\in\mathfrak{I}_n$. By theorem \ref{38}(iv) and axiom \ref{19}, $\underset{n\rightarrow\omega}{\lim}(I_k\in\mathfrak{I}_{n})\Leftrightarrow I_k\in\underset{n\rightarrow\omega}{\lim}\mathfrak{I}_{n}$. So by corollary \ref{18}, for any $k<\omega$, $I_k\in\underset{n\rightarrow\omega}{\lim}\mathfrak{I}_{n}$. By corollary \ref{20}, $I_{\omega}=\underset{k\rightarrow\omega}{\lim}I_k\in\underset{n\rightarrow\omega}{\lim}\mathfrak{I}_{n}$ and so $\mathfrak{I}^+\subset\underset{n\rightarrow \omega}{\lim}\mathfrak{I}_{n}$. On the other hand, by theorem \ref{38}(iv), $\underset{n\rightarrow\omega}{\lim}\mathfrak{I}_{n}$ is unique. Thus $\underset{n\rightarrow \omega}{\lim}\mathfrak{I}_{n}\subset\mathfrak{I}^+$ and $\underset{n\rightarrow \omega}{\lim}\mathfrak{I}_{n}=\mathfrak{I}^+$. So by (\ref{InfinitonNType}), corollary \ref{17} and \ref{14}
\begin{align*}
\phi_{\omega}\,\Longleftrightarrow\,\underset{n\rightarrow\omega}{\lim}\,\phi_{n}\,&\,\Longleftrightarrow\,\underset{n\rightarrow\omega}{\lim}\quad\smashoperator{\bigwedge_{1\leqslant i\leqslant n-2}}\,\,\bigl(x_{i}\in x_{i+1}\,\wedge\, h(x_{i+1})=h(x_i)+1\bigr)
\\
&\qquad\,\,\,\wedge\,\underset{n\rightarrow\omega}{\lim}\bigl(x_{n-1}\in x_{n}\,\wedge\, h(x_{n})=h(x_{n-1})+1\bigr)
\\
\,&\Longleftrightarrow\underset{n<\omega}{\displaystyle\bigwedge}\bigl(x_{n-1}\in x_{n}\,\wedge\, h(x_{n})=h(x_{n-1})+1\bigr)\,\wedge\,\bigl(x_{\omega}\in x_{\omega}\bigr)
\end{align*}
where $h$ is the height function (definition \ref{DefHeightFunction}) and $\underset{n\rightarrow\omega}{\lim}x_n=x_{\omega}$. Since $\mathfrak{I}^+\vDash\phi_{\omega}$, any tuple from $\{I_n\colon n\leqslant\omega\}$ satisfies $\phi_{\omega}$. By theorem \ref{38}(iii), Th$(\mathfrak{I}^+)$ has only one type. So by proposition \ref{25},  $\mathfrak{I}^+$ is atomic with $\phi_{\omega}$ being its complete formula. Since any type of Th$(\mathfrak{I}^+)$ can be derived from it,  $\phi_{\omega}$ is a $\omega$-type with countable free variables.\smallskip
\end{proof}

\begin{corollary}
\label{31} \ Suppose $I_{0}^{\prime}\,=\underset{k}{\,\underbrace{\{\dots\{G_{0}\}\dots\}}}$ and $I_{n}^{\prime}= \{I_{n-1}^{\prime}\}$ for each $n<\omega$. Then $\underset{n\rightarrow\omega}{\lim}I_{n}^{\prime}\,= \underset{n\rightarrow\omega}{\lim}I_{n}$.
\end{corollary}

\begin{proof}
\ Since $I_{n}\,=\underset{n}{\,\underbrace{\{\dots\{G_{0}\}\dots\}}}$, $I_{n}^{\prime}=I_{n+k}$. So it follows by corollary \ref{42}.\smallskip
\end{proof}

\begin{remark}
\ Corollary \ref{31} shows that different base generators could generate the same infiniton. By choosing the one with the least rank, an infiniton is unique to its base generator.
\end{remark}

\begin{definition}
\label{DefBaseGeneratorInfiniton} \ The \textbf{base generator} $G_{0}$ of an infiniton $I_{\omega}$ is the one with the least rank (in $V_{\omega}$).
\end{definition}

\begin{corollary}
\label{32} \ $2$ \textit{infinitons are identical if and only if their base generators are the same.}
\end{corollary}

\begin{proof}
\ Suppose $I_{n}\,=\underset{n}{\,\underbrace{\{\dots\{G_{0}\}\dots\}}}$ and $I_{n}^{\prime}\,=\underset{n} {\,\underbrace{\{\dots\{G_{0}^{\prime}\}\dots\}}}$. If $G_{0}=G_{0}^{\prime}$, then $I_{n}=I_{n}^{\prime}$ for any $n<\omega$. So by corollary \ref{42}, $I_{\omega}=I_{\omega}^{\prime}$. Conversely, by theorem \ref{30}, if $I_{\omega}=I_{\omega}^{\prime}$, then $I_{\omega}$ and $I_{\omega}^{\prime}$ have the same $\in$-sequence of length $\omega$. Thus by definition \ref{DefBaseGeneratorInfiniton}, $G_{0}=G_{0}^{\prime}$.
\end{proof}

\begin{corollary}
\ Suppose $I$ is an infiniton and $S|_{\mathcal{I}}$ a set of infinitons.
\end{corollary}

\begin{enumerate}
\item $I\neq\varnothing$.

\item $D(I)\,=\,D(S|_{\mathcal{I}})\,=\,\aleph_{0}$.

\item \textit{Any infiniton and set of infinitons are TNWF.}

\item $I\notin V$.
\end{enumerate}

\begin{proof}
\ (i) \ Since $I\in I$ and $\varnothing\notin\varnothing,\,I\neq\varnothing$.\medskip

(ii) \ If $D(I)<\aleph_{0}$, then by (\ref{DefMembershipDimension}), $D(I)<D\left(\{I\}\right)=D(I)$, contradiction. And by definition \ref{DefSetInfinitons}, $D(S|_{\mathcal{I}})\geqslant D(I)\,=\,\aleph_{0}$.\medskip

(iii) \ By definition \ref{DefWFAndNWFSet}, \ref{DefInfiniton} and \ref{DefSetInfinitons}.\medskip

(iv) \ By lemma \ref{1} and (iii).\bigskip
\end{proof}

The tree structures for infinitons and a set of infinitons are shown in Figure \ref{Fig1}. Intuitively, any infiniton consists of one infinite (broken) branch, and all branches of a set of infinitons are infinite.

\begin{figure}[h]
	\centering
	\includegraphics{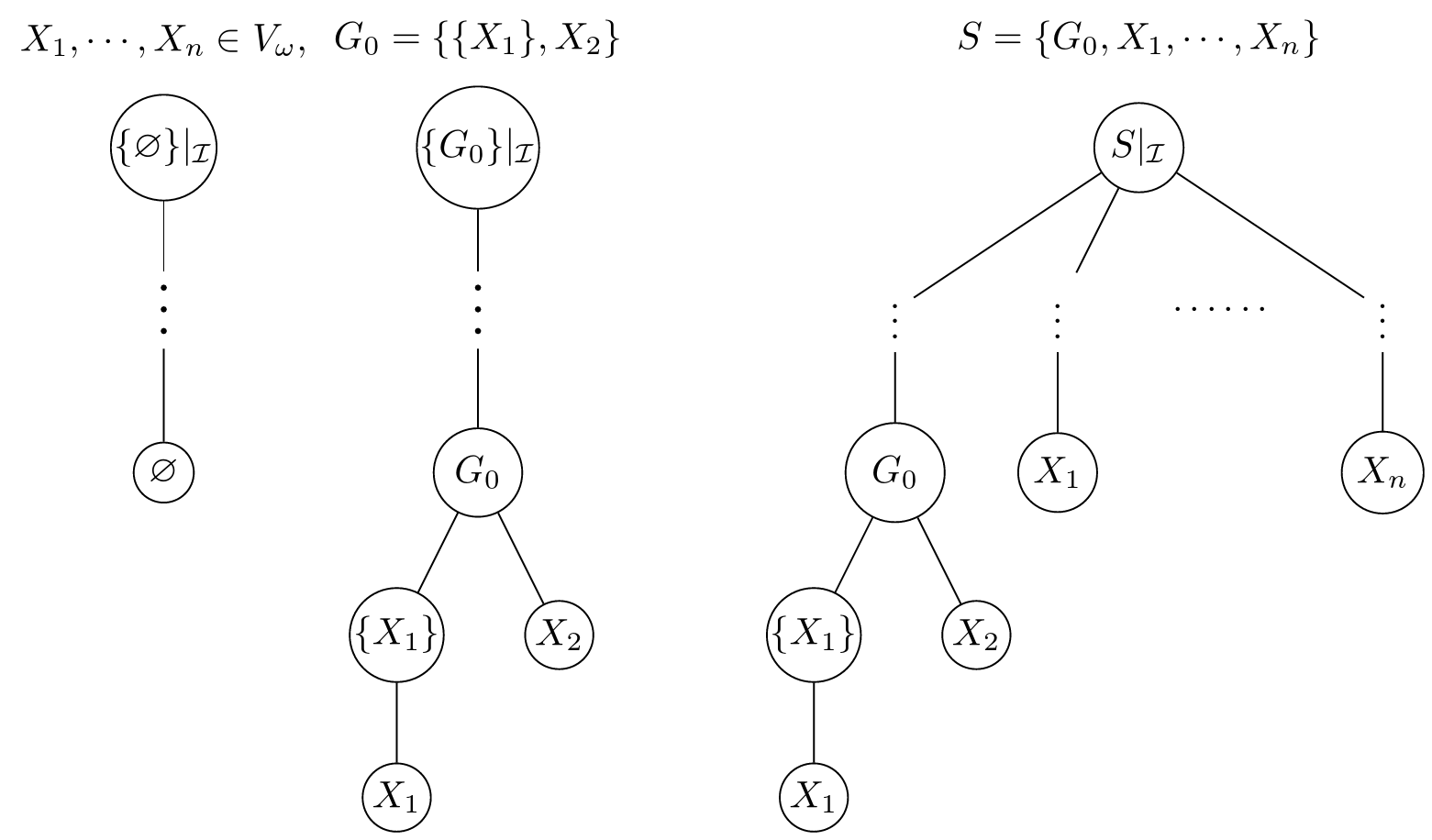}
	\caption{Diagrams of infinitons and a set of infinitons.}
	\label{Fig1}
	\centering
\end{figure}

\begin{theorem}
\qquad
\end{theorem}

\begin{enumerate}
\item $S_{1}\subset S_{2}\,\Longrightarrow\,S_{1}|_{\mathcal{I}}\subset S_{2}|_{\mathcal{I}}$

\item $\left(S_{1}\cup S_{2}\right)|_{\mathcal{I}}\,=\,S_{1}|_{\mathcal{I}}\cup S_{2}|_{\mathcal{I}}$

\item $\left(S_{1}\cap S_{2}\right)|_{\mathcal{I}}\,\subset S_{1}|_{\mathcal{I}}\cap S_{2}|_{\mathcal{I}}$

\item $S_{1}|_{\mathcal{I}}-S_{2}|_{\mathcal{I}}\,\subset\left(S_{1}-S_{2}\right)|_{\mathcal{I}}$

\item $\Bigl(\underset{\alpha\in D}{\displaystyle\bigcup}S_{\alpha}\Bigr)\Big\vert_{\mathcal{I}}\,= \underset{\alpha\in D}{\displaystyle\bigcup}S_{\alpha}\vert_{\mathcal{I}}$
\end{enumerate}

\begin{proof}
\ (i) is obvious.\medskip

(ii) \ $``\supset\textquotedblright$ follows from (i) for $\left(S_{1}\cup S_{2}\right) |_{\mathcal{I}}\supset S_{1}|_{\mathcal{I}}$ and $\left(S_{1}\cup S_{2}\right)|_{\mathcal{I}}\supset S_{2}|_{\mathcal{I}}$. For any $\{G_{0}\}|_{\mathcal{I}}\in\left(S_{1}\cup S_{2}\right)|_{\mathcal{I}}$, $G_{0}\in S_{1}\cup S_{2}$. So $\{G_{0}\}|_{\mathcal{I}}\in S_{1}|_{\mathcal{I}}$ or $\{G_{0}\}|_{\mathcal{I}}\in S_{2}|_{\mathcal{I}}$. This proves $``\subset\textquotedblright$.\medskip

(iii) \ By (i).\medskip

(iv) \ For any $\{G_{0}\}|_{\mathcal{I}}\in S_{1}|_{\mathcal{I}}-S_{2}|_{\mathcal{I}}$, $G_{0}\in S_{1}$ and $G_{0}\notin S_{2}$. So $\{G_{0}\}|_{\mathcal{I}}\in\left(S_{1}-S_{2}\right)|_{\mathcal{I}}$.\medskip

(v) \ By (i), for any $\alpha\in D$
\[
\left.S_{\alpha}\right\vert_{\mathcal{I}}\,\subset\Bigl(\underset{\alpha\in D}{\displaystyle\bigcup}S_{\alpha}\Bigr)  \Big\vert_{\mathcal{I}}\quad\text{and so}\quad\underset{\alpha\in D}{\displaystyle\bigcup}\left.S_{\alpha}\right\vert _{\mathcal{I}}\,\subset\Bigl(\underset{\alpha\in D}{\displaystyle\bigcup}S_{\alpha}\Bigr)\Big\vert_{\mathcal{I}}
\]

For any $\{G_{0}\}|_{\mathcal{I}}\,\in\Bigl(\underset{\alpha\in D}{\displaystyle\bigcup}S_{\alpha}\Bigr)\Big\vert _{\mathcal{I}}$, \ there is a $\beta\in D$ that $G_{0}\in S_{\beta}$. So
\[
\{G_{0}\}|_{\mathcal{I}}\in\left.S_{\beta}\right\vert_{\mathcal{I}}\,\subset\underset{\alpha\in D}{\displaystyle\bigcup} \left.S_{\alpha}\right\vert_{\mathcal{I}}
\]

Thus (v) follows.\smallskip
\end{proof}

\begin{corollary}
\label{33} \ The axiom of regularity fails for any infiniton and set of infinitons.
\end{corollary}

\begin{proof}
\ Suppose $I=\{I\}$ and $S|_{\mathcal{I}}=\{I_{k}\colon I_{k}=\{I_{k}\}\,\wedge\,k\in\mathbb{N}\}$. AR fails for $I$ since $I\in I$. For each $I_{k}\in S|_{\mathcal{I}}$, $I_{k}\cap S|_{\mathcal{I}}\,=\,I_{k}$ for $I_{k}\in I_{k}$. Since no $y\in S|_{\mathcal{I}}$ satisfies $y\cap S|_{\mathcal{I}}=\varnothing$, AR fails for $S|_{\mathcal{I}}$ too.
\end{proof}

\begin{corollary}
\label{34} \ \textit{Any set of infinitons is not a member of itself.}
\end{corollary}

\begin{proof}
\ Let $S=\{I_{1},I_{2},\dots\}$ where $I_{k}$ are distinct infinitons. If $S\in S$, then $S=I_{j}=\{I_{j}\}$. This leads to $I_{1}=I_{2}=\dots=I_{j}$, contradiction.\bigskip
\end{proof}

Corollary \ref{34} can be extended to any set like $\{\{\{I_{1},I_{2}\},I_{3}\},I_{4}\},$ i.e. all sets of infinitons, all subset of sets of infinitons and so on (corollary \ref{62}).

\subsection{Semi-Infiniton}\label{SectionSemiInfiniton}

A semi-infiniton is a set that is a member of itself, i.e. $X\in X$. From (\ref{SemiInfiniton0}) and (\ref{InfGeneratedSet}), we can see that an infinitely generated set with only one principal generator is a semi-infiniton.

\begin{theorem}
\label{37} \ Suppose for each $n<\omega,\,Z_{n}$ is a finitely generated set that $Z_{n}=\{\ast G,Z_{n-1}\}$ with $Z_{0}\,=\,G_{0}\in V_{\omega}$ and $G\in V_{\omega}$ where $R_V(G_0)\geqslant R_V(G)$\footnote{If $R_V(G_0)< R_V(G)$, let $G_0=\{*G,G_0\}$.}, $\mathcal{Z}_n=\langle \{Z_n\},\in,h,G_0,G\rangle$, $\mathfrak{Z}_n=\langle\{Z_j\colon j\leqslant n\}, \in,h,G_0,G\rangle$ and $\mathfrak{Z}=\bigcup_{n<\omega}\mathfrak{Z}_n$. Then
\end{theorem}

\begin{enumerate}
\item $Z_{n}\,=\,\underset{n}{\underbrace{\{\ast G,\{\ast G,\dots\{\ast G,G_{0}\}\dots\}}}$.

\item \textit{Th$(\mathfrak{Z})$ is $\aleph_{0}$-categorical and has quantifier elimination.}

\item \textit{$\mathfrak{Z}$ has only one complete type.}

\item $\underset{n\rightarrow\omega}{\lim}\mathcal{Z}_{n}$ \textit{and} $\underset{n\rightarrow\omega}{\lim}\mathfrak{Z}_{n}$ \textit{is unique.}
\end{enumerate}

\begin{proof}
\ (i) follows easily by replacing $Z_{n}$ recursively $n$ times.\medskip

(ii) \ For any $i<j<\omega$, suppose $f\colon Z_{i}\to Z_{j}$ is an isomorphism by mapping $Z_{i-1}$ to $Z_{j-1}$ and $G$ to $G$. Clearly, $f\subset g$ where $g\colon Z_{n}\to Z_{n+j-i}$ (for any $n<\omega$) is an automorphism on $\mathfrak{Z}$. So $\mathfrak{Z}$ is ultrahomogeneous. By proposition \ref{21} and \ref{22}, Th$(\mathfrak{Z})$ is $\aleph_{0}$-categorical and has quantifier elimination.\medskip

(iii)\ Suppose $\mu(y)\Leftrightarrow (\forall z\in G)(z\in y)$,  $C=\underset{l}{\underbrace{\{\ast G,\{\ast G,\dots\{\ast G,G_{0}\}\dots\}}}\:(l<\omega)$ and $C=*\{*G,G_0\}$ for $l=0$, and
\begin{equation}
\delta_{n}(x)\,\Longleftrightarrow\,\exists!\,y_{n-1}\,\cdots\, \exists!\,y_{0}\,\bigl(\,\,\,\smashoperator{\bigwedge_{1\leqslant i\leqslant n-1}}\,\, \bigl(y_{i-1}\in y_{i}\wedge\mu(y_i)\bigr)\wedge \bigl(y_0=C\wedge y_{n-1}\in x\wedge\mu(x)\bigr)\bigr)\label{SemiInfinitonOneType}
\end{equation} 
Then the validity of $\delta_n(x)$ means that there is a unique $\in$-sequence of length $n$ in $x$, each sublevel of which satisfies $\mu$. If $l=0$, $\delta_n(x)\Leftrightarrow h(x)=R_V(G_0)+n$ ($h(x)$ is the height function of $x$ as in definition \ref{DefHeightFunction}), indicating that the terminals of the unique $\in$-sequence are $G_0$ and any of $G$. Thus $(\delta_n(x))$ is a $1$-type of $\mathfrak{Z}$ for $\mathcal{Z}_n\vDash\delta_n[Z_n]$. Next suppose
\begin{equation}
\lambda_n(x_1,\dots,x_n)\,\Longleftrightarrow\,\,\,\smashoperator{\bigwedge_{1\leqslant i\leqslant n-1}}\,\,\bigl(x_{i}\in x_{i+1}\,\wedge\,\mu(x_{i+1})\,\wedge\,h(x_{i+1})=h(x_i)+1\bigr)\label{SemiInfinitonNType}
\end{equation}
Then $(\lambda_n(x_1,\dots,x_n))$ is a $n$-type of $\mathfrak{Z}$ for $\mathfrak{Z}_n\vDash\lambda_n[Z_1,\dots,Z_n]$. Clearly 
\[
\bigl\{\,\delta_1(x_1),\dots, \delta_n(x_1),\dots,\lambda_2(x_1,x_2),\dots, \lambda_n(x_1,x_2,\dots,x_n)\dots\bigr\}
\]
generates a maximum consistent set of formulas involving $G_0$ and $G$ that is the only complete type of $\mathfrak{Z}$.\medskip

(iv)\ By (iii), for any $n<\omega$, $\mathcal{Z}_n \vDash \delta_n[Z_n]$, and for any $k>n,\,\mathcal{Z}_k\vDash \delta_n[Z_k]$. So $(\mathcal{Z}_{n},\delta_{n})$ is a homogeneous sequence. By (ii) and definition \ref{DefLimit}, $\underset{n\rightarrow\omega}{\lim}\mathcal{Z}_{n}$ is unique. Likewise, since for any $k>n$, $\mathfrak{Z}_k\vDash \lambda_n[Z_{k-n+1},\dots,Z_k]$, $(\mathfrak{Z}_{n},\lambda_n)$ is a homogeneous sequence. Thus $\underset{n\rightarrow\omega}{\lim}\mathfrak{Z}_{n}$ is unique.
\end{proof}

\begin{definition}
\label{DefSemiInfiniton} \ In theorem \ref{37}, $\underset{n\rightarrow\omega}{\lim}Z_{n}$ is known as the \textbf{semi-infiniton} generated by $G\,(G\neq\varnothing)$ and $G_{0}$, and is
denoted as:
\[
\underset{n\rightarrow\omega}{\lim}\mathcal{Z}_{n}\,=\underset{n\rightarrow\omega}{\lim}Z_{n}\,=\,Z_{\omega}\,=\,\underset{\aleph_{0}}{\underbrace{\{\dots\{\ast G,\{\ast G,G_{0}\}\dots\}}}\,=\,\{G,G_{0}\}|_{\mathfrak{S}}
\]
Where $G$ is the \textbf{principal generator} and $G_{0}$ the \textbf{base generator} of $Z_{\omega}$. $\bar{\mathscr{L}}^{\prime}=\{\in, H_{\omega}, I_{\omega},Z_{\omega}\}$ after $Z_{\omega}$ is added as a constant.\footnote{In the rest discussion, we will no longer distinguish $\underset{n\rightarrow\omega}{\lim}\mathcal{Z}_{n}$ and $\underset{n\rightarrow\omega}{\lim}Z_{n}$.}
\end{definition}

\begin{theorem}
\label{36}\ Suppose $Z_n$ is the same as in theorem \ref{37} and $\mathfrak{Z}^+=\langle\{Z_n\colon n\leqslant\omega\}, \in,h,G_0,G\rangle$. Then
\end{theorem}

\begin{enumerate}
\item $Z_{\omega}\,=\,\{\ast G,Z_{\omega}\}$.

\item $Z_{\omega}$ \textit{is $\omega$-invariant.}

\item \textit{A type of $Z_{\omega}$ is that $Z_{\omega}$ has a unique $\in$-sequence of length $\omega$ and is the member of itself with each of its sublevel $n\,(n\leqslant\omega)$ containing the members of $G$.}  

\item $\underset{n\rightarrow\omega}{\lim}\mathfrak{Z}_{n}=\mathfrak{Z}^+$ and $\mathfrak{Z}^+$ \textit{is atomic}.
\end{enumerate}

\begin{proof}
\ (i) \ By theorem \ref{37}, $Z_{\omega}$ is unique. So by corollary \ref{42} and \ref{29}
\begin{align*}
Z_{\omega} \: & =\: \underset{n\rightarrow\omega}{\lim}\{\ast G,Z_{n-1}\}
\\
\: & =\: \underset{n\rightarrow\omega}{\lim}\left(G\,\cup\{Z_{n-1}\}\right)
\\
\: & =\: G\,\cup\{\underset{n\rightarrow\omega}{\lim}Z_{n-1}\}
\\
\: & =\: \{\ast G,Z_{\omega}\}
\end{align*}

(ii) \ Obviously, for any $\alpha$ that $\omega<\alpha<\omega2$
\[
Z_{\alpha}\,=\,\{\ast G,Z_{\alpha-1}\}\,=\dots=\,\underset{\alpha-\omega}{\underbrace{\{\ast G,\dots\{\ast G,\{\ast G,Z_{\omega}\}\}\dots\}}}\,=\,Z_{\omega}
\]

Then by transfinite induction, for any $\alpha,\,Z_{\alpha}\,=\,Z_{\omega}$. Also for any $z\in G,\,z\in Z_{\omega}$. So (ii) follows by definition \ref{DefOmegaInvariant}.\bigskip

(iii) \ Let $\mu(y)\Leftrightarrow (\forall z\in G)(z\in y)$ and $\left\langle Z_{j}\colon j\leqslant n\right\rangle$ be the unique $\in$-sequence of length $n$ in $Z_{n}$ with each sublevel $j$ $(j<n)$ of $Z_{n}$ containing the members of $G$. Then by (\ref{SemiInfinitonOneType})
\[
\delta_{n}(Z_n)\,\Longleftrightarrow\,\exists!\,Z_{n-1}\,\cdots\,\exists!\,Z_{0}\,\bigl(\,\,\,\smashoperator{\bigwedge_{1\leqslant j\leqslant n-1}}\bigl(Z_{j-1} \in Z_{j}\wedge\mu(Z_j)\bigr)\bigr)\wedge\,\exists!\,Z_{n-1} \bigl(Z_{n-1}\in Z_{n}\wedge\mu(Z_n)\bigr)
\]

Since $\underset{n\rightarrow\omega}{\lim}Z_{n}\,=\,Z_{\omega}$, by axiom \ref{15}, corollary \ref{17} and \ref{67}
\begin{align*}
\delta_{\omega}\,&\Longleftrightarrow\,\underset{n\rightarrow\omega}{\lim}\quad\smashoperator{\bigwedge_{1\leqslant j\leqslant n-1}}\,\, \exists!\,Z_{j}\,\exists!\,Z_{j-1}\bigl(Z_{j-1} \in Z_{j}\wedge\mu(Z_j)\bigr)\wedge\underset{n\rightarrow\omega}{\lim}\,\exists!\,Z_{n-1}\bigl(Z_{n-1}\in Z_{n}\wedge\mu(Z_n)\bigr)
\\
\,&\Longleftrightarrow\,\smashoperator{\bigwedge_{n<\omega}}\, \exists!\,Z_{n}\,\exists!\,Z_{n-1}\bigl(Z_{n-1} \in Z_{n}\wedge\mu(Z_n)\bigr) \wedge\,\exists!\,Z_{\omega}\bigl(Z_{\omega}\in Z_{\omega}\wedge \mu(Z_{\omega})\bigr)
\end{align*}
In the above proof, by corollary \ref{24}
\[
\underset{n\rightarrow\omega}{\lim}\mu(Z_n)\Longleftrightarrow\underset{n\rightarrow\omega}{\lim}(\forall z\in G)\bigl(z\in Z_n\bigr)\Longleftrightarrow(\forall z\in G)\Bigl(z\in \underset{n\rightarrow\omega}{\lim}\,Z_n\Bigr)\Longleftrightarrow\mu(Z_{\omega})
\]
By theorem \ref{71}, $Z_{\omega}\vDash\delta_{\omega}$ where $\delta_{\omega}$ defines $\left\langle Z_{n}\colon n\leqslant\omega\right\rangle$, the unique $\in$-sequence of length $\omega$ in
$Z_{\omega}$.\medskip

(iv) \ Fix $k<\omega$, for any $n>k$, $Z_k\in\mathfrak{Z}_n$. By theorem \ref{37}(iv) and axiom \ref{19}, $\underset{n\rightarrow\omega}{\lim}(Z_k\in\mathfrak{Z}_{n})\Leftrightarrow Z_k\in\underset{n\rightarrow\omega}{\lim}\mathfrak{Z}_{n}$. So by corollary \ref{18}, for any $k<\omega$, $Z_k\in\underset{n\rightarrow\omega}{\lim}\mathfrak{Z}_{n}$. By corollary \ref{20}, $Z_{\omega}=\underset{k\rightarrow\omega}{\lim}Z_k\in\underset{n\rightarrow\omega}{\lim}\mathfrak{Z}_{n}$ and so $\mathfrak{Z}^+\subset\underset{n\rightarrow \omega}{\lim}\mathfrak{Z}_{n}$. On the other hand, by theorem \ref{37}(iv), $\underset{n\rightarrow\omega}{\lim}\mathfrak{Z}_{n}$ is unique. Thus $\underset{n\rightarrow \omega}{\lim}\mathfrak{Z}_{n}\subset\mathfrak{Z}^+$ and $\underset{n\rightarrow \omega}{\lim}\mathfrak{Z}_{n}=\mathfrak{Z}^+$. Next by (\ref{SemiInfinitonNType}), corollary \ref{17} and \ref{14}
\begin{align*}
	\lambda_{\omega}\,\Longleftrightarrow\,\underset{n\rightarrow\omega}{\lim}\,\lambda_{n}\,&\,\Longleftrightarrow\,\underset{n\rightarrow\omega}{\lim}\quad\smashoperator{\bigwedge_{1\leqslant i\leqslant n-2}}\,\,\bigl(x_{i}\in x_{i+1}\wedge\mu(x_{i+1})\wedge h(x_{i+1})=h(x_i)+1\bigr)
	\\
	&\qquad\,\,\,\wedge\,\underset{n\rightarrow\omega}{\lim}\bigl(x_{n-1}\in x_{n}\wedge\mu(x_n)\wedge h(x_{n})=h(x_{n-1})+1\bigr)
	\\
	\,&\Longleftrightarrow\underset{n<\omega}{\displaystyle\bigwedge}\bigl(x_{n-1}\in x_{n}\wedge\mu(x_n)\wedge h(x_{n})=h(x_{n-1})+1\bigr)\,\wedge\,\bigl(x_{\omega}\in x_{\omega}\wedge\mu(x_{\omega})\bigr)
\end{align*}
where $\underset{n\rightarrow\omega}{\lim}x_n=x_{\omega}$. Since $\mathfrak{Z}^+\vDash\lambda_{\omega}$, any tuple from $\{Z_n\colon n\leqslant\omega\}$ satisfies $\lambda_{\omega}$. By theorem \ref{37}(iii), Th$(\mathfrak{Z}^+)$ has only one type. So by proposition \ref{25}, $\mathfrak{Z}^+$ is atomic with $\lambda_{\omega}$ being its complete formula. Since any type of Th$(\mathfrak{Z}^+)$ can be derived from it, $\lambda_{\omega}$ is a $\omega$-type with countable free variables.\bigskip
\end{proof}

The tree structure for a semi-infiniton is shown in Figure \ref{Fig2}. Intuitively, a semi-infiniton has one infinite (broken) branch.

\begin{figure}[h]
	\centering
	\includegraphics{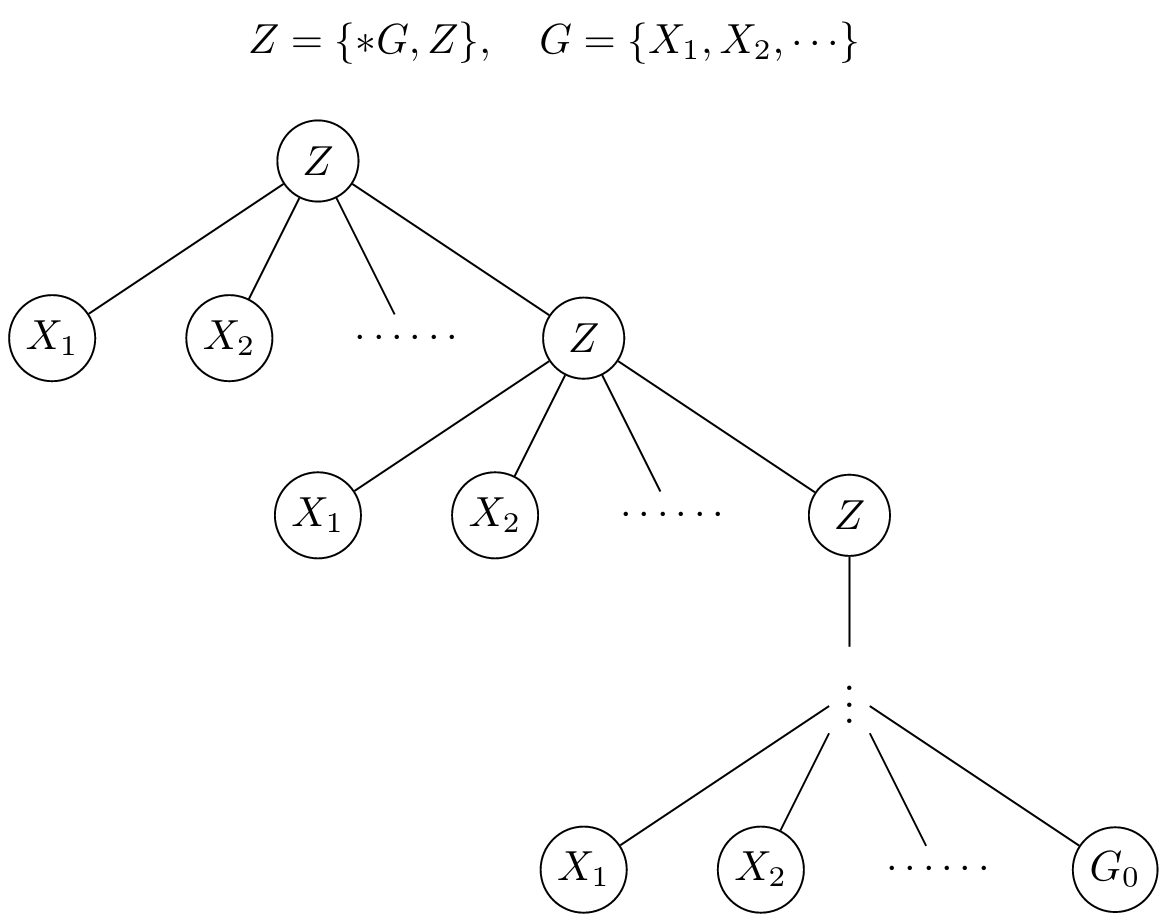}
	\caption{Diagram of a semi-infiniton.}
	\label{Fig2}
	\centering
\end{figure}

\begin{corollary}
\ Suppose $Z$ is a semi-infiniton.
\end{corollary}

\begin{enumerate}
\item $Z\neq\varnothing$

\item $Z\neq\{Z\}$

\item $D(Z)\,=\,\aleph_{0}$

\item $Z\notin V$
\end{enumerate}

\begin{proof}
\ (i) \ Since $\varnothing\notin\varnothing,\,Z\neq\varnothing$.\medskip

(ii) \ If $Z=\{Z\},\,Z$ is an infiniton, contradicting definition \ref{DefSemiInfiniton}.\medskip

(iii) \ If $D(Z)<\aleph_{0}$, then by (\ref{DefMembershipDimension}), $D(Z)<D(\{Z\})=D(Z)$, contradiction.\medskip

(iv) \ Since $Z$ has an infinite branch, it is NWF. By lemma \ref{1}, $Z\notin V$.\smallskip
\end{proof}

\begin{definition}
\label{DefSetSemiInfinitons}$\ S|_{\mathfrak{S}}=\{\{G,G_{0}\}|_{\mathfrak{S}}\colon G_{0},G\in S\}$ is known as the \textbf{set of the semi-infinitons} from $S$.
\end{definition}

\begin{corollary}
\label{39}\qquad
\end{corollary}

\begin{enumerate}
\item $S_{1}\subset S_{2}\,\Longrightarrow\,S_{1}|_{\mathfrak{S}}\subset S_{2}|_{\mathfrak{S}}$

\item $\left(S_{1}\cup S_{2}\right)|_{\mathfrak{S}}\,\supset\, S_{1}|_{\mathfrak{S}}\cup S_{2}|_{\mathfrak{S}}$

\item $\left(S_{1}\cap S_{2}\right)|_{\mathfrak{S}}\,\subset\, S_{1}|_{\mathfrak{S}}\cap S_{2}|_{\mathfrak{S}}$

\item $\Bigl(\underset{\alpha\in D}{\displaystyle\bigcup}S_{\alpha}\Bigr)\Big\vert_{\mathfrak{S}} \supset\underset{\alpha\in D}{\displaystyle\bigcup}S_{\alpha}\vert_{\mathfrak{S}}$

\item $S_{\alpha}\uparrow\,\Longrightarrow\,\Bigl(\underset{\alpha\in D}{\displaystyle\bigcup}S_{\alpha}\Bigr) \Big\vert_{\mathfrak{S}}\,=\underset{\alpha\in D}{\displaystyle\bigcup}S_{\alpha}\vert_{\mathfrak{S}}$
\end{enumerate}

\begin{proof}
\ (i) is obvious. (ii), (iii) and (iv) follow from (i).\medskip

(v) \ For any $\{G,G_{0}\}|_{\mathfrak{S}}\in\Bigl(\underset{\alpha\in D}{\displaystyle\bigcup}S_{\alpha}\Bigr) \Big\vert_{\mathfrak{S}}$, there are $\gamma_{1},\gamma_{2}\in D$ that $G\in S_{\gamma_{1}}$ and $G_{0}\in S_{\gamma_{2}}$. Let $\gamma=\max\{\gamma_{1},\gamma_{2}\}$. Then $G,G_{0}\in S_{\gamma}$ and $\{G,G_{0}\}|_{\mathfrak{S}} \in S_{\gamma}\vert_{\mathfrak{S}}\subset\underset{\alpha\in D}{\displaystyle\bigcup}S_{\alpha} \vert_{\mathfrak{S}}$. This proves $``\subset\textquotedblright$ and it follows by (iv).\bigskip
\end{proof}

The following shows that an infiniton is a special case of a semi-infiniton.

\begin{corollary}
\label{40}\qquad
\end{corollary}

\begin{enumerate}
\item $\{G_{0}\}_{\mathcal{I}}\,=\,\{\varnothing,G_{0}\}|_{\mathfrak{S}}$

\item $\left.S\right\vert_{\mathcal{I}}\,\subset\left.S\right\vert_{\mathfrak{S}}$
\end{enumerate}

\begin{proof}
\ (i) \ Suppose for each $n<\omega,\,I_{n}=\{I_{n-1}\}$, $Z_{n}=\{\ast\varnothing,Z_{n-1}\}$, and $I_{0}=Z_{0}=G_{0}\in V_{\omega}$. By lemma \ref{41}, $I_{n}=Z_{n}$ for all $n$. So by corollary \ref{42}, $\underset{n\rightarrow\omega} {\lim}I_{n}\,=\,\underset{n\rightarrow\omega}{\lim}Z_{n}$, or $\{G_{0}\}_{\mathcal{I}}\,=\, \{\varnothing,G_{0}\}|_{\mathfrak{S}}$. \medskip

(ii) follows from (i).\bigskip
\end{proof}


\begin{corollary}
$\qquad Z\in Z\,\wedge\,Z$ is transitive $\,\Longleftrightarrow\,Z=\{Z\}$
\end{corollary}

\begin{proof}
\ If $Z=\{Z\}$, then $Z\in\{Z\}=Z$. Also, $Z\in Z$ means $\{Z\}\subset Z$. So $Z\subset Z$, i.e. $Z$ is transitive. Conversely, $Z\in Z$ means $\{Z\}\subset Z$. Since $Z$ is transitive, $Z\in\{Z\}$ means $Z\subset\{Z\}$. So $Z=\{Z\}$.
\end{proof}

\subsection{Quasi-Infiniton}\label{SectionQuasiInfiniton}

A quasi-infiniton is a set that contains a vicious cycle, i.e. $Q\in X_{1},\,X_{1}\in X_{2}$, $\dots,\,X_{n-1}\in Q$. From (\ref{QuasiInfiniton0}) and (\ref{InfGeneratedSet}), we can see that an infinitely generated set whose principal generators form a finite cycle is a quasi-infiniton.

\begin{definition}
\label{FiniteCycle}\ $\{G_n\colon G_{n}\in V_{\omega}\,\wedge\,n<\omega\}$ forms a \textbf{finite cycle} if there is a finite subset $\{G_{i}\colon 1\leqslant i\leqslant l\,\wedge\,G_i\neq G_j\,\wedge\,i\neq j\,\wedge\,l>1\}$ such that for each $1\leqslant i\leqslant l$ and $n=kl+i\,(k<\omega)$, $G_{n}=G_{i}$. $l$ is known as the \textbf{length} of the cycle.
\end{definition}

\begin{theorem}
\label{43}\ Suppose for each $n<\omega,\,Q_{n}$ is a finitely generated set and $Q_{n}=\{\ast G_{n},Q_{n-1}\}$ with $Q_{0}\,=\,G_{0}$ where $\{G_{n}\}$ form a finite cycle of length $l$ and $R_V(G_0)\geqslant\max\limits_{1\leqslant i\leqslant l}\{R_V(G_i)\}$\footnote{If $R_V(G_0)< R_V(G_j)$, let $G_0=\{*G_l,\dots\{G_1,G_0\}\dots\}$.}. Let
\begin{align*}
\mathcal{G}\,&=\,\{G_{i}\colon G_{i}\in V_{\omega}\,\wedge\,0\leqslant i\leqslant l\,\wedge\,l>1\,\wedge\,G_i\neq G_j\,\wedge\,i\neq j\,\}
\\
\mathcal{Q}_n\,&=\,\langle\{Q_{n}\},\in,h,G_0,G_1,\dots,G_l\rangle
\\
\mathfrak{Q}_{n_p}\,&=\,\left\langle\allowbreak\left\{ Q_{n_j}\colon n_j\leqslant n_{p}\right\},\in,h,G_0,\dots,G_l \right\rangle
\\
\mathfrak{Q}_q\,&=\,\bigcup_{n_p<\omega}\mathfrak{Q}_{n_{p}}= \left\langle\left\{Q_{{n_{p}(q)}}\colon p<\omega\right\},\in,h, G_0,\dots,G_l\right\rangle
\\
\mathfrak{Q}\,&=\,\bigcup_{0\leqslant q< l}\mathfrak{Q}_{q}=\langle\{\allowbreak Q_{n}\colon n<\omega\},\in,h,G_0,\dots,G_l\rangle
\end{align*}
\end{theorem}

\begin{enumerate}
\item $Q_{n}\,=\,\underset{n}{\underbrace{\{\ast G_n,\{\ast G_{n-1},\dots\{\ast G_1,G_{0}\}\dots\}}}$.

\item \textit{Each $\mathfrak{Q}_q$ has only one complete type and $\mathfrak{Q}$ has $l$ complete types.}

\item $\underset{p\rightarrow\omega}{\lim}\mathcal{Q}_{n_p}$ \textit{and} $\underset{p\rightarrow\omega}{\lim}\mathfrak{Q}_{n_p}$ \textit{is unique.}

\item $\underset{n\rightarrow\omega}{\lim}\mathcal{Q}_n$ \textit{has $l$ sublimits.}
\end{enumerate}

\begin{proof}
\ (i) follows easily by replacing $Q_{n}$ recursively $n$ times.\medskip

(ii) \ First, let $q=0$ and $n_p(0)=pl$. By definition \ref{FiniteCycle} and (i) \smallskip
\begin{align*}
	Q_{l} \,& =\,\{\ast G_{l},\{\ast G_{l-1},\dots\{\ast G_{1},G_{0}\}\dots\}
	\\
	Q_{2l} \,& =\,\{\ast G_{2l},\{\ast G_{2l-1},\dots\{\ast G_{l+1},Q_{l}\}\dots\}
	\\
	\,& =\,\underset{2l}{\underbrace{\{\ast G_{l},\{\ast G_{l-1},\dots\{\ast G_{1},\{\ast G_{l},\{\ast G_{l-1},\dots\{\ast G_{1},G_{0}\}\dots\}}}
	\\
	\,& \,\,\,\vdots
	\\
	Q_{pl} \,& =\,\{\ast G_{pl},\{\ast G_{pl-1},\dots\{\ast G_{(p-1)l+1},Q_{(p-1)l}\}\dots\}\,=\,\cdots
	\\
	\,& =\,\underset{pl}{\underbrace{\{\ast G_{l},\{\ast G_{l-1},\dots\{\ast G_{1},\{\ast G_{l},\{\ast G_{l-1},\dots\{\ast G_{1},G_{0}\}\dots\}\dots\}}}
\end{align*}

Suppose
\[
\theta_{n_{p}(0)}(x)\,\Longleftrightarrow\,\,\,\smashoperator{\bigwedge_{1\leqslant i\leqslant p-1}}\text{\,\,\,\,\,\quad}\smashoperator{\bigwedge_{0\leqslant j\leqslant l-1}}\Phi_{i,j}\,\,\wedge\,\,\,\smashoperator{\bigwedge_{0\leqslant j\leqslant l-2}}\Phi_{p,j}\,\wedge\,\exists!\,y_{pl-1}\bigl(y_{pl-1}\in x\,\wedge \mu_{l}(x)\bigr)
\]
where $\,\mu_j(y)\Longleftrightarrow(\forall z\in G_j)(z\in y)$, and for $\,1\leqslant i\leqslant p\,$ and $\, 0\leqslant j\leqslant l-1$
\[
\Phi_{i,j}\,\Longleftrightarrow\,\exists!\,y_{(i-1)l+j+1}\,\exists!\,y_{(i-1)l+j}\,\bigl(y_{(i-1)l+j}\in y_{(i-1)l+j+1}\wedge\mu_{j+1}(y_{(i-1)l+j+1})\bigr)
\]
The validity of $\theta_{n_{p}(0)}(x)$ means that there is a unique $\in$-sequence of length $pl$ in $x$ with $p$ levels of cycle of length $l$. Since $\mathcal{Q}_{pl}\vDash \theta_{n_{p}(0)}[{Q}_{pl}]$, $(\theta_{n_{p}(0)}(x))$ is a $1$-type of $\mathfrak{Q}_0$.\smallskip

Now let $n_{p}(q)=pl+q\,\,(0< q<l)$ and
\begin{equation}
\theta_{n_{p}(q)}(x)\,\Longleftrightarrow\,\,\,\smashoperator{\bigwedge_{1\leqslant i\leqslant p-1}}\text{\,\,\qquad}\smashoperator{\bigwedge_{q\leqslant j\leqslant q+l-1}}\Phi_{i,j}\,\,\wedge\quad\smashoperator{\bigwedge_{q\leqslant j\leqslant q+l-2}}\Phi_{p,j}\,\wedge\,\exists!\,y_{pl-1+q}\bigl(y_{pl-1+q}\in x\,\wedge \mu_{q}(x)\bigr)\label{QuasiInfinitonOneType}
\end{equation}

In the above conjunctions, $j=j-l+1$ if $j\geqslant l$. Clearly
\begin{align*}
	Q_{l+q} \,& =\,\underset{l+q}{\underbrace{\overbrace{\{\ast G_{q},\{\ast G_{q-1},\dots\{\ast G_{1},\{\ast G_{l},\{\ast G_{l-1},\dots\{\ast G_{q+1},}^{l}\{\ast G_{q},\dots\{\ast G_{1},G_{0}\}\dots\}}}
	\\
	\,& \,\,\,\vdots
	\\
	Q_{pl+q} \,&=\,\underset{pl+q}{\underbrace{\overbrace{\{\ast G_{q},\{\ast G_{q-1},\dots\{\ast G_{1},\{\ast G_{l},\dots\{\ast G_{q+1},}^{l}\overbrace{\{\ast G_{q},\dots\{\ast G_{1},\{\ast G_{l},\dots\{\ast G_{1},G_{0}\}\dots\}\dots\}}^{(p-1)l+q}}}
\end{align*}
So $\mathcal{Q}_{pl+q}\vDash \theta_{n_{p}(q)}[{Q}_{pl+q}]$. For any $r\neq q\,\,(0\leqslant r<l)$, $\mathcal{Q}_{pl+r}\nvDash \mu_q[{Q}_{pl+r}]$ and so $\mathcal{Q}_{pl+r}\nvDash \theta_{n_{p}(q)}[{Q}_{pl+r}]$. Since $\mathcal{Q}_{pl+r}\vDash \theta_{n_{p}(r)}[{Q}_{pl+r}]$, $\theta_{n_{p}(q)}$ and $\theta_{n_{p}(r)}$ are inconsistent. Thus $(\theta_{n_{p}(q)}(x))$ is a distinct $1$-type of $\mathfrak{Q}_q$ in $\mathfrak{Q}$.\smallskip

For $p$-type of $\mathfrak{Q}_q$, suppose
\begin{equation}
	\chi_{n_p(q)}(x_1,\dots,x_{p})\,\Longleftrightarrow\,\,\,\smashoperator{\bigwedge_{1\leqslant i\leqslant p-1}}\,\bigl(\xi(x_{i},x_{i+1})\wedge\mu_q(x_{i+1})\wedge \mu_q(x_{i})\wedge h(x_{i+1})=h(x_i)+l\bigr)\label{QuasiInfinitonNType}
\end{equation}

where $h(x)$ is the height of $x$ and
\[
\xi(x,y)\,\Longleftrightarrow\,\exists!z_{l-1}\,\cdots\, \exists!z_{1}\,\bigl(\,\,\,\smashoperator{\bigwedge_{2\leqslant i\leqslant l-1}}\, (z_{i-1}\in z_{i})\wedge(x\in z_1\wedge z_{l-1}\in y)\bigr)
\]
The validity of $\xi(x,y)$ means that there is a unique $\in$-sequence of length $l$ starting in $y$ and ending in $x$. Since  $\mathfrak{Q}_{n_{p}(q)}\vDash\chi_{n_{p}(q)}\left[Q_{n_{1}(q)},Q_{n_{2}(q)},\dots,Q_{n_{p}(q)}\right]$, $\bigl(\chi_{n_{p}(q)}(x_1,\dots,x_{p})\bigr)$ is a $p$-type of $\mathfrak{Q}_q$. Clearly
\[
\bigl\{\,\theta_{n_{p}(q)}(x_1),\theta_{n_{p+1}(q)}(x_1),\dots,\chi_{2l+q}( x_1,x_2),\dots, \chi_{pl+q}(x_1,x_2,\dots,x_{p})\dots\bigr\}
\] 
generates a maximum consistent set of formulas involving $\mathcal{G}$, which is the only complete type of $\mathfrak{Q}_q$.\smallskip

For any $r\neq q\,\,(0\leqslant r<l)$, $\mathcal{Q}_{n_{p}(r)}\nvDash \mu_q[{Q}_{n_{p}(r)}]$ and so $\mathfrak{Q}_{n_{p}(r)}\nvDash\chi_{n_{p}(q)}\left[Q_{n_{1}(r)},\dots,Q_{n_{p}(r)}\right]$. Thus $\chi_{n_{p}(q)}$ and $\chi_{n_{p}(r)}$ are inconsistent and belong to different types. Consequently, $\mathfrak{Q}$ has $l$ complete types. By proposition \ref{25}, Th$(\mathfrak{Q})$ is $\aleph_{0}$-categorical.\medskip

(iii) \ Let $n_{p}(q)=pl+q\,\,(0\leqslant q<l)$. By (ii), for any $p<\omega$, $\mathcal{Q}_{n_p} \vDash \theta_{n_{p}(q)}[Q_{n_p(q)}]$, and for any $k>p,\,\mathcal{Q}_{n_k}\vDash \theta_{n_{p}(q)}[Q_{n_{k}(q)}]$. So $\left(\mathcal{Q}_{n_p},\theta_{n_{p}(q)}\right)$ is a homogeneous subsequence. By (ii) and definition \ref{DefSubLimit}, $\underset{p\rightarrow\omega}{\lim}\mathcal{Q}_{n_{p}}$ is unique. Likewise, since for any $k>p$, $\mathfrak{Q}_{n_{k}}\vDash \chi_{n_{p}(q)}[Q_{n_{k-p+1}(q)},\dots,Q_{n_{k}(q)}]$, $\left(\mathfrak{Q}_{n_{p}},\chi_{n_{p}(q)}\right)$ is a homogeneous subsequence. Thus $\underset{p\rightarrow\omega}{\lim}\mathfrak{Q}_{n_{p}}$ is unique.\medskip

(iv) \ By (iii), for each $0\leqslant q<l$, there is a unique  $\underset{p\rightarrow\omega}{\lim}\mathcal{Q}_{n_{p}}$. So there are total $l$ sublimits in $\mathcal{Q}_n$.
\end{proof}

\begin{definition}
\label{DefQuasiInfiniton} \ In theorem \ref{43}, each $\underset{p\rightarrow\omega}{\lim}\mathcal{Q}_{n_{p}}$ is known as a \textbf{quasi-infiniton} generated by $\mathcal{G}$ and is denoted as $\left(0\leqslant q<l\right)$:
\[
\underset{p\rightarrow\omega}{\lim}\mathcal{Q}_{n_{p}(q)}\,=\,\underset{p\rightarrow\omega}{\lim}{Q}_{n_{p}(q)}\,=\,Q_{\omega,q}\,=\,\{G_{k},l\}|_{\mathfrak{Q}}
\]
The collection of the sublimits of $\mathcal{Q}_{n}$ is denoted as:
\[
\underset{n\rightarrow\omega}{\lim}\mathcal{Q}_{n}\,=\,\underset{n\rightarrow\omega}{\lim}Q_{n}\,=\,Q_{\omega}\,=\,\{Q_{\omega,q}\colon0\leqslant q<l\}
\]
Where $G_{k}\,\left(1\leqslant k\leqslant l\right)$ are \textbf{principal generators} and $G_{0}$ is a \textbf{base generator} of $Q_{\omega}$. $\bar{\mathscr{L}}^{\prime}=\{\allowbreak\in, H_{\omega}, I_{\omega},Z_{\omega},Q_{\omega},Q_{\omega,q}\}$ after new constants are added.\footnote{In the rest discussion, we will no longer distinguish $\underset{n\rightarrow\omega}{\lim}\mathcal{Q}_{n}$ and $\underset{n\rightarrow\omega}{\lim}Q_{n}$.}
\end{definition}

\begin{theorem}
\label{44} \ Suppose everything is the same as in theorem \ref{43}, $\mathfrak{Q}_{q}^+=\langle\{Q_{n_p(q)}\colon p<\omega\}\cup\{Q_{\omega,q}\}, \in,h,G_0,\dots, G_l\rangle$ and $\,\mathfrak{Q}^+=\,\bigcup_{0\leqslant q< l}\mathfrak{Q}_{q}^+$. For any $0\leqslant q<l$
\end{theorem}

\begin{enumerate}
\item $Q_{\omega,q}\,=\,\{\ast G_{q},\{\ast G_{q-1},\dots\{\ast G_{q+1},Q_{\omega,q}\}\dots\}$

\item \textit{For any} $q<r<l$, $\,Q_{\omega,r}=\{\ast G_{r},\dots\{\ast G_{q+1},Q_{\omega,q}\}$.

\item \textit{Each $Q_{\omega,q}$ of $Q_{\omega}$ is $\omega$-invariant.}

\item \textit{A type for each $Q_{\omega,q}$ of $Q_{\omega}$ is that there exists a unique $\in$-sequence of length $\omega$ with each sublevel $i\,(i\leqslant \omega)$ containing the members of $G_{i}$ and a vicious cycle of length $l$.}

\item \textit{For any} $0\leqslant q<l$, $\underset{p\rightarrow\omega}{\lim}\mathfrak{Q}_{n_p}=\mathfrak{Q}_q^+$ and $\mathfrak{Q}_q^+$ \textit{is atomic}.

\item $\mathfrak{Q}^+$\textit{\ is a }$\aleph_{0}$\textit{-categorical structure with $l$ atomic substructures.}
\end{enumerate}

\begin{proof}
(i) \ Let $n_p(q)=pl+q$. Clearly we have
\begin{align*}
	Q_{q} \,& =\,\{\ast G_{q},\{\ast G_{q-1},\dots\{\ast G_{1},G_{0}\}\dots\}
	\\
	Q_{l+q} \,& =\,\{\ast G_{q},\{\ast G_{q-1},\dots\{\ast G_{1},\{\ast G_{l},\{\ast G_{l-1},\dots\{\ast G_{q+1},Q_q\}\dots\}
	\\
	\,& \,\,\, \vdots
	\\
	Q_{pl+q} \,&=\, \{\ast G_{q},\{\ast G_{q-1},\dots\{\ast G_{1},\{\ast G_{l},\dots\{\ast G_{q+1},Q_{(p-1)l+q}\}\dots\}
\end{align*}

By theorem \ref{43}(iii) and definition \ref{DefQuasiInfiniton}, $\underset{p\rightarrow\omega}{\lim}Q_{n_{p}(q)}=Q_{\omega,q}$. So by corollary \ref{42} and \ref{29}
\begin{align*}
Q_{\omega,q}\,=\,\underset{p\rightarrow\omega}{\lim}Q_{n_{p}(q)}\,&=\,\{\ast G_{q},\{\ast G_{q-1},\dots\{\ast G_{1},\{\ast G_{l},\dots\{\ast G_{q+1},\underset{p\rightarrow\omega}{\lim}Q_{n_{p-1}(q)}\}\dots\}
\\
\,&=\,\{\ast G_{q},\{\ast G_{q-1},\dots\{\ast G_{1},\{\ast G_{l},\dots\{\ast G_{q+1},Q_{\omega,q}\}\dots\}
\end{align*}

(ii) \ By (i), we have
\[
	Q_{pl+r} \,=\, \{\ast G_{r},\{\ast G_{r-1},\dots\{\ast G_{q+1},Q_{pl+q}\}\dots\}
\]

So by corollary \ref{29} and theorem \ref{43}
\begin{align*}
	Q_{\omega,r}\,=\,\underset{p\rightarrow\omega}{\lim}Q_{n_p(r)}
	\,&=\,\{\ast G_{r},\{\ast G_{r-1},\dots\{\ast G_{q+1},\underset{p\rightarrow\omega}{\lim}Q_{n_p(q)}\}\dots\}
	\\
	\,&=\, \{\ast G_{r},\{\ast G_{r-1},\dots\{\ast G_{q+1},Q_{\omega,q}\}\dots\}
\end{align*}

(iii) \ For any $\alpha$ that $\omega<\alpha<\omega2$, if $(m-1)l+q\leqslant\alpha-\omega<ml+q$, let $\beta\,=\,\omega+ml+q$. Then by (i)
\begin{align*}
Q_{\beta} \:& =\: \{\ast G_{q},\{\ast G_{q-1},\dots\{\ast G_{q+1},Q_{(m-1)l+q}\}\dots\}
\\
\:& =\: \{\ast G_{q},\{\ast G_{q-1},\dots\{\ast G_{1},\{\ast G_{l},\{\ast G_{l-1},\dots\{\ast G_{q+1},Q_{(m-2)l+q}\}\dots\}
\\
\,&\:\: \vdots
\\
\:& =\: \underset{\beta-\omega}{\underbrace{\{\ast G_{q},\dots\{\ast G_{1},\{\ast G_{l},\dots\{\ast G_{q+1},Q_{\omega,q}\}\dots\}\dots\}}}
\\
\:& =\: Q_{\omega,q}
\end{align*}

Then it follows by transfinite induction and definition \ref{DefOmegaInvariant}.\medskip

(iv) \ Suppose $n_{p}(q)=pl+q\,\,(0\leqslant q<l)$ and $\left\langle Q_{i}\colon i\leqslant n_p\right\rangle$ is the unique $\in$-sequence of length $n_p$ in $Q_{{n}_p(q)}$ with each of its sublevel $i\,(i<n_p)$ containing the members of $G_i$. Then by (\ref{QuasiInfinitonOneType})
\begin{equation*}
	\theta_{n_{p}(q)}(Q_{pl+q})\,\Longleftrightarrow\,\,\,\smashoperator{\bigwedge_{1\leqslant i\leqslant p-1}}\text{\,\,\qquad}\smashoperator{\bigwedge_{q\leqslant j\leqslant q+l-1}}\Phi_{i,j}\,\,\wedge\quad\smashoperator{\bigwedge_{q\leqslant j\leqslant q+l-2}}\Phi_{p,j}\,\wedge\,\exists!\,Q_{pl-1+q}\,\bigl(Q_{pl-1+q}\in Q_{pl+q}\,\wedge \mu_{q}(Q_{pl+q})\bigr)
\end{equation*}
where for $\,1\leqslant i\leqslant p\,$ and $\, 0\leqslant j\leqslant l-1$
\[
\Phi_{i,j}\,\Longleftrightarrow\,\exists!\,Q_{(i-1)l+j+1}\,\exists!\,Q_{(i-1)l+j}\,\bigl(Q_{(i-1)l+j}\in Q_{(i-1)l+j+1}\,\wedge\,\mu_{j+1}(Q_{(i-1)l+j+1})\bigr)
\]
(In the above conjunctions, $j=j-l+1$ if $j\geqslant l$.) And for any $i=j+1$
\[
\mu_{i}(y)\Longleftrightarrow(\forall z\in G_{i})(z\in y)
\]
By definition \ref{DefQuasiInfiniton}, for any $\, 0\leqslant j\leqslant l-1$, $\underset{p\rightarrow\omega}{\lim}Q_{n_{p}(j)}=Q_{\omega,j}$. So by axiom \ref{15}, corollary \ref{17} and \ref{67}
\begin{align*}
\theta_{\omega,q} \,\,& \Longleftrightarrow\,\,\underset{p\rightarrow\omega}{\lim}\theta_{n_{p}(q)}
\\
\,\,&\Longleftrightarrow\,\underset{p\rightarrow\omega}{\lim}\,\,\,\smashoperator{\bigwedge_{1\leqslant i\leqslant p-1}}\text{\,\,\qquad}\smashoperator{\bigwedge_{q\leqslant j\leqslant q+l-1}}\Phi_{i,j}\,\,\wedge\,\underset{p\rightarrow\omega}{\lim}\,\,\,\,\smashoperator{\bigwedge_{q\leqslant j\leqslant q+l-2}}\Phi_{p,j}\, \wedge\,\underset{p\rightarrow\omega}{\lim}\,\exists!\,Q_{pl-1+q}\,\bigl(Q_{pl-1+q}\in Q_{pl+q}\wedge \mu_{q}(Q_{pl+q})\bigr)
\\
\,\,&\Longleftrightarrow\,\,\smashoperator{\bigwedge_{p<\omega}}\text{\,\,\,\quad}\smashoperator{\bigwedge_{q\leqslant j\leqslant q+l-1}}\,\,\,\exists!\,Q_{(p-1)l+j+1}\,\exists!\,Q_{(p-1)l+j}\,\bigl(Q_{(p-1)l+j}\in Q_{(p-1)l+j+1}\wedge\mu_{j+1}(Q_{(p-1)l+j+1})\bigr)
\\
\,\,& \qquad \wedge\,\,\,\smashoperator{\bigwedge_{q\leqslant j\leqslant q+l-1}}\,\,\,\exists!\,Q_{\omega,j+1}\,\exists!\, Q_{\omega,j}\,\bigl(Q_{\omega,j}\in Q_{\omega,j+1}\wedge\mu_{j+1}(Q_{\omega,j+1})\bigr)
\end{align*}

In the above proof, for any $i=j+1$, by corollary \ref{24}
\[
\underset{p\rightarrow\omega}{\lim}\mu_i(Q_{(p-1)l+i})\Leftrightarrow\underset{p\rightarrow\omega}{\lim}(\forall z\in G_{i})\bigl(z\in Q_{(p-1)l+i}\bigr)\Leftrightarrow(\forall z\in G_i)\Bigl(z\in \underset{p\rightarrow\omega}{\lim}Q_{n_{p-1}(i)}\Bigr)\Leftrightarrow\mu_i(Q_{\omega,i})
\]
By corollary \ref{16}, $Q_{\omega,q}\vDash\theta_{\omega,q}$ where $\theta_{\omega,q}$ defines the unique $\in$-sequence of length $\omega$ in $Q_{\omega,q}$ with each sublevel $i\,(i\leqslant \omega)$ containing the members of $G_{i}$ and a vicious cycle of length $l$.\medskip

(v) \ Let $n_{p}(q)=pl+q\,\,(0\leqslant q<l)$ and fix $n_{k}<\omega$. Then for any $p>k$, $Q_{n_{k}}\in\mathfrak{Q}_{n_{p}}$. By theorem \ref{43}(iii) and axiom \ref{19}, $\underset{{p}\rightarrow\omega}{\lim}(Q_{n_{k}}\in\mathfrak{Q}_{n_{p}})\Leftrightarrow Q_{n_{k}}\in\underset{{p}\rightarrow\omega}{\lim}\mathfrak{Q}_{n_{p}}$. So by corollary \ref{18}, for any $k<\omega$, $Q_{n_{k}}\in\underset{{p}\rightarrow\omega}{\lim}\mathfrak{Q}_{n_{p}}$. By corollary \ref{20}, $Q_{\omega,q}=\underset{k\rightarrow\omega}{\lim}Q_{n_{k}}\in\underset{{p}\rightarrow\omega}{\lim}\mathfrak{Q}_{n_{p}}$ and so $\mathfrak{Q}_q^+\subset\underset{{p}\rightarrow \omega}{\lim}\mathfrak{Q}_{n_{p}}$. On the other hand, by theorem \ref{43}(iii), $\underset{{p}\rightarrow\omega}{\lim}\mathfrak{Q}_{n_{p}}$ is unique. Thus $\underset{{p}\rightarrow \omega}{\lim}\mathfrak{Q}_{n_{p}}\subset\mathfrak{Q}_q^+$ and $\underset{{p}\rightarrow \omega}{\lim}\mathfrak{Q}_{n_{p}}=\mathfrak{Q}_q^+$.\smallskip

 Furthermore, by (\ref{QuasiInfinitonNType}), corollary \ref{17} and \ref{14} 
\begin{align*}
\chi_{\omega,q}\,&\,\Longleftrightarrow\,\underset{{p}\rightarrow\omega}{\lim}\,\chi_{n_p(q)}
\\
\,&\,\Longleftrightarrow\,\underset{{p}\rightarrow\omega}{\lim}\quad\smashoperator{\bigwedge_{1\leqslant i\leqslant p-2}}\,\bigl(\xi(x_{n_{i}},x_{n_{i+1}})\wedge\mu_q(x_{n_{i+1}})\wedge\mu_q(x_{n_{i}})\wedge h(x_{n_{i+1}})=h(x_{n_{i}})+l\bigr)
\\
&\qquad\,\,\,\wedge\,\underset{{p}\rightarrow\omega}{\lim}\bigl(\xi(x_{n_{p-1}},x_{n_{p}})\wedge\mu_q(x_{n_{p}})\wedge\mu_q(x_{n_{p-1}})\wedge h(x_{n_{p}})=h(x_{n_{p-1}})+l\bigr)
\\
\,&\,\Longleftrightarrow\,\smashoperator{\bigwedge_{p<\omega}}\, \bigl(\xi(x_{n_{p}},x_{n_{p+1}})\wedge\mu_q(x_{n_{p+1}})\wedge\mu_q(x_{n_{p}})\wedge h(x_{n_{p+1}})=h(x_{n_{p}})+l\bigr)
\\
\,\,& \qquad\, \wedge\,\,\,\,\smashoperator{\bigwedge_{q\leqslant j\leqslant q+l-1}}\,\,\bigl(x_{\omega,j}\in x_{\omega,j+1}\,\wedge\,\mu_{j+1}(x_{\omega,j+1})\bigr)
\end{align*}
where $h$ is the height function, $\underset{p\rightarrow\omega}{\lim}x_{n_{p}(j)}=x_{\omega,j}$ and
\[
\xi(x_{n_{i}},x_{n_{i+1}})\Longleftrightarrow\exists!z_{l-1}\cdots \exists!z_{1}\,\bigl(\,\,\,\smashoperator{\bigwedge_{2\leqslant j\leqslant l-1}} \bigl(z_{j-1}\in z_{j}\wedge\mu_{q+j}(z_{j})\bigr)\wedge\bigl(x_{n_{i}}\in z_1\wedge\mu_{q+1}(z_{1})\wedge z_{l-1}\in x_{n_{i+1}}\bigr)\bigr)
\]
Since $\mathfrak{Q}_q^+\vDash\chi_{\omega,q}$, any tuple from $\{Q_{n_p}\colon {p}\leqslant\omega\}$ satisfies $\chi_{\omega,q}$. By theorem \ref{43}(ii), Th$(\mathfrak{Q}_q^+)$ has only one type. So by proposition \ref{25}, $\mathfrak{Q}_q^+$ is atomic with $\chi_{\omega,q}$ being its complete formula. Since any type of Th$(\mathfrak{Q}_q^+)$ can be derived from it, $\chi_{\omega,q}$ is a $\omega$-type with countable free variables.\medskip

(vi) \ By theorem \ref{43}(ii) and (v), $\mathfrak{Q}^+$ has $l$ types and so Th$(\mathfrak{Q}^+)$ is $\aleph_{0}$-categorical. By proposition \ref{25} and (v), $\mathfrak{Q}^+$ has $l$ atomic substructures of $\mathfrak{Q}_q^+$.\bigskip 
\end{proof}


The tree structure for a quasi-infiniton is shown in Figure \ref{Fig3}. Intuitively, a quasi-infiniton has one infinite branch and the nodes of the infinite branch form a finite cycle.

\begin{figure}[h]
	\centering
	\includegraphics{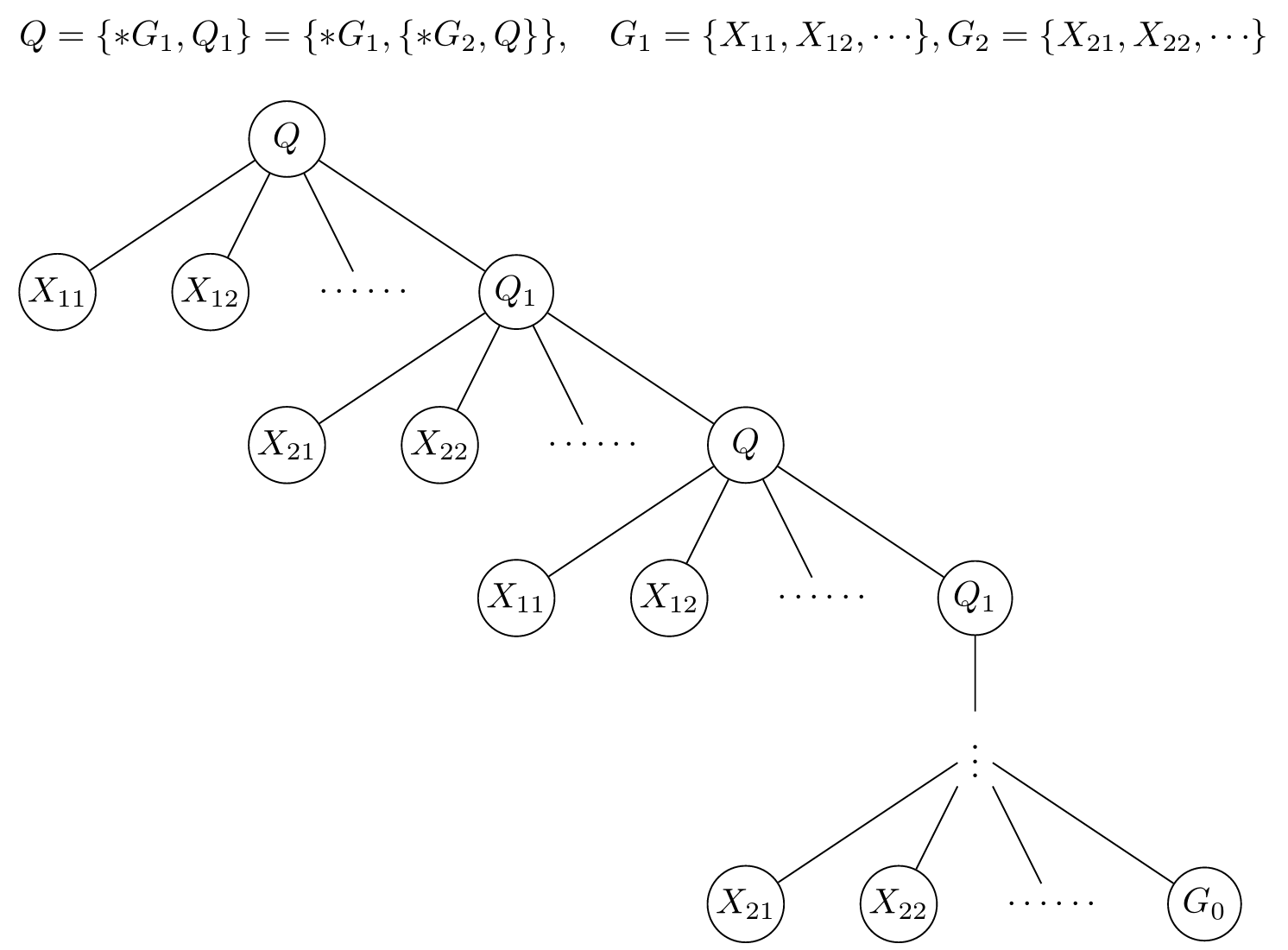}
	\caption{Diagram of a quasi-infiniton.}
	\label{Fig3}
	\centering
\end{figure}

\begin{corollary}
$\qquad$
\end{corollary}

\begin{enumerate}
\item $Q\notin Q$ \textit{and} $Q\neq\varnothing$.

\item $D(Q)\,=\,\aleph_{0}$ \textit{and} $Q\notin V$.
\end{enumerate}

\begin{proof}
\ (i) \ If $Q\in Q$, the length of $Q$ is $1$ and $Q$ is a semi-infiniton. Also, no $Q_{1}\in\varnothing$. So $\varnothing\in Q_{1}$ and $Q_{1}\in\varnothing$ are impossible.\medskip

(ii) \ By theorem \ref{44}(iv) and (\ref{DefMembershipDimension}), $D(Q)\,=\,\aleph_{0}$. Since $Q$ has an infinite branch, it is NWF. So $Q\notin V$.
\end{proof}

\begin{definition}
\label{DefSetQuasiInfiniton}\ $S|_{\mathfrak{Q}}=\{\{G_{k},l\}|_{\mathfrak{Q}}\colon G_{k}\in S,\,0\leqslant k\leqslant l,\,l>1\}$ is known as the \textbf{set of the quasi-infinitons} from $S$.
\end{definition}

\begin{corollary}
\label{45}\qquad
\end{corollary}

\begin{enumerate}
\item $S_{1}\subset S_{2}\:\Longrightarrow\:S_{1}|_{\mathfrak{Q}}\subset S_{2}|_{\mathfrak{Q}}$

\item $\left(S_{1}\cup S_{2}\right)|_{\mathfrak{Q}}\,\supset\,S_{1}|_{\mathfrak{Q}}\cup S_{2}|_{\mathfrak{Q}}$

\item $\left(S_{1}\cap S_{2}\right)|_{\mathfrak{Q}}\,\subset\,S_{1}|_{\mathfrak{Q}}\cap S_{2}|_{\mathfrak{Q}}$

\item $\Bigl(\underset{\alpha\in D}{\displaystyle\bigcup}S_{\alpha}\Bigr)\Big\vert_{\mathfrak{Q}}\, \supset\underset{\alpha\in D}{\displaystyle\bigcup}S_{\alpha}\vert_{\mathfrak{Q}}$

\item $S_{\alpha}\uparrow\hspace{1ex}\Longrightarrow\Bigl(\underset{\alpha\in D}{\displaystyle\bigcup}S_{\alpha} \Bigr)\Big\vert_{\mathfrak{Q}}\,=\underset{\alpha\in D}{\displaystyle\bigcup}S_{\alpha}\vert_{\mathfrak{Q}}$
\end{enumerate}

\begin{proof}
\ (i) is obvious. (ii), (iii) and (iv) follow from (i).\medskip

(v) \ For any $\{G_{k},l\}|_{\mathfrak{Q}}\in\Bigl(\underset{\alpha\in D}{\displaystyle\bigcup}S_{\alpha}\Bigr) \Big\vert_{\mathfrak{Q}}$, there is a $\gamma\in D$ that $G_{k}\in S_{\gamma}$ for $0\leqslant k\leqslant l$. Thus $\{G_{k},l\}|_{\mathfrak{Q}}\in S_{\gamma}\vert_{\mathfrak{Q}}\subset\underset{\alpha\in D}{\displaystyle\bigcup} S_{\alpha}\vert_{\mathfrak{Q}}$. This proves $``\subset
\textquotedblright$ and it follows by (iv).\smallskip
\end{proof}

\section{Total Universe}

In this section, we will present and investigate a hierarchy for combining the well-founded sets with the non-well-founded sets known as the total universe. We will also show that the total universe is free of Russell's paradox.

\subsection{Definitions}

First, we need to generalize the $\omega$-neighborhood and the limit of formulas. 

\begin{definition}
\label{DefNeighborhoodLimitOrdinals} \ Suppose $\alpha$ is a limit ordinal $\left(\alpha>\omega\right)$ and $\alpha_{0}$ the limit ordinal immediately below $\alpha$. The \textbf{cofinite topology} on $\alpha$ is defined as: $\,\mathfrak{T}\,=\, \{y\subset\alpha\colon y=\varnothing\,\vee\,(\alpha_{0}\subset y\,\wedge\,\alpha-y$ is finite$)\}$. A \textbf{neighborhood of} $\mathbf{\alpha}$ $(\alpha$-neighborhood$)$ $\mathfrak{H}$ is a member of $\mathfrak{T}$.
\end{definition}

\begin{lemma}
\ $\mathfrak{H}$ is a neighborhood of $\alpha$ if and only if $\exists\beta\in\alpha-\alpha_{0}$ that $\forall\gamma\left(  \beta<\gamma<\alpha\right)\Rightarrow\gamma\in\mathfrak{H}$.
\end{lemma}

\begin{proof}
\ Suppose $\mathfrak{H}$ is a neighborhood of $\alpha$ and for any $\beta\in\alpha-\alpha_{0}$, there is a $\gamma$ that $\beta<\gamma<\alpha$ and $\gamma\notin\mathfrak{H}$. Then $\alpha-\mathfrak{H}$ is not finite, contradicting definition \ref{DefNeighborhoodLimitOrdinals}. On the other hand, if there is a $\beta\in\alpha-\alpha_{0}$ such that for any $\gamma$ of $\beta<\gamma<\alpha$, $\gamma\in\mathfrak{H}$, then $\alpha-\mathfrak{H}$ is finite and $\mathfrak{H}\in\mathfrak{T}$.
\end{proof}

\begin{definition}
	\label{DefGeneralHomoSeq} \ Suppose $\mathscr{L}$ is an infinitary language of $\mathscr{L}_{\omega_{1},\omega}$, $T$ is a $\aleph_{0}$-categorical theory of $\mathscr{L}$ and $\alpha$ is a limit ordinal $\left(\alpha>\omega\right)$. Let $\phi_{\gamma}$ be types in $T$ and $\mathfrak{M}_{\gamma}$ be $\mathscr{L}$-structures that $\mathfrak{M}_{\gamma}\vDash\phi_{\gamma}$. If there exists a $\alpha$-neighborhood $\mathfrak{H}$ that for any $\beta,\gamma\in\mathfrak{H}\:(\beta>\gamma)$, $\mathfrak{M}_{\beta}\vDash\phi_{\gamma}$, then $\{(\mathfrak{M}_{\gamma}, \phi_{\gamma})\colon\mathfrak{M}_{\gamma}\vDash\phi_{\gamma}\land \gamma<\alpha\}$ is known as a \textbf{homogeneous sequence of structures defined by $\phi_{\gamma}$} in $T$.
\end{definition}

\begin{definition}
	\label{DefGeneralLimit} \ Suppose $\{(\mathfrak{M}_{\gamma},\phi_{\gamma})\colon \mathfrak{M}_{\gamma} \vDash\phi_{\gamma}\land \gamma<\alpha\}$ is a homogeneous sequence of structures in a $\aleph_{0}$-categorical theory. The unique countable atomic structure $\mathfrak{M}$ (up to isomorphism) in $\{(\mathfrak{M}_{\gamma},\phi_{\gamma})\}$ is known as the \textbf{limit} of $\mathfrak{M}_{\gamma}$ and is denoted as $\underset{\gamma\rightarrow\alpha}{\lim} \mathfrak{M}_{\gamma}=\mathfrak{M}$. The unique formula $\phi$ (up to equivalence) is known as the \textbf{limit} of $\phi_{\gamma}$ and is denoted as $\underset{\gamma\rightarrow\alpha}{\lim}\phi_{\gamma}=\phi$. In both cases, we also say that the limit of $\phi_{\gamma}$ or the limit of $\mathfrak{M}_{\gamma}$ is unique.
\end{definition}

\begin{definition}
	\label{DefGeneralSubLimit} \ Suppose in a sequence of structures $\{(\mathfrak{M}_{\gamma},\phi_{\gamma})\colon \mathfrak{M}_{\gamma} \vDash\phi_{\gamma}\land \gamma<\alpha\}$ in a $\aleph_{0}$-categorical theory, there are finitely many homogeneous subsequences of structures $\{(\mathfrak{M}_{\gamma_i}, \phi_{\gamma_i})\colon \allowbreak \mathfrak{M}_{\gamma_i} \vDash\phi_{\gamma_i}\land \gamma_i<\alpha\}$. Then each $\underset{\gamma_i\rightarrow\alpha}{\lim}\,\mathfrak{M}_{\gamma_i}$ is known as a \textbf{sublimit} of $\mathfrak{M}_{\gamma}$, and  each $\underset{\gamma_i\rightarrow\alpha}{\lim}\,\phi_{\gamma_i}$ is known as a \textbf{sublimit} of $\phi_{\gamma}$. If some sublimits of $\mathfrak{M}_{\gamma}/\phi_{\gamma}$ are different, we say $\underset{\gamma\rightarrow\alpha}{\lim}\mathfrak{M}_{\gamma}/ \underset{\gamma\rightarrow\alpha}{\lim}\phi_{\gamma}$ exist (but not unique).
\end{definition}

Most conclusions in section \ref{SectionOfLimitFormula} hold for the limit ordinals as well. We can simply replace $n$ with $\gamma$ ($\gamma$ is any successor ordinal in a neighborhood of a limit ordinal $\alpha$ above $\omega$), and $\underset{n\rightarrow\omega} {\lim}$ with $\underset{\gamma\rightarrow\alpha}{\lim}$. Now we define the total universe based upon the von Neumann universe.

\begin{definition}
\ The \textbf{total universe} is:\footnote{$\mathrm{SOrd}$ is the set of all successor ordinals and $\mathrm{LOrd}$ is the set of all limit ordinals.}
\begin{align}
T_{0} \, & =\,\varnothing,  \nonumber  
\\
T_{\alpha} \, & =\,\mathcal{P}(T_{\alpha-1}),\qquad\qquad\qquad\qquad\qquad\quad \alpha\in\mathrm{SOrd} \nonumber  
\\
T_{\alpha} \,& =\underset{\beta<\alpha}{\bigcup} T_{\beta}\,\cup\Bigl(\underset{\beta<\alpha}{\bigcup} T_{\beta}\Bigr)\Big \vert_{\aleph_{0}},\qquad\qquad\qquad \alpha\in\mathrm{LOrd}  \nonumber
\\
T \, & =\,\,\smashoperator{\bigcup_{\alpha\in \mathrm{Ord}}}\,T_{\alpha}.\label{TotalHierarchy}
\end{align}
\end{definition}

\begin{remark}
Note that (\ref{TotalHierarchy}) is based on the generalization of definition \ref{DefInfGeneratedSet} and \ref{DefSetOfIGS}. Since the total universe contains the well-founded sets, it is similar to (\ref{CumulativeHierarchy}). The key difference is that the infinitely generated sets are created at each limit ordinal in addition to the von Neumann universe.
\end{remark}

\begin{definition}
\label{DefGenInfGeneratedSet} \ Suppose $\alpha$ is a limit ordinal and $H_{\gamma}$ is the same as $H_{n}$ in  (\ref{FiniteGeneratedSet}) and (\ref{InfGeneratedSet}) except $G_{\gamma}\in\underset{\beta<\alpha}{\displaystyle\bigcup} T_{\beta}$ $\left(\gamma<\omega\right)$ and $\mathcal{G}=\{G_{\gamma}\colon G_{\gamma}\in\underset{\beta<\alpha} {\displaystyle\bigcup}T_{\beta},\gamma<\omega\}$. An \textbf{infinitely generated set} (at $\alpha$) is defined as:
\begin{equation}
H_{\alpha}(\mathcal{G)}\,=\,\underset{\gamma\rightarrow\alpha}{\lim}\,H_{\gamma}(G_{\gamma},\dots,G_{0}) \label{GenInfGeneratedSet}
\end{equation}
Where $G_{\gamma}\, (\gamma\geqslant1)$\ are \textbf{principal} \textbf{generators} and $G_{0}$ is a \textbf{base generator} of $H_{\alpha}$. The language of set theory is expanded to $\bar{\mathscr{L}}^{\prime}=\{\in, H_{\alpha}\}$.
\end{definition}

\begin{definition}
\label{DefSetOfIGS}\ $S\vert_{\aleph_{0}}\,=\,\{H_{\alpha}(\mathcal{G)}\colon\mathcal{G}=\{G_{\gamma}\colon G_{\gamma}\in S,\gamma<\omega\}\}$ is known as the \textbf{set of the infinitely generated sets} from $S$.
\end{definition}

\begin{remark}
\label{48}\ All $G_{n}\in V_{\omega}$ in definition \ref{DefInfiniton}, \ref{DefSemiInfiniton} and \ref{DefQuasiInfiniton} are changed to $G_{\gamma}\in\underset{\beta<\alpha}{\displaystyle\bigcup}T_{\beta}$ $\left(\gamma<\omega\right)$.
\end{remark}

The axiom of extensionality for IGS can be modified from axiom \ref{50}.

\begin{axiom}
\label{49} \ Suppose $\mathcal{G}_{1}=\{G_{\gamma}^{1}\colon G_{\gamma}^{1}\in\underset{\beta<\alpha}{\displaystyle\bigcup} T_{\beta},\,\gamma<\omega\}$ and $\mathcal{G}_{2}=\{G_{\gamma}^{2}\colon G_{\gamma}^{2}\in\underset{\beta<\alpha} {\displaystyle\bigcup}T_{\beta},\,\gamma<\omega\}$. Then
\[
\left(\forall\gamma<\omega\right)\left(G_{\gamma}^{1}\,=\,G_{\gamma}^{2}\right)\:\Longrightarrow\:H_{\alpha}(\mathcal{G}_{1})\,=\,H_{\alpha}(\mathcal{G}_{2})
\]
\end{axiom}

\begin{corollary}
\qquad
\end{corollary}

\begin{enumerate}
\item $S_{1}\subset S_{2}\:\Longrightarrow\:S_{1}|_{\aleph_{0}}\subset S_{2}|_{\aleph_{0}}$

\item $S_{1}|_{\aleph_{0}}\cup S_{2}|_{\aleph_{0}}\,\subset\,\left(S_{1}\cup S_{2}\right)|_{\aleph_{0}}$

\item $\left(S_{1}\cap S_{2}\right)|_{\aleph_{0}}\,\subset\,S_{1}|_{\aleph_{0}}\cap S_{2}|_{\aleph_{0}}$

\item $\underset{\alpha\in D}{\displaystyle\bigcup}\left.S_{\alpha}\right\vert_{\aleph_{0}}\,\subset\Bigl(\underset{\alpha\in D}{\displaystyle\bigcup}S_{\alpha}\Bigr)\Big\vert_{\aleph_{0}}$
\end{enumerate}

\begin{proof}
\ (i) \ By definition \ref{DefSetOfIGS}, for any $H_{\alpha}(\mathcal{G)}\in S_{1}\vert_{\aleph_{0}}$ and any $G_{\gamma}\in S_{1}$, $G_{\gamma}\in S_{2}$. So $H_{\alpha}(\mathcal{G)}\in S_{2}\vert_{\aleph_{0}}$. The rest follow from (i).
\end{proof}

\begin{corollary}
\qquad
\[
\Bigl(\underset{\beta<\alpha}{\bigcup}T_{\beta}\Bigr)\Big \vert_{\mathfrak{S}}\cup\Bigl(\underset{\beta<\alpha}{\bigcup}T_{\beta}\Bigr)\Big\vert_{\mathfrak{Q}}\subset\Bigl(\underset{\beta<\alpha}{\bigcup}T_{\beta}\Bigr)\Big\vert_{\aleph_{0}}
\]
\end{corollary}

\begin{proof}
\ By definition \ref{DefSetSemiInfinitons}, \ref{DefSetQuasiInfiniton}, \ref{DefSetOfIGS} and (\ref{TotalHierarchy}).\bigskip
\end{proof}

Rank in the total universe is the same as that of the von Neumann universe (definition \ref{DefModifiedCumulativeHierarchyRank}).

\begin{definition}
\label{DefTotalModelRank}\ The \textbf{rank} of $X$ in $T$ is defined as the least $\alpha$ that $X\in T_{\alpha}$ and denoted as $R_{T}(X)$, i.e. $R_{T}(X)=\,\smashoperator{\inf\limits_{\alpha\in\mathrm{Ord}}}\,\,\,\{\alpha\colon X\in T_{\alpha}\}$.
\end{definition}

The following lemma can be easily proved.

\begin{lemma}
\label{35}\ Suppose $X\in T$. Then
\begin{alignat*}{2}
	R_{T}(X)=\alpha\,\,&\Longleftrightarrow\,\bigl(X\in T_{\alpha}\,\wedge\,X\notin T_{\alpha-1}\bigr),& \alpha\in\mathrm{SOrd}
	\\
	\,\,&\Longleftrightarrow\,\bigl(X\in T_{\alpha}\,\wedge\, (\forall\beta<\alpha)(X\notin T_{\beta})\bigr),\qquad\qquad& \alpha\in\mathrm{LOrd}	
\end{alignat*}
\end{lemma}

The height function in $V_{\omega}$ (definition \ref{DefHeightFunction}) can be extended to $T$.

\begin{definition}
\label{DefGeneralHeightFunction}\ The \textbf{height function} of $X$ in $T$ is a function $h\colon T\to \mathrm{Ord}$ and $h(X)=\sup\{R_T(Y)\colon Y\in X\wedge X\in T\,\}$ where $h(\varnothing)=0$.
\end{definition}

The notion of $\omega$-invariance (definition \ref{DefOmegaInvariant}) can be extended to any limit ordinal as well.

\begin{definition}
\label{DefGenOmegaInvariant} \ Suppose $\alpha$ is a limit ordinal, $H_{\alpha}$ is an IGS and $H_{\alpha+\xi}=\{\ast G_{\alpha+\xi},\,H_{\alpha+\xi-1}\}$, where $G_{\alpha+\xi}\in\underset{\delta<\alpha}{\displaystyle\bigcup}T_{\delta}$ $\left(\xi\geqslant1\right)$. If for any $\gamma>\alpha$, there is a (successor ordinal) $\beta>\gamma$ that $H_{\beta}=H_{\alpha}$, then $H_{\alpha}$ is called $\,\mathbf{\omega}$\textbf{-invariant}.
\end{definition}

\begin{lemma}
\label{65} \ Suppose $\alpha$ is a limit ordinal, $H_{\alpha}$ is $\,\omega$-invariant and $G_{\xi}\in \underset{\mathfrak{\delta}<\alpha}{\displaystyle\bigcup}T_{\mathfrak{\delta}}\,\left(\xi\leqslant\beta\right)$. Then
\[
H_{\alpha}\mathcal\,=\,\{\ast G_{\beta},\{\ast G_{\beta-1},\dots\{\ast G_{\alpha+1},H_{\alpha}\}\dots\}
\]
\end{lemma}

\begin{proof}
\ Suppose $\alpha^{\prime}$ is the limit ordinal immediately above $\alpha$. WLOG, by definition \ref{DefGenOmegaInvariant}, assume for any $\gamma$ $\left(\alpha<\gamma<\alpha^{\prime}\right)$, there is a $\beta$ $\left(\gamma<\beta< \alpha^{\prime}\right)$ that $H_{\beta}=H_{\alpha}$. Then
\[
H_{\beta}\,=\,\{\ast G_{\beta},H_{\beta-1}\}\,=\,\{\ast G_{\beta},\{\ast G_{\beta-1},\dots\{\ast G_{\alpha+1},H_{\alpha}\} \dots\}\,=\,H_{\alpha}
\]
where $G_{\xi}\in\underset{\mathfrak{\delta}<\alpha}{\displaystyle\bigcup}T_{\mathfrak{\delta}}$ for $\xi\leqslant\beta$.
\end{proof}

\begin{remark}
\ Lemma \ref{65} shows that a $\omega$-invariant set always has an immediate member, while (in general) an IGS does not have one.
\end{remark}

Furthermore, the union operator and transitive closure need be extended to the transfinite case.

\begin{definition}
\label{DefGenUnionOperator} \ Suppose ${\bigcup}S=\{z\colon\exists y\left(y\in S\,\wedge\,z\in y\right)\}$. The $\mathbf{\alpha}^{th}$ \textbf{union operator} is defined (recursively) as:
\begin{alignat*}{2}
{\textstyle\bigcup}^{0}S\,&=\,S,
\\
{\textstyle\bigcup}^{\alpha}S\,&=\,{\textstyle\bigcup}\:{\textstyle\bigcup}^{\alpha-1}S,& \alpha\in\mathrm{SOrd}
\\
{\textstyle\bigcup}^{\alpha}S\,&=\,\underset{\beta<\alpha}{\bigcup}{\textstyle\bigcup}^{\beta}S,\qquad\qquad\qquad& \alpha\in\mathrm{LOrd}
\end{alignat*}
\end{definition}

\begin{definition}
\ Suppose $\alpha_{0}$ is the least ordinal $\alpha$ that \ ${\bigcup}^{\alpha}S\,=\,{\bigcup} ^{\alpha+1}S$. Then the \textbf{transitive closure} of $S$ is:
\begin{equation}
TC(S)\,=\underset{\alpha\leqslant\alpha_{0}}{\bigcup}{\textstyle\bigcup}^{\alpha}S\label{TransitiveClosure}
\end{equation}
\end{definition}

\begin{corollary}
\label{53}\qquad
\end{corollary}

\begin{enumerate}
\item \textit{For any $\alpha<\omega$, $V_{\alpha}=T_{\alpha}$. Otherwise, $V_{\alpha}\subsetneq T_{\alpha}$.}

\item $V\subsetneq T$.

\item \textit{$T_{\alpha}$ contains all ordinals less than $\alpha$. $T$ contains all ordinals.}
\end{enumerate}

\begin{proof}
\ (i) \ By (\ref{CumulativeHierarchy}) and (\ref{TotalHierarchy}), for any $\alpha<\omega$, $V_{\alpha}=T_{\alpha}$. Clearly
\[
T_{\omega}\,=\,\underset{\beta<\omega}{\displaystyle\bigcup}T_{\beta}\,\cup\Bigl(\underset{\beta<\omega}{\bigcup}T_{\beta}\Bigr)\Big\vert_{{\aleph}_{0}}\,\supsetneq\,\underset{\beta<\omega} {\displaystyle\bigcup}T_{\beta}\,=\,V_{\omega}
\]

Suppose $V_{\beta}\subsetneq T_{\beta}$ for any $\beta$, $\omega<\beta<\alpha$. Then if $\alpha$ is a successor ordinal
\[
T_{\alpha}\,=\,\mathcal{P}(T_{\alpha-1})\,\supsetneq\,\mathcal{P}(V_{\alpha-1})\,=\,V_{\alpha}
\]

If $\alpha$ is a limit ordinal, then
\[
T_{\alpha}\,=\,\underset{\beta<\alpha}{\displaystyle\bigcup}T_{\beta}\,\cup\Bigl(\underset{\beta<\alpha}{\bigcup}T_{\beta}\Bigr)\Big\vert_{\aleph_{0}}\,\supsetneq\,\underset{\beta<\alpha} {\displaystyle\bigcup}T_{\beta}\,\supsetneq\,\underset{\beta<\alpha}{\displaystyle\bigcup}V_{\beta}\,=\,V_{\alpha}
\]

(ii) \ By (i), (\ref{CumulativeHierarchy}) and (\ref{TotalHierarchy}).\medskip

(iii) \ By definition \ref{DefGenInfGeneratedSet}, $\Bigl(\underset{\beta<\alpha}{\displaystyle\bigcup}T_{\beta}\Bigr)\Big\vert_{\aleph_{0}}$ contains no ordinals since any set in it is NWF, but all ordinals are WF. So $T_{\alpha}$ has the same ordinals as $V_{\alpha}$ and by corollary \ref{3}, $T_{\alpha}$ contains all ordinals less than $\alpha$. By (\ref{TotalHierarchy}), $T$ contains all ordinals.\smallskip
\end{proof}

\begin{corollary}
\label{54}\qquad
\end{corollary}

\begin{enumerate}
\item \textit{Each $T_{\alpha}$ is transitive.}

\item $T$ \textit{is transitive.}
\end{enumerate}

\begin{proof}
\ (i) \ We prove by the transfinite induction. First, since $T_{1}=\{\varnothing\}$ and $\varnothing\subset T_{1}$, $T_{1}$ is transitive. Suppose it is true for any $\beta<\alpha$. Then if $\alpha$ is a successor ordinal, for any $X\in T_{\alpha} \,=\,\mathcal{P}(T_{\alpha-1})$, $X\subset T_{\alpha-1}$. So for any $Y\in X$, $Y\in T_{\alpha-1}$. Since $T_{\alpha-1}$ is transitive, $Y\subset T_{\alpha-1}$ and $Y\in T_{\alpha}$. Thus $T_{\alpha}$ is transitive.\medskip

If $\alpha$ is a limit ordinal, then for any $X\in T_{\alpha}$, if $X\in\underset{\beta<\alpha}{\displaystyle\bigcup} T_{\beta}$, there is a $\gamma<\alpha$ that $X\in T_{\gamma}$. Since $T_{\gamma}$ is transitive, $X\subset T_{\gamma}\subset T_{\alpha}$. If $X\in\Bigl(\underset{\beta<\alpha}{\displaystyle\bigcup}T_{\beta}\Bigr)\Big \vert_{\aleph_{0}}$ and $X$ is $\omega$-invariant, then by lemma \ref{65}
\[
X\,=\,\{\ast G_{\eta},\{\ast G_{\eta-1},\dots\}\dots\},\:G_{\xi}\in\underset{\beta<\alpha}{\displaystyle\bigcup}T_{\beta} \,\subset\,T_{\alpha}\text{ \ for }\xi\leqslant\eta.
\]

Then there is a $\gamma<\alpha$ that $G_{\eta}\in T_{\gamma}$. Since $T_{\gamma}$ is transitive, for any $z\in G_{\eta}$, $z\in T_{\gamma}\subset T_{\alpha}$. Furthermore, $\{\ast G_{\eta-1},\{\ast G_{\eta-2},\dots\}\dots\}$ is also $\omega$-invariant and
\[
\{\ast G_{\eta-1},\{\ast G_{\eta-2},\dots\}\dots\}\in\Bigl(\underset{\beta<\alpha}{\displaystyle\bigcup}T_{\beta}\Bigr)\Big\vert_{\aleph_{0}}\subset\,T_{\alpha}
\]

So $X\subset T_{\alpha}$. If $X$ is not $\omega$-invariant, then $X$ does not have immediate members. Both cases show that $T_{\alpha}$ is transitive if $\alpha$ is a limit ordinal.\medskip

(ii) \ By (i) and (\ref{TotalHierarchy}).\smallskip
\end{proof}

\begin{corollary}
\label{4}\qquad
\end{corollary}

\begin{enumerate}
\item \textit{For any} $\alpha$, $T_{\alpha}\,\subset\,\mathcal{P}(T_{\alpha})$

\item $\alpha<\beta\:\Longrightarrow\:T_{\alpha}\subset T_{\beta}$
\end{enumerate}

\begin{proof}
\ (i) \ By corollary \ref{54}, $T_{\alpha}$ is transitive. So for any $X\in T_{\alpha}$, $X\subset T_{\alpha}$. Thus $X\in\mathcal{P}(T_{\alpha})$ and (i) follows.\medskip

(ii) \ Suppose it is true for any $\alpha<\gamma<\beta$. If $\beta$ is a successor ordinal, then by (i)
\[
T_{\beta}\,=\,\mathcal{P}(T_{\beta-1})\,\supset\,T_{\beta-1}\,\supset\,T_{\alpha}
\]

If $\beta$ is a limit ordinal, then
\[
T_{\beta}\,=\,\underset{\gamma<\beta}{\displaystyle\bigcup}T_{\gamma}\,\cup\Bigl(\underset{\gamma<\beta}{\displaystyle\bigcup}T_{\gamma}\Bigr)\Big\vert_{\aleph_{0}}\,\supset\,\underset{\gamma<\beta} {\displaystyle\bigcup}T_{\gamma}\,\supset\,T_{\alpha}
\]
\end{proof}

\subsection{Rank in Total Universe}\label{SectionRankTotalUniverse}

Now we investigate more on rank in the total universe.

\begin{corollary}
\label{55} \ Suppose $X\in T$. Then
\[
Y\in X\:\Longrightarrow\:R_{T}(Y)\leqslant R_{T}(X)
\]
\end{corollary}

\begin{proof}
\ Suppose $R_{T}(X)=\alpha$. Then by definition \ref{DefTotalModelRank}, $X\in T_{\alpha}$. By corollary \ref{54}, $T_{\alpha}$ is transitive. Thus for $Y\in X$, $Y\in T_{\alpha}$, i.e. $R_{T}(Y)\leqslant\alpha$.
\end{proof}

\begin{corollary}
\label{56}\ If $R_T(X)$ is a successor ordinal, then
\[
Y\in X\:\Longrightarrow\:R_{T}(Y)<R_{T}(X)
\]
\end{corollary}

\begin{proof}
\ Suppose $R_{T}(X)=\alpha$ is a successor ordinal. Then $X\in T_{\alpha}\,=\,\mathcal{P}(T_{\alpha-1})$ and $X\subset T_{\alpha-1}$. Thus $Y\in T_{\alpha-1}$ and $R_{T}(Y)\leqslant\alpha-1<R_{T}(X)$.
\end{proof}

\begin{corollary}
\ If $R_T(X)$ is a successor ordinal, then
\[
Y\subset X\:\Longrightarrow\:R_{T}(Y)\leqslant R_{T}(X)
\]
\end{corollary}

\begin{proof}
\ Suppose $R_{T}(X)=\alpha$ is a successor ordinal. Then $X\in T_{\alpha}$ and $X\subset T_{\alpha-1}$. So $Y\subset T_{\alpha-1}$ and $Y\in T_{\alpha}$. Thus $R_{T}(Y)\leqslant\alpha=R_{T}(X)$.\smallskip

Note that this can fail if $R_T(X)$ is a limit ordinal. For example, let $Z=\{*G,Z\}$ where $G=\{a,b\}\in V_{\omega}$. Then $\{a,Z\}\subset Z$. But $R_T(\{a,Z\})=\omega+1$ while $R_T(Z)=\omega$.
\end{proof}

\begin{corollary}
\label{57} \ If $R_{T}(X)$ is a limit ordinal, then $X$ is an infinitely generated set.
\end{corollary}

\begin{proof}
\ Suppose $R_{T}(X)=\alpha$ is a limit ordinal. If $X\in\underset{\beta<\alpha}{\displaystyle\bigcup}T_{\beta}$, then there is a $\gamma<\alpha$ that $X\in T_{\gamma}$, i.e. $R_{T}(X)\leqslant\gamma<\alpha$, contradiction. So $X\in\Bigl(\underset{\beta<\alpha}{\displaystyle\bigcup}T_{\beta}\Bigr)\Big\vert_{\aleph_{0}}$.\smallskip
\end{proof}

\begin{corollary}
\label{66} \ Suppose $\alpha$ is a limit ordinal and $H$ is $\omega$-invariant. Then $R_{T}(H)=\alpha$.
\end{corollary}

\begin{proof}
\ By lemma \ref{65}, $H=\{\ast G_{\gamma},\{\ast G_{\gamma-1},\dots\{\ast G_{\alpha+1},H\}\dots\}$, where $G_{\xi}\in \underset{\beta<\alpha}{\displaystyle\bigcup}T_{\beta}$ for $\xi\leqslant\gamma$. Let $H_{1}=\{\ast G_{\gamma-1},\dots \{\ast G_{\alpha+1},H\}\dots\}$. If $R_{T}(H)>\alpha$, then by corollary \ref{56}, $R_{T}(H)<R_{T}(H_{1})<R_{T}(H)$, contradiction. Thus $R_{T}(H)=\alpha$.\bigskip
\end{proof}

Note that there are sets that appear to have successor ranks but actually have limit ordinal ranks.

\begin{example}
\ Suppose $a,b\in V_{\omega}$ and $Q=\{a,b,2\}|_{\mathfrak{Q}}$. Then $R_{T}(\{b,Q\})=\omega$.
\end{example}

At first sight, $\{b,Q\}\in T_{\omega+1}$ for $\{b,Q\}\subset T_{\omega}$. But by theorem \ref{43}, $\{a,\{b,Q\}\}=Q$. So $\{b,\{a,\{b,Q\}\}\}=\{b,Q\}$, i.e. $\{b,Q\}$ is $\omega$-invariant. By corollary \ref{66}, $R_{T}(\{b,Q\})=\omega$.\bigskip

The total universe can be partitioned by rank as follows.

\begin{theorem}
[Partition Formula by Rank]\label{59} \ Suppose $A_{\alpha}=\{X\colon R_{T}(X)=\alpha\,\wedge\, X\in T\}$.
\end{theorem}

\begin{enumerate}
\item \textit{If $\alpha$ is a successor ordinal, then}\qquad $A_{\alpha}\,=\,T_{\alpha}-T_{\alpha-1}$

\item \textit{If $\alpha$ is a limit ordinal, then}
\[
A_{\alpha}\,=\,T_{\alpha}\,-\underset{\beta<\alpha}{\displaystyle\bigcup}T_{\beta}\,\,=\Bigl(\underset{\beta<\alpha}{\displaystyle\bigcup}T_{\beta}\Bigr)\Big\vert_{\aleph_{0}}-\,\underset{\beta<\alpha}{\displaystyle\bigcup} T_{\beta}\,\,=\,\underset{\beta<\alpha}{\displaystyle\bigcap}\left(T_{\alpha}-T_{\beta}\right)
\]

\item \textit{In general, $T$ can be partitioned by rank as:}
\[
T\,\,=\,\,\,\smashoperator{{\displaystyle\bigcup}_{\alpha\in \mathrm{SOrd}}}\,\,(T_{\alpha}-T_{\alpha-1})\,\,\cup\,\,\, \smashoperator{{\displaystyle\bigcup}_{\alpha\in \mathrm{LOrd}}} \quad\underset{\beta<\alpha}{\displaystyle\bigcap}(T_{\alpha}-T_{\beta})
\]
\end{enumerate}

\begin{proof}
\ (i) \ If $\alpha$ is a successor ordinal, by lemma \ref{35}, $R_{T}(X)=\alpha$ if and only if $X\in T_{\alpha}-T_{\alpha-1}$.\medskip

(ii) \ If $\alpha$ is a limit ordinal, by lemma \ref{35}, $R_{T}(X)=\alpha$ if and only if $X\in T_{\alpha}$ and for any $\beta<\alpha$, $X\notin T_{\beta}$. So by (\ref{TotalHierarchy})
\[
A_{\alpha}\,=\,T_{\alpha}\,-\underset{\beta<\alpha}{\displaystyle\bigcup}T_{\beta}\,\,=\,\underset{\beta<\alpha}{\bigcup} T_{\beta}\,\cup\Bigl(\underset{\beta<\alpha}{\displaystyle\bigcup}T_{\beta}\Bigr)\Big\vert_{\aleph_{0}}-\underset{\beta<\alpha}{\displaystyle\bigcup}T_{\beta}\,\,=\Bigl(\underset{\beta<\alpha}{\displaystyle\bigcup}T_{\beta}\Bigr)\Big \vert_{\aleph_{0}}-\,\underset{\beta<\alpha}{\displaystyle\bigcup}T_{\beta}
\]

On the other hand, clearly we have
\[
T_{\alpha}\,-\underset{\beta<\alpha}{\displaystyle\bigcup}T_{\beta}\,\,=\,\underset{\beta<\alpha}{\displaystyle\bigcap} \left(T_{\alpha}-T_{\beta}\right)
\]

(iii) \ For any $X\in T$ and $R_{T}(X)=\alpha$, if $\alpha$ is a successor ordinal, then by (i)
\[
X\in T_{\alpha}-T_{\alpha-1}\,\subset\,\,\,\smashoperator{\bigcup_{\alpha\in \mathrm{SOrd}}}\,\left(T_{\alpha}-T_{\alpha-1}\right)
\]

If $\alpha$\ is a limit ordinal, then by (ii)
\[
X\in T_{\alpha}-\underset{\beta<\alpha}{\displaystyle\bigcup}T_{\beta}\,\,\subset\,\,\,\smashoperator{{\displaystyle\bigcup}_{\alpha\in \mathrm{LOrd}}} \quad\underset{\beta<\alpha}{\displaystyle\bigcap}(T_{\alpha}-T_{\beta})
\]

Obviously, each category is disjoint from others.\bigskip
\end{proof}

We also have the following theorems for rank in the total universe.

\begin{theorem}
\label{60} \ Suppose $X\in T$. Then
\end{theorem}

\begin{enumerate}
\item \textit{If $R_{T}(X)$ is a successor ordinal, then there is either a $Y\in X$ that $R_{T}(Y)=R_{T}(X)-1$, or a sequence $Y_{n}\in X$ that $R_{T}(Y_{n})\rightarrow R_{T}(X)-1$ as $n\rightarrow\infty$.}

\item \textit{If $R_{T}(X)$ is a limit ordinal and $X$ is $\omega$-invariant, then there is one and only one $Y\in X$
that $R_{T}(Y)=R_{T}(X)$.}
\end{enumerate}

\begin{proof}
\ (i) \ Suppose $R_{T}(X)=\alpha$ is a successor ordinal. By corollary \ref{56}, for any $Y\in X$, $R_{T}(Y)\leqslant\alpha-1$. If there is a $Y\in X$ that $R_{T}(Y)=\alpha-1$, then done. So suppose there is no $Y\in X$ that $R_{T}(Y)=\alpha-1$. \smallskip

If $\alpha-1$ is a successor ordinal, then for any $Y\in X$, $Y\in T_{\alpha-2}$. So $X\subset T_{\alpha-2}$ and $X\in T_{\alpha-1}$, i.e. $R_{T}(X)\leqslant\alpha-1$, contradiction. If $\alpha-1$ is a limit ordinal, then there is a sequence $Y_{n}\in X$ that $R_{T}(Y_{n})\rightarrow\alpha-1$. Otherwise suppose there is a $\beta<\alpha-1$ that for any $Y\in X$, $R_{T}(Y)\leqslant\beta$, i.e. $Y\in T_{\beta}$. So $X\subset T_{\beta}$ and $X\in T_{\beta+1}$, i.e. $R_{T}(X)\leqslant \beta+1<\alpha$, contradiction.\medskip

(ii) \ Suppose $R_{T}(X)=\alpha$ is a limit ordinal. Then by lemma \ref{65}, $X=\,\{\ast G_{\gamma},\{\ast G_{\gamma-1}, \dots\}\dots\}$, where $G_{\xi}\in\underset{\beta<\alpha}{\displaystyle\bigcup}T_{\beta}$ for $\xi\leqslant\gamma$. Let $Y=\{\ast G_{\gamma-1},\{\ast G_{\gamma-2},\dots\}\dots\}$. Then $Y$ is $\omega$-invariant and by corollary \ref{66}, $R_{T}(Y)=\alpha$. Furthermore, since $G_{\gamma}\in\underset{\beta<\alpha}{\displaystyle\bigcup}T_{\beta}$, $R_{T}(G_{\gamma})<\alpha$. So by corollary \ref{55}, for any $a\in G_{\gamma}$, $R_{T}(a)<\alpha$, i.e. $Y$ is the only one in $X$ that $R_{T}(Y)=\alpha$.\smallskip
\end{proof}

\begin{corollary}
\ Suppose $X\in T$. Then
\end{corollary}

\begin{enumerate}
\item \textit{If $R_{T}(X)$ is a successor ordinal, then}
\[
R_{T}(X)\,=\,\sup\{R_{T}(Y)\colon Y\in X\}\,+\,1
\]

\item \textit{If $R_{T}(X)$ is a limit ordinal and $X$ is $\omega$-invariant, then}
\[
R_{T}(X)\,=\,\sup\{R_{T}(Y)\colon Y\in X\}
\]
\end{enumerate}

\begin{proof}
\ (i) \ By corollary \ref{56}, for any $Y\in X$, $R_{T}(Y)\leqslant R_{T}(X)-1$. Then it follows by theorem \ref{60}(i).\medskip

(ii) \ By theorem \ref{60}(ii).\smallskip
\end{proof}

\subsection{Spectrum of Power Set Operations}

In this section, we will introduce the notion of the power set spectrum from which the total universe (\ref{TotalHierarchy}) can be proved. First, we have the following definitions.

\begin{definition}
\label{DefAlphaPowerSetOperations}\ Suppose $\mathcal{P}(S)\,=\,\{x\colon x\subset S\}$. Then the $\mathbf{\alpha}^{th}$ \textbf{power set operation} (on $S$) is defined (recursively) as:
\begin{alignat*}{2}
\mathcal{P}^{(0)}(S)\,&=\,S,
\\
\mathcal{P}^{(\alpha)}(S)\,&=\,\mathcal{P}\bigl(\mathcal{P}^{(\alpha-1)}(S)\bigr),&\alpha\in\mathrm{SOrd}
\\
\mathcal{P}^{(\alpha)}(S)\,&=\,\underset{\gamma\rightarrow\alpha}{\lim}\mathcal{P}^{(\gamma)}(S),\qquad\qquad\qquad& \alpha\in\mathrm{LOrd}
\end{alignat*}
\end{definition}

%

\begin{definition}
	\ A \textbf{spectrum of power set operations} is defined as:
	\begin{equation}
		\mathfrak{Sp}\,=\,\underset{\alpha\in\Omega}{\displaystyle\bigcup}\mathcal{P}^{(\alpha)}(\varnothing)\label{PowerSetSpectrum}
	\end{equation}
	Where $\Omega$ is a collection of ordinals and is known as a \textbf{domain} of the spectrum.
\end{definition}

First, we show that the von Neumann universe is the power set spectrum at the successor ordinals.

\begin{lemma}
	\label{69}$\qquad V\,=\,\,\,\smashoperator{\bigcup_{\alpha\in \mathrm{SOrd}}}\,\mathcal{P}^{(\alpha)}(\varnothing)$
\end{lemma}

\begin{proof}
	\ First, we prove for any successor ordinal $\alpha$, $V_{\alpha}=\mathcal{P}^{(\alpha)}(\varnothing)$. Clearly, $\mathcal{P}^{(1)} (\varnothing)=V_{1}$. Suppose it is true for any successor ordinal $\beta<\alpha$. Then by definition \ref{DefAlphaPowerSetOperations} and (\ref{CumulativeHierarchy})
	\[
	V_{\alpha}\,=\,\mathcal{P}(V_{\alpha-1})\,=\,\mathcal{P}(\mathcal{P}^{(\alpha-1)}(\varnothing)) \,=\, \mathcal{P}^{(\alpha)}(\varnothing)
	\]
	
	If $\alpha$ is a limit ordinal, let $V_{\alpha}\,=\underset{\beta<\alpha}{\displaystyle\bigcup}\mathcal{P}^{(\beta)} (\varnothing)$. Since $V_{\alpha}$ is transitive, $V_{\alpha}\subset V_{\alpha+1}$. Thus
	\[
	V\,=\,\,\smashoperator{\bigcup_{\alpha\in \mathrm{Ord}}}\,V_{\alpha}\,=\,\,\,\smashoperator{\bigcup_{\alpha\in \mathrm{SOrd}}}\,V_{\alpha} \,\cup\,\,\,\smashoperator{\bigcup_{\alpha\in \mathrm{LOrd}}}\,V_{\alpha}\,=\,\,\,\smashoperator{\bigcup_{\alpha\in \mathrm{SOrd}}}\,V_{\alpha} \,=\,\,\,\smashoperator{\bigcup_{\alpha\in \mathrm{SOrd}}}\,\mathcal{P}^{(\alpha)}(\varnothing)
	\]
\end{proof}

In order to prove that the total universe is the power set spectrum at all (limit) ordinals, we need the following results.

\begin{theorem}
	\label{51} \ Suppose $\mathfrak{P}_n=\langle\mathcal{P}^{n}(\varnothing), \in,\varnothing\rangle\:(n<\omega)$ and $\,\mathfrak{P}=\bigcup_{n<\omega}\mathfrak{P}_n$. Then
\end{theorem}

\begin{enumerate}
	\item $\mathfrak{P}$ \textit{is ultrahomogeneous.}	
	
	\item $\underset{n\rightarrow\omega}{\lim}\mathfrak{P}_{n}\,=\,\underset{n\rightarrow\omega}{\lim}\mathcal{P}^{(n)}(\varnothing)\,=\,\mathcal{P}^{(\omega)}(\varnothing)$.
\end{enumerate}

\begin{proof}
\ (i) \ We prove $\underset{n<\omega}{\displaystyle\bigcup}\mathcal{P}^{(n)}(\varnothing)$ is an amalgamation class. Then (i) follows from proposition \ref{9}. First heredity is obvious. For joint embedding, suppose $X_{i},X_{j} \in\mathcal{P}^{(n)}(\varnothing)$. Since for $X_{k}\supset X_{i}\cup X_{j}$, there are embeddings $f_{0}\colon X_{i} \rightarrow X_{k}$, $f_{1}\colon X_{j}\rightarrow X_{k}$ and $X_{k}\in\mathcal{P}^{(n+1)}(\varnothing)$, joint embedding holds. For amalgamation, suppose $X_{i},X_{j},X_{k}\in\mathcal{P}^{(n)}(\varnothing)$ and embeddings $f_{0}\colon X_{i}\rightarrow X_{j}$ and $f_{1}\colon X_{i}\rightarrow X_{k}$. Then there is $X_{l}\supset X_{j}\cup X_{k}$, $g_{0}\colon X_{j}\rightarrow X_{l}$, $g_{1}\colon X_{k}\rightarrow X_{l}$ and $g_{0}\circ f_{0}=g_{1}\circ f_{1}$. Since $X_{l}\in \mathcal{P}^{(n+1)}(\varnothing)$, amalgamation property holds.\medskip

\ (ii) \ By (i) and proposition \ref{21}, Th$(\mathfrak{P})$ is $\aleph_{0}$-categorical. Suppose
\begin{equation}
\psi_{n}(x)\,\Longleftrightarrow\,\exists!\,y_{n} \,\cdots\,\exists!\,y_{1}\,\bigl(\,\,\,\smashoperator{\bigwedge_{1\leqslant j\leqslant n-1}} \,\,\bigl(y_{j}\in y_{j+1}\wedge(\forall z\in y_{j+1})(z\subset y_{j})\bigr)\,\wedge\,y_n=x\bigr)\label{PowerSetOneType}
\end{equation}

Then the validity of $\psi_n(x)$ means that there is a unique $\in$-sequence of length $n$ in $x$ with each $y_{j+1}=\mathcal{P}(y_{j})$. Since $\mathfrak{P}_n\vDash\psi_n\left[\mathcal{P}^{(n)}(\varnothing)\right]$, $(\psi_n(x))$ is a $1$-type of $\mathfrak{P}$. Furthermore, for any $k>n,\,\mathfrak{P}_k\vDash \psi_n\left[\mathcal{P}^{(k)}(\varnothing)\right]$. So $\left(\mathfrak{P}_{n},\psi_{n}\right)$ is a homogeneous sequence. And by definition \ref{DefLimit}, $\underset{n\rightarrow\omega}{\lim}\mathfrak{P}_{n}$ is unique. We no longer distinguish $\underset{n\rightarrow\omega}{\lim}\mathfrak{P}_{n}$ and $\underset{n\rightarrow\omega}{\lim}\mathcal{P}^{(n)}(\varnothing)$. So (ii) follows by definition \ref{DefAlphaPowerSetOperations}.
\end{proof}

\begin{theorem}
	\label{26} \ Suppose $H_{n}(G_{n},\dots,G_{0})$ is defined in (\ref{FiniteGeneratedSet}) ($G_{i}\in V_{\omega}\wedge i<n<\omega$) and $\underset{n\rightarrow\omega}{\lim}H_{n}$ is unique. Then
\end{theorem}

\begin{enumerate}
	\item $\bigl(\exists X\in\mathcal{P}^{(n)}(\varnothing)\bigr)\, (X=H_{n})$.
	
	\item $(\forall n<k<\omega)\,\bigl(\mathcal{P}^{(n)}(\varnothing)\,\subset\,\mathcal{P}^{(k)}(\varnothing)\bigr)$.
	
	\item $\bigl(\exists X\in\mathcal{P}^{(\omega)}(\varnothing)\bigr)\, \Bigl(X=\underset{n\rightarrow\omega}{\lim}H_{n}\Bigr)$.

	\item $(\forall n<\omega)\,\bigl(\mathcal{P}^{(n)}(\varnothing)\,\subset\,\mathcal{P}^{(\omega)} (\varnothing)\bigr)$.
	
	\item \textit{A type of $\mathcal{P}^{(\omega)}(\varnothing)$ is that there exists a unique $\in$-sequence of length $\omega$ in $\mathcal{P}^{(\omega)}(\varnothing)$, i.e.}
	\[
	\mathcal{P}^{(1)}(\varnothing)\in\mathcal{P}^{(2)}(\varnothing)\in\dots\in\mathcal{P}^{(n)}(\varnothing)\in\dots\in\mathcal{P}^{(\omega)}(\varnothing)
	\]
\end{enumerate}

\begin{proof}
\ (i) \ We prove by induction. Since $\mathcal{P}^{(1)}(\varnothing)=\{\varnothing\}$, let $G_{0}=\ast\varnothing$ and $G_{1}=\varnothing$. Then $X\,=\,H_{1}(\varnothing,\ast\varnothing)\,=\,\varnothing$. Suppose it is true for $n\leqslant k$, i.e. there is a $H_{k}\in\mathcal{P}^{(k)}(\varnothing)$ that
\[
H_{k}\,=\,\{\ast G_{k},\{\ast G_{k-1},\dots\{\ast G_{1},G_{0}\}\dots\}
\]
Then for any $G_{k+1}\in\mathcal{P}^{(k)}(\varnothing)$, $X\,=\,\{\ast G_{k+1},H_{k}\}\subset\mathcal{P}^{(k)} (\varnothing)$. So $X\in\mathcal{P}^{(k+1)}(\varnothing)$ and $X=H_{k+1}$.\medskip

(ii) \ For any $n<\omega$, $\mathcal{P}^{(n)}(\varnothing)\in\mathcal{P}^{(n+1)}(\varnothing)$. Since $\mathcal{P}^{(n)} (\varnothing)$ is transitive, $\mathcal{P}^{(n)}(\varnothing)\subset\mathcal{P}^{(n+1)}(\varnothing)$. Generally, it follows by induction.\medskip
	
(iii) \ By axioms \ref{15}, \ref{19} and theorem \ref{51}
\begin{align*}
	\underset{n\rightarrow\omega}{\lim}\bigl(\exists X\in\mathcal{P}^{(n)}(\varnothing)\bigr)\,(X=H_{n})
	\,& \Longleftrightarrow\,\underset{n\rightarrow\omega}{\lim}\,\exists X\bigl(X\in\mathcal{P}^{(n)}(\varnothing)\wedge X=H_{n}\bigr)
	\\
	\,& \Longleftrightarrow\,\underset{n\rightarrow\omega}{\lim}\,\exists X\bigl(X\in\mathcal{P}^{(n)}(\varnothing)\wedge\, \forall Y(Y\in X\Longleftrightarrow Y\in H_n)\bigr)
	\\
	\,& \Longleftrightarrow\,\,\exists X\Bigl(X\in\underset{n\rightarrow\omega}{\lim} \mathcal{P}^{(n)}(\varnothing)\,\wedge\,\forall Y\bigl(Y\in X\Longleftrightarrow Y\in \underset{n\rightarrow\omega}{\lim}H_n\bigr)\Bigr) 
	\\
	\,& \Longleftrightarrow\,\,\exists X\Bigl( X\in \mathcal{P}^{(\omega)}(\varnothing)\,\wedge\,X= \underset{n\rightarrow\omega}{\lim}H_n\Bigr) 
	\\
	\,& \Longleftrightarrow\,\bigl(\exists X\in\mathcal{P}^{(\omega)}(\varnothing) \bigr)\,\Bigl(X=\underset{n\rightarrow\omega}{\lim}H_{n}\Bigr)
\end{align*}

By (i), for any $n<\omega$, $\bigl(\exists X\in\mathcal{P}^{(n)}(\varnothing)\bigr)(X=H_{n})$ is true. So (iii) follows by corollary \ref{18}.\medskip

(iv) \ By theorem \ref{51} and corollary \ref{24}
\[
\underset{k\rightarrow\omega}{\lim}\forall X\bigl(X\in\mathcal{P}^{(n)}(\varnothing)\Longrightarrow X\in\mathcal{P}^{(k)}(\varnothing)\bigr)\,\Longleftrightarrow\,\forall X\bigl(X\in\mathcal{P}^{(n)}(\varnothing) \Longrightarrow X\in\mathcal{P}^{(\omega)}(\varnothing)\bigr)
\]

By (ii), for any $n<k<\omega$, the left side is true. So (iv) follows by corollary \ref{18}.\medskip

(v) \ By (\ref{PowerSetOneType}) and corollary \ref{67}
\begin{align*}
	\psi_{\omega} \,& \Longleftrightarrow\,\underset{n\rightarrow\omega}{\lim}\,\psi_{n}(\mathcal{P}^{(n)} (\varnothing))
	\\
	\,& \Longleftrightarrow\,\underset{n\rightarrow\omega}{\lim}\,\,\smashoperator{\bigwedge_{1\leqslant j\leqslant n}}\,\,\exists!\, \mathcal{P}^{(j)}(\varnothing)\,\exists!\,\mathcal{P}^{(j-1)}(\varnothing)\left(
	\begin{array}
		[c]{c}
		\mathcal{P}^{(j-1)}(\varnothing)\in\mathcal{P}^{(j)}(\varnothing)\,\wedge \\[1.1ex]
		\left(\forall Z\in\mathcal{P}^{(j)}(\varnothing)\right)\left(\forall z\in Z\right)\left(z\in\mathcal{P}^{(j-1)} (\varnothing)\right)
	\end{array}
	\right)  \\[1.3ex]
	\,&\qquad\wedge\,\,\underset{n\rightarrow\omega}{\lim}\,\exists!\,\mathcal{P}^{(n-1)}(\varnothing)\left(
	\begin{array}
		[c]{c}
		\mathcal{P}^{(n-1)}(\varnothing)\in\mathcal{P}^{(n)}(\varnothing)\,\wedge  \\[1.1ex]
		\left(\forall Z\in\mathcal{P}^{(n)}(\varnothing)\right)\left(\forall z\in Z\right)\left(z\in\mathcal{P}^{(n-1)} (\varnothing)\right)
	\end{array}
	\right)  \\[1.5ex]
	\,& \Longleftrightarrow\,\underset{n<\omega}{\displaystyle\bigwedge}\exists!\,\mathcal{P}^{(n)}(\varnothing)\,\exists!\,\mathcal{P}^{(n-1)} (\varnothing)\Bigl(\mathcal{P}^{(n-1)}(\varnothing)\in\mathcal{P}^{(n)}(\varnothing)\wedge\bigl(\forall Z\in\mathcal{P}^{(n)}(\varnothing)\bigr)\bigl(Z\subset\mathcal{P}^{(n-1)}(\varnothing)\bigr)\Bigr)  
	\\
	\,& \qquad\wedge\,\,\exists!\,\mathcal{P}^{(\omega)}(\varnothing)\,\Bigl(\mathcal{P}^{(\omega)}(\varnothing)\in \mathcal{P}^{(\omega)}(\varnothing)\wedge\bigl(\forall Z\in\mathcal{P}^{(\omega)}(\varnothing)\bigr)\bigl(  Z\subset\mathcal{P}^{(\omega)}(\varnothing)\bigr)\Bigr)
\end{align*}

By theorem \ref{71}, $\mathcal{P}^{(\omega)}(\varnothing)\vDash\psi_{\omega}$ where $\psi_{\omega}$ describes $\left\langle \mathcal{P}^{(n)}(\varnothing)\colon n\leqslant\omega\right\rangle$, the unique $\in$-sequence of length $\omega$ in $\mathcal{P}^{(\omega)}(\varnothing)$.\bigskip	
\end{proof}

The last two theorems can be extended to any limit ordinal and we omit the proof.

\begin{corollary}
	\label{28} \ Suppose $\alpha>\omega$ is a limit ordinal, $\mathfrak{P}_{\gamma}=\langle\mathcal{P}^{\gamma}(\varnothing), \in,\varnothing\rangle\:(\gamma<\alpha)$ and $\,\mathfrak{P}=\bigcup_{\beta<{\alpha}}\mathfrak{P}_{\beta}$. Then
\end{corollary}

\begin{enumerate}
	\item $\mathfrak{P}$ \textit{is ultrahomogeneous.}	

	\item $\underset{\gamma\rightarrow\alpha}{\lim}\mathfrak{P}_{\gamma}\,=\,\underset{\gamma\rightarrow\alpha}{\lim}\mathcal{P}^{(\gamma)}(\varnothing)\,=\,\mathcal{P}^{(\alpha)}(\varnothing)$.
\end{enumerate}

\begin{corollary}
	\label{68} \ Suppose $\alpha>\omega$ is a limit ordinal and $H_{\xi}(G_{\xi},\dots,G_{0})$ is defined in (\ref{GenInfGeneratedSet}) with $G_{\xi}\in\underset{\beta<\alpha}{\displaystyle\bigcup}T_{\beta}$ $\left(\xi<\omega \right)$. Let $\gamma$ be a successor ordinal. Then
\end{corollary}

\begin{enumerate}
	\item $(\forall\gamma<\alpha)\,\bigl(\exists X\in\mathcal{P}^{(\gamma)}(\varnothing)\bigr)(X=H_{\gamma})$.
	
	\item $(\forall\gamma<\eta<\alpha)\,\bigl(\mathcal{P}^{(\gamma)}(\varnothing)\,\subset\,\mathcal{P}^{(\eta)} (\varnothing)\bigr)$.
	
	\item $\bigl(\exists X\in\mathcal{P}^{(\alpha)}(\varnothing)\bigr)\,\bigl(X=\underset{\gamma\rightarrow\alpha}{\lim} H_{\gamma}\bigr)$.
	
	\item $(\forall\gamma<\alpha)\,\bigl(\mathcal{P}^{(\gamma)}(\varnothing)\,\subset\,\mathcal{P}^{(\alpha)} (\varnothing)\bigr)$.
	
	\item \textit{A type of $\mathcal{P}^{(\alpha)}(\varnothing)$ is that there exists a unique $\in$-sequence $\langle \mathcal{P}^{(\gamma)}(\varnothing)\colon \gamma\leqslant\alpha\rangle$ in $\mathcal{P}^{(\alpha)}(\varnothing)$.}\medskip
\end{enumerate}

\begin{theorem}
	\label{6}$\qquad T\,=\,\,\,\smashoperator{\bigcup_{\alpha\in \mathrm{Ord}}}\,\mathcal{P}^{(\alpha)}(\varnothing)$
\end{theorem}

\begin{proof}
\ First, we prove for any successor ordinal $\alpha$, $T_{\alpha}=\mathcal{P}^{(\alpha)}(\varnothing)$. Clearly, $\mathcal{P}^{(1)}(\varnothing)=T_{1}$. Suppose it is true for any $\beta<\alpha$. Then by definition \ref{DefAlphaPowerSetOperations} and (\ref{TotalHierarchy})
\[
T_{\alpha}\,=\,\mathcal{P}(T_{\alpha-1})\,=\,\mathcal{P}(\mathcal{P}^{(\alpha-1)}(\varnothing)) \,=\, \mathcal{P}^{(\alpha)}(\varnothing)
\]

If $\alpha$ is a limit ordinal, we prove the following.
\[
T_{\alpha}\,=\,\mathcal{P}^{(\alpha)}(\varnothing)\,=\underset{\beta<\alpha}{\displaystyle\bigcup}\mathcal{P}^{(\beta)}(\varnothing)\,\cup\Bigl(\underset{\beta<\alpha}{\displaystyle\bigcup}\mathcal{P}^{(\beta)}(\varnothing)\Bigr)\Big\vert_{\aleph_{0}}\tag{1}
\]

By corollary \ref{68}(iii), let
\[
\Bigl(\underset{\beta<\alpha}{\displaystyle\bigcup}\mathcal{P}^{(\beta)}(\varnothing)\Bigr)\Big\vert_{\mathcal {\aleph}_{0}}=\,\bigl\{X\colon X\in\mathcal{P}^{(\alpha)}(\varnothing)\,\wedge\,X=\underset{\gamma\rightarrow\alpha}{\lim} H_{\gamma}\bigr\}\tag{2}
\]

First, consider $\alpha=\omega$. For any $n<\omega$, by theorem \ref{26}(iv)
\[
\mathcal{P}^{(n)}(\varnothing)\,\subset\,\mathcal{P}^{(\omega)}(\varnothing)\quad\text{and so}\quad\underset{n<\omega} {\displaystyle\bigcup}\mathcal{P}^{(n)}(\varnothing)\,\subset\,\mathcal{P}^{(\omega)}(\varnothing)
\]

So by (2), we have
\[
\underset{n<\omega}{\displaystyle\bigcup}\mathcal{P}^{(n)}(\varnothing)\,\cup\Bigl(\underset{n<\omega} {\displaystyle\bigcup}\mathcal{P}^{(n)}(\varnothing)\Bigr)\Big\vert_{\aleph_{0}}\subset\,\mathcal{P}^{(\omega)} (\varnothing)
\]

On the other hand, for any $X\in\mathcal{P}^{(\omega)}(\varnothing)$, either $X\in\mathcal{P}^{(n)}(\varnothing)$ or $X\in\Bigl(\underset{n<\omega} {\displaystyle\bigcup}\mathcal{P}^{(n)}(\varnothing)\Bigr)\Big\vert_{\mathcal {\aleph}_{0}}$. So
\[
\mathcal{P}^{(\omega)}(\varnothing)\,\subset\underset{n<\omega}{\displaystyle\bigcup}\mathcal{P}^{(n)}(\varnothing)\,\cup \Bigl(\underset{n<\omega} {\displaystyle\bigcup}\mathcal{P}^{(n)}(\varnothing)\Bigr)\Big\vert_{\aleph_{0}}
\]

Thus (1) holds for $\omega$. The general case of (1) can be proved by corollary \ref{28}, \ref{68} and transfinite induction. Hence by
(\ref{TotalHierarchy})
\[
T\,=\,\,\,\smashoperator{\bigcup_{\alpha\in \mathrm{Ord}}}\,T_{\alpha}\,=\,\,\,\smashoperator{\bigcup_{\alpha\in \mathrm{Ord}}}\, \mathcal{P}^{(\alpha)}(\varnothing)
\]
\end{proof}\vspace{-2ex}

\begin{remark}
	\ Theorem \ref{6} shows that the total universe and the non-well-founded sets are the results of the power set spectrum at the limit ordinals. The von Neumann universe only involves the power set spectrum at the successor ordinals (lemma \ref{69}) and thus has only well-founded sets.
\end{remark}


\subsection{Set Theory for Total Universe}\label{SectionSetTheoryTotalUniverse}

In this section, we will introduce an expanded theory of Zermelo-Fraenkel set theory known as \textbf{EZF} to handle both the well-founded and non-well-founded sets. The language of EZF is an expanded language of set theory, i.e. $\bar{\mathscr{L}}^{\prime}=\{\allowbreak\in, H_{\alpha}, I_{\alpha},Z_{\alpha},Q_{\alpha},Q_{\alpha,q}\}$ where $\alpha$ is any limit ordinal. First, we list the \textbf{axioms of EZF} as follows.

\begin{enumerate}
\item[\textbf{I.}] \hypertarget{I}{}\textbf{Extensionality.}

\item Let $X$ and $Y$ be two IGS that are generated by $\mathcal{G}_{X}= \{G_{\gamma}^{X}\colon G_{\gamma}^{X}\in S,\,\gamma<\omega\}$ and $\mathcal{G}_{Y}= \{G_{\gamma}^{Y}\colon G_{\gamma}^{Y}\in S,\,\gamma<\omega\}$ respectively (definition \ref{DefGenInfGeneratedSet} and \ref{DefSetOfIGS}). Then\footnote{This is the same as axiom \ref{49}.}
\[
\left(\forall\gamma<\omega\right)\left(G_{\gamma}^{X}\,=\,G_{\gamma}^{Y}\right)\:\Longrightarrow\:X(\mathcal{G}_{X}) \,=\,Y(\mathcal{G}_{Y})
\]

\item If $X$\ and $Y$\ are any other sets, then
\[
\forall z\left(z\in X\Longleftrightarrow z\in Y\right)\,\Longleftrightarrow\,X=Y
\]

\item[\textbf{II.}] \hypertarget{II}{}\textbf{Non-well-foundedness.}\ \ Suppose $\alpha\geqslant\omega$ is an ordinal and $\alpha_{0}$ is the limit ordinal immediately below $\alpha$. Then
\[
\exists\, X\bigl(\,\,\,\smashoperator{\bigwedge_{\alpha_{0}\leqslant\gamma\leqslant\alpha}}\, \exists\,X_{\gamma} \,\exists\,X_{\gamma-1}(X_{\gamma-1}\in X_{\gamma})\,\wedge\underset{\gamma<\alpha_{0}}{\displaystyle\bigwedge} \exists\,X_{\gamma}\,\exists\,X_{\gamma-1}(X_{\gamma-1}\in X_{\gamma}) \,\wedge\,X_{\alpha}=X\bigr)
\]

\item[\textbf{III.}] \hypertarget{III}{}\textbf{Union.}\ \ Suppose ${\bigcup}^{\alpha}X$ is given in definition \ref{DefGenUnionOperator}. Then for any $\alpha\geqslant1$
\[
\forall X\,\exists Y\left(Y=\,{\textstyle\bigcup}^{\alpha}X\right)
\]

\item[\textbf{IV.}] \hypertarget{IV}{}\textbf{Pairing.}
\[
\forall x\,\forall y\,\exists S\,\forall z\left(z\in S\,\Longleftrightarrow\, z=x\vee z=y\right)
\]

\item[\textbf{V.}] \hypertarget{V}{}\textbf{Power Set.}
\[
\forall X\,\exists Y\,\forall y\left(y\in Y\Longleftrightarrow\,y\subset
X\right)
\]

\item[\textbf{VI.}] \hypertarget{VI}{}\textbf{Infinity.}
\[
\exists S\,\bigl(\varnothing\in S\,\wedge\,(\forall x\in S)\left(x\cup\{x\}\in S\right)\bigr)
\]

\item[\textbf{VII.}] \hypertarget{VII}{}\textbf{Replacement.}\ \ Suppose $\phi$ is a formula and $p$ is a tuple. Then
\[
\forall x\,\forall y\,\forall z\left(\phi(x,y,p)\wedge\phi(x,z,p)\Longrightarrow y=z\right)\:\Longrightarrow\:\forall X\,\exists Y\,\forall y\,\bigl(y\in Y\Longleftrightarrow\left(\exists x\in X\right)\phi(x,y,p)\bigr)
\]

\item[\textbf{VIII.}] \hypertarget{VIII}{}\textbf{Separation.}\ \ Suppose $\phi$ is a formula and $p$ is a tuple. Then
\[
\forall X\,\forall p\,\exists Y\,\forall y\,\bigl(y\in Y\Longleftrightarrow\, y\in X\wedge\phi(y,p)\bigr)
\]
\end{enumerate}

Since infinitely generated sets (generally) do not have immediate members, the axiom of extensionality for EZF must be modified to handle IGS. As a result, I(i) is added to decide if two IGS are equal. For all other sets, the axiom of extensionality in ZF applies. By corollary \ref{33}, the axiom of regularity can fail. Thus AR is replaced by the non-well-foundedness axiom (II) which means that there exist sets of infinite $\in$-sequences in EZF.\medskip

Since non-well-founded sets contain infinite $\in$-sequences, the union axiom in ZF must be modified as well. III which is based on definition \ref{DefGenUnionOperator} indicates that the result of any union operation is a set. The rest axioms (IV,$\dots$,VIII) of EZF remain the same as those of ZF. Consequently, we reach a main conclusion of this paper that the total universe is a model of ZF minus the axiom of regularity.

\begin{theorem}
\label{72}\ \ $T$ \textit{is a model of} EZF.
\end{theorem}

\begin{proof}
\ \textit{Extensionality}. All sets in $T$ are either WF or NWF. For a NWF set, it either has immediate members or not. I(ii) applies for all WF and NWF sets with immediate members. For two IGS without immediate members, I(i) can decide if they are equal by checking the equality of their generators, which can be either WF or NWF sets at lower ranks. Thus by transfinite recursion, I holds for all sets in $T$.\medskip

\textit{Non-well-foundedness}. \ By theorem \ref{30}, $I_{\omega}$ contains an infinite $\in$-sequence. (By corollary \ref{33}, $I_{\omega}$ also fails the axiom of regularity.) Thus II holds in $T$.\medskip

\textit{Union}. \ Suppose $X\in T$ and $R_{T}(X)=\gamma$. If $\gamma$ is a successor ordinal, $X\subset T_{\gamma-1}$. By corollary \ref{54} and theorem \ref{60}, for any $y\in X$ and $z\in y$, $z\in T_{\gamma-1}$ and so ${\bigcup} X\in T_{\gamma}$. Suppose for any $\beta<\alpha$, ${\bigcup}^{\beta}X\in T_{\gamma}$. By definition \ref{DefGenUnionOperator}, ${\bigcup}^{\alpha}X\in T_{\gamma}$. If $\gamma$ is a limit ordinal, suppose $X$ is $\omega$-invariant. By lemma \ref{65}, ${\bigcup}^{\alpha}X\in T_{\gamma+\omega}$. In both cases, we have ${\bigcup}^{\alpha}X\in T$. Thus III holds in $T$.\medskip

\textit{Pairing}. \ For any sets $a,b\in T$, $a,b\in T_{\alpha}$. So $\{a,b\}\subset T_{\alpha}$ and $\{a,b\}\in T_{\alpha+1}$. Thus IV holds in $T$.\medskip

\textit{Power set}. \ For any set $X\in T$, $X\in T_{\alpha}$. By corollary \ref{54}, $X\subset T_{\alpha}$. So for any $Y\subset X$, $Y\subset T_{\alpha}$ and $Y\in\mathcal{P}(T_{\alpha})=T_{\alpha+1}$. Thus $\mathcal{P}(X)\subset T_{\alpha+1}$ and $\mathcal{P}(X)\in T_{\alpha+2}$, and V holds in $T$.\medskip

\textit{Infinity}. \ Since $\omega\in T$, VI holds in $T$.\medskip

\textit{Replacement}. \ Suppose $f=\{\left(x,y\right)\in T\colon\phi(x,y,p)\}$. Then the first part of VII implies $f$ is a function. Let $Y=f(X)$. Then there is a (least) $\alpha$ that $f\in T_{\alpha}$ and $X\in T_{\alpha}$. So $f(X)=\{f(x)\colon x\in X\}\in T_{\alpha}$. Thus
\[
\forall y\left(y\in Y\,\Longleftrightarrow\,\left(\exists x\in X\right)\left(\left(x,y\right)\in f\right)  \,\Longleftrightarrow\,\left(\exists x\in X\right)\left(\phi(x,y,p)\right)\right)
\]
And VII holds in $T$.\medskip

\textit{Separation}. \ Let $\varphi(x,y,p)\Leftrightarrow(x=y\wedge\phi(x,p))$. Clearly, $\varphi$ is a functional formula. So for $X\in T$, by VII, $Y=\{y\colon\left(\exists x\in X\right)\varphi(x,y,p)\}\in T$ and
\[
\forall y\left(y\in Y\,\Longleftrightarrow\,\left(\exists x\in X\right)\left(x=y\wedge\phi(x,p)\right) \,\Longleftrightarrow\,\left(y\in X\wedge\phi(y,p)\right)\right)
\]
Thus VIII holds in $T$.\smallskip
\end{proof}

\subsection{Solution to Russell's paradox}

In ZF set theory, Russell’s paradox is avoided by banning all non-well-founded sets through the axiom of regularity. This is an overkill. In EZF, non-well-founded sets are allowed, and it is possible to have a set being member of itself (infinitons and semi-infinitons) as well as a vicious cycle (quasi-infinitons). But all the infinitons form a class that can not be member of itself. Hence Russell’s paradox can be avoided in the total universe. First, we have the following conclusions. 

\begin{lemma}
\label{27}\qquad
\end{lemma}

\begin{enumerate}
\item $T\vert_{\mathfrak{S}}\,=\,\,\smashoperator{{\displaystyle\bigcup}_{\alpha\in \mathrm{Ord}}}\,T_{\alpha} \vert_{\mathfrak{S}}$

\item $T\vert_{\mathfrak{Q}}\,=\,\,\smashoperator {{\displaystyle\bigcup}_{\alpha\in \mathrm{Ord}}}\,T_{\alpha} \vert_{\mathfrak{Q}}$
\end{enumerate}

\begin{proof}
\ By (\ref{TotalHierarchy}), corollary \ref{39}(v), \ref{4}(ii) and \ref{45}(v).\smallskip
\end{proof}

\begin{definition}
\label{DefNonInfinitonClass} \ The \textbf{infiniton class} $\mathcal{F}$ of the total universe includes all the semi-infinitons and quasi-infinitons in $T$. The \textbf{non-infiniton class} $\mathcal{N}$ of $T$ is the complement of $\mathcal{F}$ in $T$, i.e. $\mathcal{N}=T-\mathcal{F}$.
\end{definition}

\begin{theorem}
\label{61} \ Suppose $X\in T$. Then
\end{theorem}

\begin{enumerate}
\item \textit{If $X\in\mathcal{F}$, then $R_{T}(X)$ is a limit ordinal.}

\item \textit{If $R_{T}(X)$ is a successor ordinal, then $X\in\mathcal{N}$.}
\end{enumerate}

\begin{proof}
\ (i) \ Suppose $X\in\mathcal{F}$ and $R_{T}(X)=\alpha$ is a successor ordinal. Then $X\in T_{\alpha}=\mathcal{P} (T_{\alpha-1})$ and $X\subset T_{\alpha-1}$. If $X\in X$, then $X\in T_{\alpha-1}$ and $R_{T}(X)\leqslant
\alpha-1<\alpha$, contradiction.\medskip

If $X$ is a quasi-infiniton, suppose $X\in Y_{1},Y_{1}\in Y_{2},\dots,Y_{n}\in X$. Then by corollary \ref{56}, $R_{T}(Y_{n})<R_{T}(X)$. Thus by corollary \ref{55}, $R_{T}(X)\leqslant R_{T}(Y_{1})\leqslant\dots\leqslant R_{T}(Y_{n}) <R_{T}(X)$, contradiction again. \medskip

(ii) \ If $X\notin\mathcal{N}$, then $X\in\mathcal{F}$. By (i), $R_{T}(X)$ is a limit ordinal.\bigskip
\end{proof}

From theorem \ref{61}, lemma \ref{1} and \ref{2}, we can see that no well-founded sets are in the infiniton class, i.e. no well-founded set is a member of itself or contains a vicious cycle.

\begin{corollary}
\qquad$V\cap\mathcal{F}\,=\,\varnothing$
\end{corollary}

\begin{corollary}
\label{62} \ Suppose $X\in T$. Then
\end{corollary}

\begin{enumerate}
\item \textit{$X$ is not a semi-infiniton except $X\in T\vert_{\mathfrak{S}}$.}

\item \textit{$X$ is not a quasi-infiniton except $X\in T\vert_{\mathfrak{Q}}$.}
\end{enumerate}

\begin{proof}
\ We prove by transfinite induction.\medskip

(i) \ First, for any $X\in V_{\omega}$, $R_{T}(X)<\omega$. So by theorem \ref{61}(ii), $X\notin X$. Suppose it is true for $R_{T}(X)<\alpha$. Then for $R_{T}(X)=\alpha$, if $\alpha$ is a successor ordinal, $X\notin X$. If $\alpha$ is a limit ordinal, by corollary \ref{57}, $X\in\Bigl(\underset{\beta<\alpha}{\displaystyle\bigcup}T_{\beta}\Bigr)\Big \vert_{\aleph_{0}}$. Thus $X\in X$ if and only if $X\in\Bigl(\underset{\beta<\alpha}{\displaystyle\bigcup}T_{\beta}\Bigr)\Big\vert_{\mathfrak{S}}\subset T\vert_{\mathfrak{S}}$.\medskip

(ii) is similar to (i).\smallskip
\end{proof}

\begin{remark}
\ Corollary \ref{62} shows that all semi-infinitons and quasi-infinitons in $T$ are infinitely generated. Thus all semi-infinitons and quasi-infinitons in $T$ are precisely determined.
\end{remark}

\begin{corollary}
\label{63}\qquad
\end{corollary}

\begin{enumerate}
\item $\mathcal{F}\,=\, T\vert_{\mathfrak{S}}\cup T\vert_{\mathfrak{Q}}$

\item $\mathcal{F}\,\subset\,\,\,\smashoperator{{\displaystyle\bigcup}_{\alpha\in \mathrm{LOrd}}}\quad\underset{\beta<\alpha}{\displaystyle\bigcap} (T_{\alpha}-T_{\beta})\,\subset\,T$

\item $T\,=\,\mathcal{N}\cup\,\mathcal{F}$ \ \ \textit{and} \ \ $\mathcal{N}\cap\,\mathcal{F}\,=\,\varnothing$

\item $\mathcal{N}\,=\,\,\,\smashoperator{{\displaystyle\bigcup}_{\alpha\in \mathrm{Ord}}}\,\left(T_{\alpha}-\mathcal{F}\right)\,=\,\,\,\smashoperator{{\displaystyle\bigcup}_{\alpha\in \mathrm{Ord}}}\,\mathcal{N}_{\alpha}$

\item $\mathcal{F}$ \textit{contains no ordinals. $\mathcal{N}_{\alpha}$ contains all ordinals less than $\alpha$. $\mathcal{N}$ contains all ordinals.}
\end{enumerate}

\begin{proof}
\ (i) \ By\ definition \ref{DefNonInfinitonClass} and corollary \ref{62}. \medskip

(ii) \ By theorem \ref{59} and \ref{61}. \medskip

(iii) \ By definition \ref{DefNonInfinitonClass}. \medskip

(iv) \ By (iii) and (\ref{TotalHierarchy}).\medskip

(v) \ Since all ordinals are WF and all sets in $\mathcal{F}$ are NWF, $\mathcal{F}$ contains no ordinals. So $\mathcal{N}_{\alpha}=T_{\alpha}-\mathcal{F}$ and $\mathcal{N}$ contain the same ordinals as $T_{\alpha}$ and
$T$. And it follows by corollary \ref{53}(iii).\bigskip
\end{proof}

The total universe is shown in Figure \ref{Fig4}. The non-infiniton class is not a member of itself and the total universe, a key fact which enables us to show that the total universe is free of Russell's paradox.

\begin{figure}[h]
	\centering
	\includegraphics{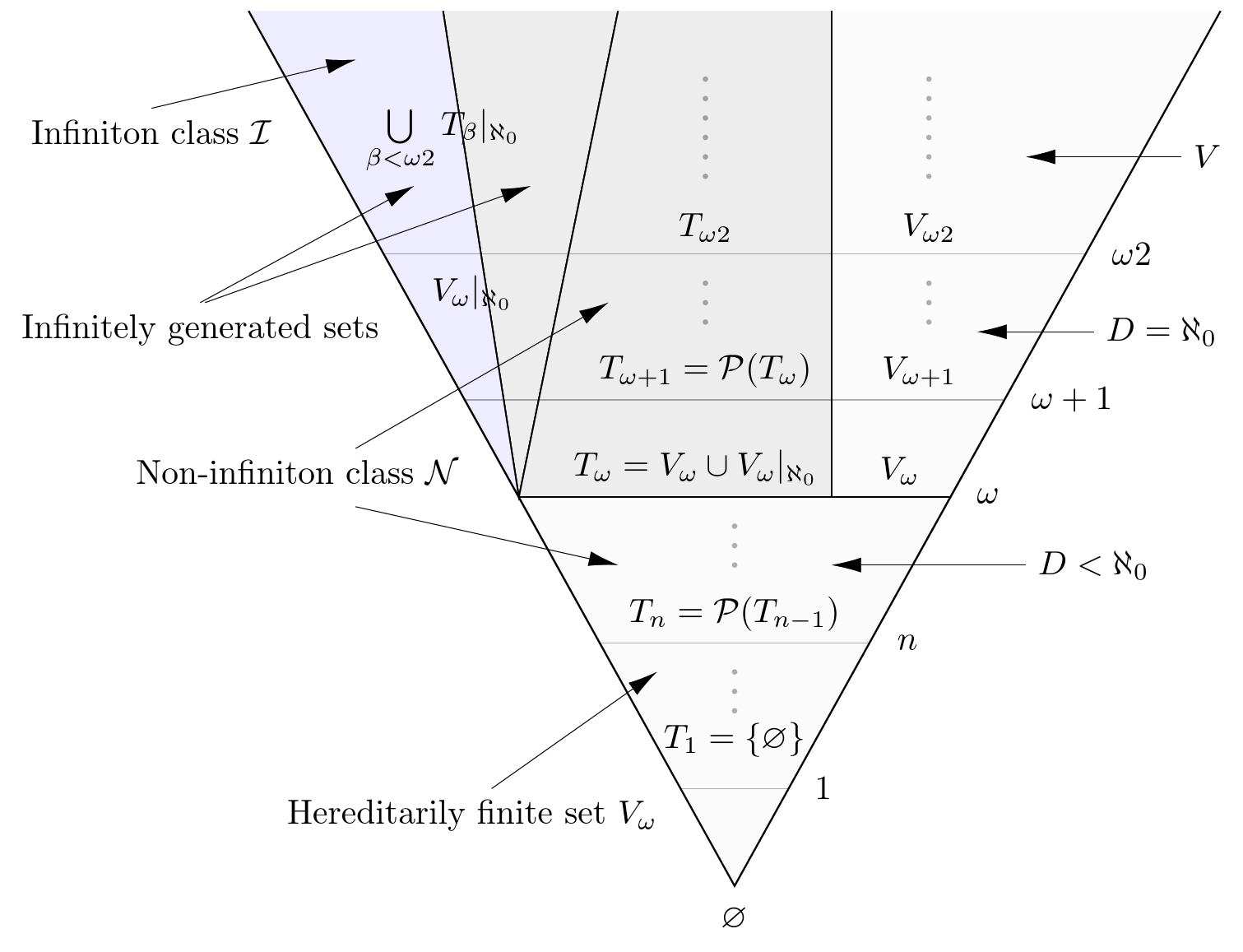}
	\caption{Diagram of the total universe.}
	\label{Fig4}
	\centering
\end{figure}

\begin{theorem}
\label{64}\qquad
\end{theorem}

\begin{enumerate}
\item $\mathcal{N\notin N},\quad\mathcal{N}\notin T,\quad T\notin T,\quad T\notin\mathcal{N}$

\item \textit{There is no vicious cycle for $\mathcal{N}$ in $T$.}

\item $T$ \textit{is free of Russell's paradox.}
\end{enumerate}

\begin{proof}
\ (i) \ By corollary \ref{63}(v), $\mathcal{N}$ contains all ordinals, but no $\mathcal{N}_{\alpha}$ contains all ordinals. For any $X\in\mathcal{N}$, suppose $X\in\mathcal{N}_{\alpha}$. So $X$ contains the ordinals less than $\alpha$. Hence $\mathcal{N}\neq X$, i.e. $\mathcal{N}\notin\mathcal{N}$. Also, by corollary \ref{53}(iii), $T$ contains all ordinals, but no $T_{\alpha}$ contains all ordinals. So $\mathcal{N}\notin T$ for no $X\in T$ containing all ordinals. The rest follow similarly.\medskip

(ii) \ Since $\mathcal{N}$ contains all ordinals and no $T_{\alpha}$ contains all ordinals, there are no $\mathcal{M}_{k}\in T$ that $\mathcal{N}\in\mathcal{M}_{1}$, $\mathcal{M}_{1}\in\mathcal{M}_{2}$, $\dots$, $\mathcal{M}_{n}\in\mathcal{N}$. \medskip

(iii) \ In the axiom of separation, $\urcorner\left(x\in y_{1},\,y_{1}\in y_{2},\,\dots,\,y_{n}\in x\right)$ must be considered along with $x\notin x$ because it can also lead to contradiction (p129 - 131, \cite{Quine}). By definition \ref{DefNonInfinitonClass} and corollary \ref{63}(i), $\mathcal{N}$ contains all the non-semi-infinitons and non-quasi-infinitons in $T$. Thus ($n=0$ reduces to $x\in x$)
\[
x\in\mathcal{N}\:\Longleftrightarrow\:x\in T\,\wedge\,\urcorner\left(x\in y_{1}\,\wedge\, y_{1}\in y_{2}\,\wedge\dots \wedge\, y_{n}\in x\right)\tag{1}
\]

For $n=0$, set $x=\mathcal{N}$ and $x=T$ in (1)
\[
\mathcal{N}\in\mathcal{N}\:\Longleftrightarrow\:\mathcal{N}\in T\,\wedge\,\mathcal{N}\notin\mathcal{N}\qquad\text{and}\qquad T\in\mathcal{N}\:\Longleftrightarrow\:T\in T\,\wedge\,T\notin T
\]

By (i), in both cases, the left and right side are false. So there is no contradiction.\medskip

For $n>0$, set $x=\mathcal{N}$ in (1)
\[
\mathcal{N}\in\mathcal{N}\:\Longleftrightarrow\:\mathcal{N}\in T\,\wedge\,\urcorner\left(\mathcal{N}\in\mathcal{M}_{1} \,\wedge\,\mathcal{M}_{1}\in\mathcal{M}_{2}\,\wedge\,\cdots\,\wedge\,\mathcal{M}_{n}\in\mathcal{N}\right)
\]

Again both sides are false and there is no contradiction. Thus $T$ is free of Russell's paradox.\smallskip
\end{proof}

\section{Conclusion\label{SectionConclusion}}

In this section we will discuss the validity of the axiom of regularity. Suppose $Z=\{\varnothing,Z\}$ is a semi-infiniton. Then $\varnothing$ is the $\in$-minimum element of $Z$. So the axiom of regularity holds for $Z$, but $Z\in Z$. This example suggests that the axiom of regularity not only can not exclude non-well-founded sets but rather holds for a (large) number of them. As a matter of fact, a well-known result that proves no set being member of itself by the axiom of regularity is actually erroneous.

\begin{conclusion}
\label{8}\ The standard theorem which uses the axiom of regularity to prove that there is no set being a member of itself is flawed.
\end{conclusion}

\begin{proof}
\ The proof is by contradiction \cite[p54]{Suppes}. First suppose $A\in A$ and $A=\{a,b,\dots,A\}$. Then $A\in A\cap\{A\}$ for $A\in\{A\}$. Since $A $ is the only member of $\{A\}$, by AR, $A\cap\{A\}=\varnothing$, contradiction. So we get $A\notin A$.\medskip

The problem in the proof is that if $A\in A$, $\{A\}\subset A$. Then $A\cap\{A\}\,=\,\{A\}\,\neq\,\varnothing$ since $A\neq\varnothing$. Thus we can not prove $A\cap\{A\}=\varnothing$, which means that AR actually can not prove that no set can be a member of itself.\bigskip
\end{proof}

%

Conclusion \ref{8} is consistent with the fact that there are non-well-founded models of ZF. For example (exercise 2.1.7 in \cite{Chang-Keisler}), suppose $\mathscr{L}=\{\in\}$ is the language of set theory and $\mathscr{L}^{\prime}=\mathscr{L}\,\cup\,\{c_i\colon i\in\omega\}$ is a new language with the distinct constants $\{c_i\colon i\in\omega\}$. Then by the compactness theorem, the theory $\Sigma=\operatorname{ZF}\,\cup\,\{c_{n+1}\in c_{n}\colon n\in\omega\}$ has a model of ZF. Another example of non-well-founded models of ZF is the ultrapower $\operatorname{Ult}_U(V)$ where $U$ is a non-principle non $\sigma$-complete ultrafilter over $\omega$. By Los's theorem, $\operatorname{Ult}_U(V)$ is a model of ZF because it is elementarily equivalent to $V$. By a well-known theorem, $\operatorname{Ult}_U(V)$ is well-founded iff $U$ is a non-principle $\sigma$-complete ultrafilter \cite{Jech}.\medskip

These examples show that the axiom of regularity is consistent with infinite $\in$-sequences. In other words, it can not ban infinite $\in$-sequences and sets being member of themselves as currently believed.\medskip

As a result, we conclude that the axiom of regularity is not valid even in defining the well-founded sets and so is dropped in this paper. This invalidity holds in any system that is consistent with ZF set theory. All well-founded sets are defined in definition \ref{DefWFAndNWFSet}, which is stronger than the axiom of regularity as in the following result.

\begin{corollary}
\label{5}\ If every branch of $S$ is finite, then there is a $x\in S$ that $x\cap S=\varnothing$. The converse is not true.
\end{corollary}

\begin{proof}
\ Suppose $x\in\{y\colon y\in S$, $R_{V}(y)=\min\}$. Then $x\cap S=\varnothing$. Conversely, let $S=\{\varnothing,S\}$. Then $S\cap\varnothing=\varnothing$, but $S$ has an infinite branch.\bigskip
\end{proof}


\appendix
\section{Introduction to Generalized Von Neumann Universe}

Due to the length and originality of this paper, a short version (summary) of it is given in Appendix A to help readers understand the key concepts and schemes in the paper. It is a separate part from the paper that can be omitted.

\subsection{Introduction}

The investigation of non-well-founded sets began with the work of Mirimanoff in 1917 \cite{Mirimanoff}. A number of axiomatic systems of non-well-founded sets have been proposed thereafter, introducing non-well-founded sets by replacing the axiom of regularity with separate anti-foundation axioms. The main problem with these systems is the lack of precise mathematical descriptions for non-well-founded sets. \smallskip

In this paper, we will present a model for precisely defining the non-well-founded sets based on the enlarged von Neumann universe ($V$). First, we point out an error in the current definition of rank in $V$. Consequently, we show that $V$ is incomplete because it does not have limit ordinal ranks. This fact is of fundamental importance because it implies that non-well-founded sets necessarily exist and should be added to $V$ as the limit ordinal ranks (initially). Furthermore, limits of finite structures and formulas along with an algebra to handle the limit operations are given for a rigorous treatment of non-well-founded sets. The expanded universe is known as the total universe, which can be shown to be a model of ZF minus the axiom of regularity and free of Russell’s paradox. The axiom of regularity is invalid in any system that is consistent with ZF set theory. \medskip


The von Neumann universe (also known as the cumulative hierarchy) is well known as the class of hereditary well-founded sets and is defined as follows:
\begin{align*}
	V_{0}\, & =\,\varnothing\mathfrak{;}\nonumber  
	\\
	V_{\alpha}\, & =\,\mathcal{P}(V_{\alpha-1})\text{,\qquad\qquad\ }\alpha\text{ is any successor ordinal;}\nonumber
	\\
	V_{\alpha}\, & =\underset{\beta<\alpha}{\bigcup}V_{\beta},\text{\qquad\qquad\quad \ }\alpha\text{ is any limit ordinal;} \nonumber  
	\\
	V \, & =\,\,\smashoperator{\bigcup_{\alpha\in \mathrm{Ord}}}\,V_{\alpha}. 
\end{align*}
Where Ord denotes the collection of all the ordinals.\medskip

The structure of any set $S$ can be represented by a tree, in which $S$ can be regarded as the root and all the objects in the transitive closure of $S$ form the nodes of the tree.  A {branch of the tree is a sequence of nodes connected by \textquotedblleft$\in$\textquotedblright\ from the root to an end node known as a terminal. Clearly, the only terminal in $V\,$is $\varnothing$. A finite branch consists of a finite number of nodes, while an infinite branch contains an infinite number of nodes. \medskip
	
	A transfinite sequence $\gamma_{\alpha}=\left\langle \gamma_{\xi}\colon \xi\leqslant\alpha\right\rangle $\ is a function with an ordinal domain where $\alpha$ is its length \cite{Jech}.\ A $\in\hspace{-0.04in}-\hspace{0.02in}$sequence is a transfinite sequence $\gamma_{\alpha}$ that $\gamma_{0}\in\cdots\in \gamma_{\xi}\in\gamma_{\xi+1}\in \cdots\in\gamma_{\alpha}$. Obviously, any branch of $S$ in $V$ can be represented by a $\in\hspace{-0.04in}-\hspace{0.02in}$sequence like $\varnothing=\gamma_{0}\in\cdots\in\gamma_{\alpha} =S$. As a result, well-founded and non-well-founded sets can be defined upon $\in\hspace{-0.04in}-\hspace{0.02in}$sequences as follows.
	
	\begin{definition}
		 \label{DefWFAndNWFSetApendix} \ Suppose $S$ is a set. Then $S$ is \textbf{well-founded} (WF) if any $\in\hspace{-0.04in}-\hspace{0.02in}$sequence of $S$ has a finite length. $S$ is \textbf{non-well-founded} (NWF) if one $\in\hspace{-0.04in}-\hspace{0.02in}$sequence of $S$ has an infinite length.
	\end{definition}
	
	We can easily show that all sets in $V$ are well-founded. The rank of a set $X$ in $V$ is defined as the least $\alpha$ that $X\in V_{\alpha+1}$. This rank is nonetheless incorrect for the following reasons. Rank in a universe of the sets is a function $R$ mapping each set to a unique ordinal number and satisfies the property of monotonicity, i.e. for any $Y\in X, \, R(Y) < R(X)$ and $R(\{X\}) = R(X) + 1$. Suppose $I_{n}=\underset{n}{\underbrace{\{\cdots\{\varnothing\}\cdots\}}}$. Then $I_n$ contains a $\in\hspace{-0.04in}-\hspace{0.02in}$sequence of length $n$: $\varnothing\in \{\varnothing\}\in\cdots\in I_n$. So by the monotonicity of $R$, we have
	\[
	R(I_{n})\,=\,R(\{I_{n-1}\})\,=\,R(I_{n-1})+1\,=\,R(\varnothing)+n
	\]
	As $n$ approaches infinity, $R(I_{\omega}) = \omega$ ($I_{\omega}$ is known as an infiniton). Therefore, the rank $\omega$ should belong to non-well-founded sets like $I_{\omega}$ that has an infinite branch rather than $\omega$ under the current definition of rank. On the other hand, even if for any $n\in \omega,\, R(n) < R(\omega)$, we cannot conclude that $R(\omega) = \omega$ because $\omega + 1$ also qualifies. These facts suggest that rank in $V$ should be modified as follows. 
	
	\begin{definition}
		\label{DefModifiedCumulativeHierarchyRankApendix}\ The \textbf{rank} of $X$ in $V$ is defined as the least $\alpha$ that $X\in V_{\alpha}$ and denoted as $R_{V}(X)$.
	\end{definition}
	
	From the above definition, we can prove easily the following fundamental result which shows $V$ is incomplete and needs be expanded.
	
	\begin{lemma}
		\label{2Apendix} \ No set in $V$ has a limit ordinal rank.
	\end{lemma}
	
	\begin{proof}
		\ Suppose $\alpha$ is a limit ordinal and $R_{V}(X)=\alpha$. Then there is a $\gamma<\alpha$ that $X\in V_{\gamma}$ or $R_{V}(X)<\alpha$, contradiction.\bigskip
	\end{proof}
	
	Since $R(I_{\omega}) = \omega$, some non-well-founded sets like $I_{\omega}$ should be added to $V$ as limit ordinal ranks to form a complete universe of sets. Therefore, non-well-founded sets not only exist but form an integral part with well-founded sets in the new universe of sets. \medskip
	
	Furthermore, the following result easily follows from definition \ref{DefModifiedCumulativeHierarchyRankApendix} and shows that the rank of any ordinal in $V$ is a successor ordinal, which is consistent with the fact that $V$ contains no rank of a limit ordinal.
	
	\begin{corollary}
		\label{3Apendix} \ For any von Neumann ordinal $\alpha\in V,\,R_{V}(\alpha)\,=\,\alpha+1$ (corollary \ref{3}).
	\end{corollary}
	
%

	\subsection{Limit of Structures and Formulas}
	
	In order to study non-well-founded sets rigorously, we must extend the current infinite model theory. As a result, limits of finite structures and formulas are proposed and investigated. The limit of finite structures is an infinite structure that can be described by an infinitely long formula involving countably many conjunctions and disjunctions and finitely many quantifiers \cite{Scott}.\smallskip
	
	An atomic theory in a language is a theory in which every of its consistent formulas can be derived from a complete formula in it. An atomic structure is a structure in which every of its tuples satisfies a complete formula in its theory. A $\aleph_0$-categorical theory in a language has exactly one countable structure up to isomorphism which is known as the $\aleph_0$-categorical structure of it. By a famous theorem of Engeler and Ryll-Nardzewski, a $\aleph_0$-categorical structure consists of finitely many atomic structures \cite{Chang-Keisler}. The limits of finite structures and formulas are investigated under a $\aleph_0$-categorical theory. 
	
	\begin{definition} 
		\label{DefHomoSeqApendix}\ Suppose $L$ is an infinitary language. Let $\phi_n$ be formulas (of finite length) in a $\aleph_0$-categorical theory of $L$ and $M_n$ be $L$-structures described by $\phi_n$. If for any $n < k < \omega,\, M_k \vDash \phi_n$, then $\{(M_n,\phi_n) \colon M_n \vDash \phi_n \wedge n < \omega\}$ is known as a \textbf{homogeneous sequence} of structures and formulas. 
	\end{definition}
	
	We can show that for a homogeneous sequence of structures, there exist a unique (countable) atomic structure $M$ (up to isomorphism) and a unique formula $\phi$ in $L$ (up to equivalence) such that $M \vDash \phi$ (theorem \ref{71}). So we have
	
	\begin{definition}
		\label{DefLimitApendix}\ Suppose $(M_n,\phi_n)$ is a homogeneous sequence. Its unique structure and formula are known as the \textbf{limit} of $M_n$ and $\phi_n$, denoted as $\underset{n\rightarrow \omega}{\lim} M_n=M$ and $\underset{n\rightarrow\omega}{\lim} \phi_n\Leftrightarrow\phi$. In both cases, we say the limit of $\phi_{n}$ or the limit of $M_{n}$ is unique. If a subsequence of pair $\{(M_{n_i},\phi_{n_i})\colon M_{n_i} \vDash \phi_{n_i} \,\wedge\, n_i < \omega\}$ is homogeneous, its limit is known as a \textbf{sublimit} of $M_n$ and $\phi_n$.
	\end{definition}
	
	The limit of formulas and structures satisfies the following axioms. Suppose $\phi_n$ and $\varphi_n$ are consistent formulas in a $\aleph_{0}$-categorical theory. If $\underset{n\rightarrow\omega}{\lim} \phi_n$ and $\underset{n\rightarrow\omega}{\lim} \varphi_n$ are unique, then
	
	\begin{enumerate}
		\item[I.] $\underset{n\rightarrow\omega}{\lim}\left(\phi_{n} \wedge\varphi_{n}\right)$ is unique and
		\[
		\underset{n\rightarrow\omega}{\lim}\left(\phi_{n}\wedge\varphi_{n}\right)\,\Longleftrightarrow\,\underset{n\rightarrow\omega}{\lim}\phi_{n} \wedge\underset{n\rightarrow\omega}{\lim}\,\varphi_{n}
		\]
		
		\item[II.] $\underset{n\rightarrow\omega}{\lim} \urcorner\phi_{n}$ is unique and
		\[
		\underset{n\rightarrow\omega}{\lim}\urcorner\phi_{n}\,\Longleftrightarrow\,\urcorner\,\underset{n\rightarrow\omega}{\lim}\phi_{n}
		\]
		
		\item[III.] $\underset{n\rightarrow\omega}{\lim}\exists x \,\phi_{n}$ is unique ($x$ is a variable in $\phi_{n}$), and 
		\[
		\underset{n\rightarrow\omega}{\lim}\exists x\,\phi_{n}\,\Longleftrightarrow\,\exists x\,\underset{n\rightarrow\omega}{\lim}\phi_{n}
		\]
		
		\item[IV.]  Suppose $M_{n}$ are $L$-structures in a $\aleph_{0}$-categorical theory,  $\underset{n\rightarrow\omega}{\lim}M_{n}$ is unique and $N$ is a substructure of $M_{n}$. Then
		\[
		\underset{n\rightarrow\omega}{\lim}\left(N\in M_{n}\right)\,\Longleftrightarrow\,\left(N\in\underset {n\rightarrow\omega}{\lim}M_{n}\right)
		\]	
	\end{enumerate}
	
	I$,\cdots,$IV are known as the \textbf{limit algebra} for structures and formulas. From these and calculus of logic, we can prove the following results.
	
	\begin{corollary}
		\ Suppose $\phi_{n}$ and $\varphi_{n}$ are consistent, $\underset{n\rightarrow\omega}{\lim}\phi_{n}$ and $\underset{n\rightarrow\omega}{\lim}\,\varphi_{n}$ are unique in a $\aleph_{0}$-categorical theory. Then
	\end{corollary}
	
	\begin{enumerate}
		\item $\underset{n\rightarrow\omega}{\lim}\left(\phi_{n}\vee\varphi_{n}\right) \,\Longleftrightarrow\,\underset{n\rightarrow\omega}{\lim} \phi_{n}\vee\underset{n\rightarrow\omega}{\lim}\,\varphi_{n}$
		
		\item $\underset{n\rightarrow\omega}{\lim}\left(\phi_{n}\Longrightarrow\varphi_{n}\right)\,\Longleftrightarrow\,\left(\underset {n\rightarrow\omega}{\lim}\phi_{n}\,\Longrightarrow\,\underset{n\rightarrow\omega}{\lim}\,\varphi_{n}\right)$
		
		\item $\underset{n\rightarrow\omega}{\lim}\left(\phi_{n}\Longleftrightarrow\varphi_{n}\right)\,\Longleftrightarrow\,\left(\underset {n\rightarrow\omega}{\lim}\phi_{n}\,\Longleftrightarrow\,\underset{n\rightarrow\omega}{\lim}\,\varphi_{n}\right)$
		
		\item $\underset{n\rightarrow\omega}{\lim}\,\forall x\,\phi_{n}\,\Longleftrightarrow\,\forall x\,\underset{n\rightarrow\omega}{\lim}\phi_{n}$
	\end{enumerate}
	
	\begin{theorem}\label{1Apendix}\ Suppose $A$, $M_{n}$ and $N_{n}$ are $L$-structures in a $\aleph_{0}$-categorical theory, $\underset{n\rightarrow\omega}{\lim}M_{n}$ and $\underset{n\rightarrow\omega}{\lim}N_{n}$ are unique, $K$ is a structure without relation and function symbols, and $\underset{n\rightarrow\omega}{\lim}\phi_n$ is unique. 
	\end{theorem}
	
	\begin{enumerate}
		\item $\underset{n\rightarrow\omega}{\lim}\,K\cup\{M_{n}\}\,=\,K\cup\left\{\underset{n\rightarrow\omega}{\lim} M_{n}\right\}$
		
		\item $\underset{n\rightarrow\omega}{\lim}\,\left(\forall{A}\in{M}_{n}\right)\phi_{n}\,\Longleftrightarrow\,\bigl(\forall {A}\in\underset{n\rightarrow\omega}{\lim}\,{M}_{n}\bigr)\underset{n\rightarrow\omega}{\lim}\phi_{n}$
		
		\item $\underset{n\rightarrow\omega}{\lim}\,\left(\exists{A}\in{M}_{n}\right)\phi_{n}\,\Longleftrightarrow\,\bigl(\exists {A}\in\underset{n\rightarrow\omega}{\lim}\,{M}_{n}\bigr)\underset{n\rightarrow\omega}{\lim}\phi_{n}$
		
		\item $\underset{n\rightarrow\omega}{\lim}\,\left(\exists\,N_{n}\in M_{n}\right)  \phi_{n}\,\Longleftrightarrow\,\left(\exists\,\underset{n\rightarrow \omega}{\lim}\,N_{n}\in\underset{n\rightarrow\omega}{\lim} M_{n}\right)\underset{n\rightarrow\omega}{\lim}\phi_{n}$
	\end{enumerate}
	
	Refer to section \ref{SectionOfLimitFormula} for more details.
	
	\subsection{Non-Well-Founded Sets}
	
	Mirimanoff initiated the study of non-well-founded sets in 1917 \cite{Mirimanoff}. Several axiomatic systems of non-well-founded sets such as AFA (by Aczel, Forti and Honsell \cite{Aczel}), SAFA (by Scott), FAFA (by Finsler), and BAFA (by Boffa), have been proposed thereafter. Since ZF set theory bans $\in\hspace{-0.04in}-\hspace{0.02in}$sequences of infinite length by the axiom of regularity (also known as the axiom of foundation), most of these systems incorporate non-well-founded sets by replacing the axiom of regularity with distinct anti-foundation axioms and are essentially models of ZF minus the axiom of regularity. A notable exception is New Foundations by Quine \cite{Quine} that allows non-well-founded sets without a specific axiom and avoids Russell’s paradox by permitting only stratified formulas. \smallskip
	
	The main problem of the above systems is the lack of precise mathematical descriptions for non-well-founded sets. For instance, a non-well-founded set such as a Quine atom in AFA is an accessible pointed graph that can be unfolded into an infinite tree. However, AFA only describes it intuitively and does not provide enough mathematical rigor for its description and operations.\smallskip
	
	In this paper, we will present a new scheme to generate non-well-founded sets by enlarging the von Neumann universe along with the exact process to generate these sets. Since $V$ is incomplete for lack of the limit ordinal ranks, non-well-founded sets are necessarily existent and some of them should take on the limit ordinal ranks in a complete universe of sets. Thus, non-well-founded sets are added to $V$ initially as infinitely generated sets with limit ordinal ranks. The expanded universe is known as the total universe ($T$). Furthermore, limits of finite structures and formulas provide rigorous analysis for three types of infinitely generated sets as infinitons, semi-infinitons and quasi-infinitons that appear in Russell’s paradox. Refer to section \ref{SectionNonWellFoundedSets} for more details.
	
	\subsubsection{Infinitely Generated Sets}
	
	An infinitely generated set is a generator of non-well-founded sets that contains only one infinite branch. It is the limit of finitely generated sets which are well-founded. For example, $\underset{n}{\underbrace{\{\cdots\{\varnothing\}\cdots\}}}$ discussed in the previous section is well-founded, and so is $\underset{n}{\underbrace{\{a,\{a, \cdots\{a,\varnothing\}\cdots\}}}$, where $a\in V_{\omega}$. Likewise, $a$ can be replaced by any collection of elements in $V_{\omega}$, i.e.
	\[
	Z_n\,=\,\underset{n}{\underbrace{\{ a_1,a_2,\cdots,\{ a_1,a_2,\cdots,\cdots\{ a_1,a_2,\cdots,\varnothing\}\cdots\}}}
	\]
	and $Z_n$ is a finitely generated set. Therefore, for simplicity and accuracy, we need a new symbol for the collection of elements without the curly brackets and call it the \textbf{unpacking operator} (definition \ref{DefUnpackOperator}). Let $G = \{a_1, a_2,\cdots\}$. Then $\ast G$ is the unpacking operator of G and denotes $a_1, a_2,\cdots$. For the above example, $Z_n$ can be simply written as $\{\ast G, \{\ast G, \cdots \{*G, \{\varnothing\}\}\cdots\}$. Refer to section \ref{SectionUnpackingOperator} for more details.\smallskip
	
	\begin{definition}
		\ Suppose $G_{k}\in V_{\omega} \,\left(0\leqslant k\leqslant n\right)$. A \textbf{finitely generated set} is a finite $L$-structure:
		\begin{equation*}
			H_{n}(G_{n},\cdots,G_{0})\,=\,\{\ast G_{n},\{\ast G_{n-1},\cdots\{\ast G_{1},G_{0}\}\cdots\}
		\end{equation*}
		Where $G_{k} \left(1\leqslant k\leqslant n\right)$ are principal generators and $G_{0}$ is a base generator of $H_{n}$.	
	\end{definition}
	
	An infinitely generated set is defined as the limit of finitely generated sets. 
	
	\begin{definition}
		\label{DefInfGeneratedSetApendix} \ Suppose $H_{n}$ is defined above and $\mathcal{G}=\{G_{n}\colon G_{n}\in V_{\omega},\,n<\omega\}$. An \textbf{infinitely generated set (IGS)} (at $\omega$) is defined as:
		\begin{equation*}
			H_{\omega}(\mathcal{G})\,=\,\underset{n\rightarrow\omega}{\lim}\,H_{n}(G_{n},\cdots,G_{0})
		\end{equation*}
		Where $G_{n} \left(n\geqslant1\right)$\ are principal generators and $G_{0}$ is a base generator of $H_{\omega}$.
	\end{definition}
	
	Note that each IGS has one infinite branch. A non-well-founded set with multiple infinite branches can be formed from IGS through power set operations. Next, we will discuss three types of non-well-founded sets that appear in Russell’s paradox as special cases of infinitely generated sets.
	
	\subsubsection{Infiniton}
	
	An \textbf{infiniton} is a set that contains itself as the only member, i.e. $I = \{I\}$. Intuitively, $I$ has infinite membership for 
	\begin{equation*}
		I=\underset{\aleph_{0}}{\,\underbrace{\{\dots\{G_0\}\dots\}}\,}=\underset{\aleph_{0}+1}{\,\underbrace{\{\{\dots\{G_0\}\dots\}\}}\,}=\{\underset{\aleph_{0}}{\,\underbrace{\{\dots\{G_0\}\dots\}}\,}\}  =\,\{I\}
	\end{equation*}
	Fortunately, this process can be precisely determined by the limit of finite structures as follows. Suppose $I_n = \{I_{n-1}\} =\cdots = \underset{n}{\underbrace{\{ \cdots\{G_0\}\cdots\}}}$ ($G_0$ is the base generator of $I_n$). Then $I_n$ is described by:
	\[
	\phi_{n}(I_n)\,\Longleftrightarrow\,(\left(Y_{n}=I_n\right)\wedge\,\,\smashoperator{\bigwedge_{1\leqslant j\leqslant n}}\,\,\exists!\,Y_{j} \,\exists!\,Y_{j-1}\left(Y_{j-1}\in Y_{j}\right))
	\]
	Since for any $k>n,\,{I}_k\vDash \phi_n[I_k]$, $(I_n, \phi_{n})$ is a homogeneous sequence. So there is a unique atomic structure $I$ such that $\underset {n\rightarrow\omega}{\lim}I_{n}=I$. By theorem \ref{1Apendix}
	\[
	I\,=\underset{n\rightarrow\omega}{\lim}I_{n}\,=\underset{n\rightarrow\omega}{\lim}\{I_{n-1}\}\,=\,\left\{\underset{n\rightarrow\omega}{\lim}I_{n-1}\right\}\,=\,\{I\}
	\] 
	And the complete formula for $I$ is:
	\[
	\phi\,\Longleftrightarrow\underset{n\rightarrow\omega}{\lim}\,\phi_{n}\,\Longleftrightarrow\underset{n<\omega}{\displaystyle\bigwedge}\exists! \,I_{n}\,\exists!\,I_{n-1}\left(I_{n-1}\in I_{n}\right)\,\wedge\,\exists!\,I\left(I\in I\right)
	\]
	Furthermore, we have the following conclusion. 
	\begin{corollary}
		\label{18Apendix}	\ AR fails for any infiniton and collections of infinitons.
	\end{corollary}
	
	\begin{proof}
		Suppose $I = \{I\}$ and $S = \{I_{\alpha}\colon I_{\alpha} = \{I_{\alpha}\}\}$ ($\alpha\in D$). Since the only member of $I$ is $I$ itself and $I \cap I = I \ne \varnothing$, AR fails for $I$.\smallskip
		
		For each $I_{\alpha}\in S$, $I_{\alpha} \cap S = I_{\alpha}$ for $I_{\alpha}\in I_{\alpha}$. Since no $y\in S$ satisfies $y \cap S = \varnothing$, AR fails for $S$ too.\bigskip
	\end{proof}
	
	Refer to section \ref{SectionInfiniton} for more details.
	
	\subsubsection{Semi-Infiniton}
	
	A \textbf{semi-infiniton} is a set that is a member of itself, i.e. $Z\in Z$. Suppose $G\,=\,\{a_{1},a_{2},\dots\}$ ($G\in V_{\omega}$) and $\ast G$ is the unpacking operator. Intuitively, $Z$ takes on the following form,
	\begin{equation*}
		Z\,=\,\underset{\aleph_{0}}{\underbrace{\{\ast G,\{\ast G,\dots\{\ast G,G_0\}\dots\}}}
	\end{equation*}
	for
	\[
	Z\,=\,\underset{\aleph_{0}+1}{\underbrace{\{\ast G,\{\ast G,\dots\{\ast G,G_0\}\dots\}}}\,=\,\{\ast G,\underset{\aleph_{0}}{\underbrace{\{\ast G,\dots\{\ast G,G_0\}\dots\}}}\}\,=\,\{\ast G,Z\}
	\]
	Thus $Z$ is precisely the limit of finitely generated sets in which all the principal generators are the same, i.e. $G_1= \cdots = G_n$. More specifically, suppose $Z_{n}\,=\,\underset{n}{\underbrace {\{\ast G,\{\ast G,\cdots\{\ast G,G_{0}\}\cdots\}}}$ ($G$ is the principal generator and $G_0$ the base generator of $Z_n$). Then $Z_n$ is described by:
	\[
	\varphi_{n}(Z_n)\,\Longleftrightarrow(\left(Y_{n}=Z_n\right)\wedge\,\smashoperator{\bigwedge_{1\leqslant j\leqslant n}}\,\,\exists!\,Y_{j} \left(\exists!\,Y_{j-1}\left(Y_{j-1}\in Y_{j}\right)\,\wedge\,\left(\forall z\in G\right)\left(z\in Y_{j}\right)\right))
	\]
	Since for any $k>n,\,{Z}_k\vDash \varphi_n[Z_k]$, $(Z_n, \varphi_{n})$ is a homogeneous sequence. So there is a unique atomic structure $Z$ that $\underset {n\rightarrow\omega}{\lim}Z_{n}=Z$. By theorem \ref{1Apendix}
	\[
	Z\,=\underset{n\rightarrow\omega}{\lim}\{\ast G,Z_{n-1}\}\,=\underset{n\rightarrow\omega}{\lim}\left(G\,\cup\{Z_{n-1}\}\right) \,=\,G\,\cup\left\{\underset{n\rightarrow\omega}{\lim}Z_{n-1}\right\}\,=\,\{\ast G,Z\}
	\] 
	And the complete formula for $Z$ is:
	\[
	\varphi\,\Longleftrightarrow\underset{n<\omega}{\displaystyle\bigwedge}\exists!\,Z_{n}\left(\exists!\,Z_{n-1}\left(Z_{n-1}\in Z_{n}\right) \,\wedge\,\left(\forall z\in G\right)\left(z\in Z_{n}\right)\right)\,\wedge\,\exists!\,Z \left(\left(Z\in Z\right)\,\wedge\,\left(\forall z\in G\right) \left(z\in Z\right)\right)
	\]
	
	Refer to section \ref{SectionSemiInfiniton} for more details.
		
	\subsubsection{Quasi-Infiniton}
	
	A \textbf{quasi-infiniton} is a set that contains a vicious cycle, e.g. $Q\in S_1, S_1\in S_2, \cdots, S_{n}\in Q$. $Q$ is the limit of finitely generated sets in which all the principal generators form a finite cycle. In other words, there are only $n\,(n > 1)$ distinct principal generators $G_j\,(1 \leqslant j \leqslant n)$ for $Q$. We can show that there are $n$ sublimits in $Q$, each of which forms an atomic substructure generated by a cycle of $G_j$. Therefore, $Q$ is overall a $\aleph_0$-categorical structure consisting of $n$ atomic substructures. 	Refer to section \ref{SectionQuasiInfiniton} for more details.
	
	\subsection{Total Universe}
	
	As discussed in earlier sections, the von Neumann universe is incomplete because it does not have limit ordinal ranks. The infinitely generated sets as generators of non-well-founded sets should take on limit ordinal ranks in a complete universe of sets. This new universe known as the total universe contains both well-founded and non-well-founded sets and is an expansion of the von Neumann universe. The \textbf{total universe} is defined as follows.
	\begin{align*}
		T_{0} \, & =\,\varnothing;  \nonumber  
		\\
		T_{\alpha} \, & =\,\mathcal{P}(T_{\alpha-1})\text{,\qquad\qquad\qquad\qquad\,}\alpha\operatorname{\: is\:any\: successor \:ordinal;} \nonumber  
		\\
		T_{\alpha} \,& =\underset{\beta<\alpha}{\bigcup} T_{\beta}\,\cup\Bigl(\underset{\beta<\alpha}{\bigcup} T_{\beta}\Bigr)\Big \vert_{\aleph_{0}},\text{\qquad\quad}\alpha\operatorname{\: is\: any\: limit\: ordinal;}  \nonumber
		\\
		T \, & =\,\,\smashoperator{\bigcup_{\alpha\in \mathrm{Ord}}}\,T_{\alpha}.
	\end{align*}
	Where $\Bigl(\underset{\beta<\alpha}{\bigcup} T_{\beta}\Bigr)\Big \vert_{\aleph_{0}}$ is the IGS generated from $\underset{\beta<\alpha}{\bigcup} T_{\beta}$ at the limit ordinal $\alpha$. $T$ has the same hierarchy as $V$ with the difference at each limit ordinal where IGS is created. 
	
	\subsubsection{Rank in Total Universe}
	
	Rank in $T$ is the same as that of $V$ (definition \ref{DefModifiedCumulativeHierarchyRankApendix}). 
	
	\begin{definition}\label{DefTotalModelRankApendix}
		\ The rank of $X$ in $T$ is defined as the least $\alpha$ that $X\in T_{\alpha}$ and denoted as $R_T(X)$.
	\end{definition}
	
	We have the following results on rank in the total universe.
	\begin{lemma}
		\ If $R_T(X)$ is a limit ordinal, then $X$ is an IGS.
	\end{lemma}
	
	\begin{proof}
		\ Suppose $R_T(X)=\alpha$ is a limit ordinal. If $X\in \underset{\beta<\alpha} {\bigcup}T_{\beta}$, then there is a $\gamma<\alpha$ such that $X\in T_{\gamma}$, i.e. $R_T(X) \leqslant\gamma< \alpha$, contradiction. So $X\in \Bigl(\underset{\beta<\alpha}{\bigcup} T_{\beta}\Bigr)\Big \vert_{\aleph_{0}}$.
	\end{proof}
	
	\begin{lemma}
		\ Suppose $X\in T$ and $Y\in X$.
	\end{lemma}
	
	\begin{enumerate}
		\item $R_T(Y) \leqslant R_T(X)$.
		
		\item If $R_T(X)$ is a successor ordinal, $\,R_T(Y) < R_T(X)$.
	\end{enumerate}
	
	\begin{proof}
		\ (i) \ Suppose $R_{T}(X)=\alpha$. Then by definition \ref{DefTotalModelRankApendix}, $X\in T_{\alpha}$. Since $T_{\alpha}$ is transitive, for $Y\in X$, $Y\in T_{\alpha}$, i.e. $R_{T}(Y)\leqslant\alpha$.\medskip
		
		(ii) \ Suppose $R_{T}(X)=\alpha$ is a successor ordinal. Then $X\in T_{\alpha}\,=\,\mathcal{P}(T_{\alpha-1})$ and $X\subset T_{\alpha-1}$. Thus $Y\in T_{\alpha-1}$ and $R_{T}(Y)\leqslant\alpha-1<R_{T}(X)$.\smallskip
	\end{proof}
	
	\begin{theorem}
		\ (Partition by Rank in $T$) Suppose $A_{\alpha}$ is the collection of elements with rank $\alpha$ in $T$.
	\end{theorem}
	
	\begin{enumerate}
		\item If $\alpha$ is a successor ordinal, 	$\,A_{\alpha} = T_{\alpha}- T_{\alpha-1}$.
		
		\item If $\alpha$ is a limit ordinal,  $\,A_{\alpha}=\,T_{\alpha}\,-\underset{\beta<\alpha}{\displaystyle\bigcup}T_{\beta}\,=\,\underset{\beta<\alpha}{\displaystyle\bigcap}\left(T_{\alpha}-T_{\beta}\right)$
		
		\item $T$ is partitioned by rank as,
		\[
		T\,\,=\,\,\,\smashoperator{{\displaystyle\bigcup}_{\alpha\in \mathrm{SOrd}}}\,\,(T_{\alpha}-T_{\alpha-1})\,\,\cup\,\,\, \smashoperator{{\displaystyle\bigcup}_{\alpha\in \mathrm{LOrd}}} \quad\underset{\beta<\alpha}{\displaystyle\bigcap}(T_{\alpha}-T_{\beta})
		\]
	\end{enumerate}
	Where SOrd denotes all successor ordinals and LOrd all limit ordinals. See theorem \ref{59} for a detailed proof.\smallskip
	
	\begin{theorem}
		\ Suppose $X\in T$.
	\end{theorem}
	
	\begin{enumerate}
		\item If $R_T(X)$ is a successor ordinal, $\,R_T(X)\,=\,\sup\{R_T(Y)\colon Y\in X\} + 1$.
		
		\item If $R_T(X)$ is a limit ordinal, $\,R_T(X)\,=\,\sup\{R_T(Y)\colon Y\in X\}$.
	\end{enumerate}
	See theorem \ref{60} for a detailed proof. For more results, refer to section \ref{SectionRankTotalUniverse}.\smallskip
	
	\subsubsection{Set Theory for Total Universe}
	
	An extension of ZF set theory known as EZF is proposed to handle both the well-founded and non-well-founded sets. The language of EZF is the language of ZF with the primitive symbol of \textquotedblleft$\in$\textquotedblright.  Since there are non-well-founded sets in EZF, infinite $\in\hspace{-0.04in}-\hspace{0.02in}$sequences are allowed. By corollary \ref{18Apendix}, the axiom of regularity can fail and is replaced by the non-well-foundedness axiom in EZF. Also, the extensionality and union axioms in ZF must be modified to handle infinite $\in\hspace{-0.04in}-\hspace{0.02in}$sequences. First, we extend the union operator and transitive closure to the transfinite case.
	
	\begin{definition}
		\label{DefGenUnionOperatorApendix} \ Suppose ${\bigcup}S=\{z\colon\exists y\left(y\in S\,\wedge\,z\in y\right)\}$. The $\mathbf{\alpha}^{th}$ \textbf{union operator} is defined (recursively) as:
	\end{definition}
	
	\begin{enumerate}
		\item \textit{If $\alpha$ is a successor ordinal, then}
		\[
		{\textstyle\bigcup}^{\alpha}S\,=\,{\textstyle\bigcup}\:{\textstyle\bigcup}^{\alpha-1}S\quad \left({\textstyle\bigcup}^{0}S\,=\,S\right)
		\]
		
		\item \textit{If $\alpha$ is a limit ordinal, then}
		\[
		{\textstyle\bigcup}^{\alpha}S\,=\,\underset{\beta<\alpha}{\bigcup}{\textstyle\bigcup}^{\beta}S
		\]
	\end{enumerate}
	
	\begin{definition}
		\label{TransitiveClosureApendix} \ Suppose $\alpha_{0}$ is the least ordinal $\alpha$ that \ ${\bigcup}^{\alpha}S\,=\,{\bigcup} ^{\alpha+1}S$. Then the \textbf{transitive closure} of $S$ is:
		\[
		TC(S)\,=\underset{\alpha\leqslant\alpha_{0}}{\bigcup}{\textstyle\bigcup}^{\alpha}S	
		\]
	\end{definition}
	The \textbf{axioms of EZF} are as follows.
	\begin{enumerate}
		\item[I.] Extensionality.	If $X$ and $Y$ are IGS with the same set of generators (principle and base), then $X=Y$. For other sets, it is the same as that of ZF (\cite{Jech} and \cite{Suppes}).
		
		\item[II.] Non-well-foundedness.  There exists a non-well-founded set.
		
		\item[III.]	Union.  For any set $X$ and any ordinal $\alpha$, ${\textstyle\bigcup}^{\alpha}X$ is a set.
		
		%
		%
		%
		%
		
	\end{enumerate}
	
	The rest axioms IV$,\dots,$VIII are the axioms of pairing, power set, infinity, replacement, and separation that are the same as those in ZF. For more details, refer to section \ref{SectionSetTheoryTotalUniverse}. \medskip
	
	A main result is that the total universe is a model of EZF, i.e. ZF minus AR (theorem \ref{72}). 
	
	\subsubsection{Solution to Russell’s paradox}
	
	In ZF, Russell’s paradox is avoided by banning all non-well-founded sets through the axiom of regularity. This is an overkill. In EZF, non-well-founded sets are allowed, and it is possible to have a set being member of itself (infinitons and semi-infinitons). But all the infinitons form a class that can not be member of itself. Hence Russell’s paradox is avoided in the total universe as well.\smallskip
	
	Let $\mathcal{F}$ be the \textbf{infiniton class} in $T$, denoting the collection of all infinitons, semi-infinitons and quasi-infinitons in the total universe. Then the \textbf{non-infiniton class} $\mathcal{N}$ of $T$ is the complement of $\mathcal{F}$ in $T$, i.e. $\mathcal{N}=T-\mathcal{F}$. First, we have the following results. 
	
	\begin{lemma}
		\label{11Apendix}\ Any element in $\mathcal{F}$ has rank of a limit ordinal. Any element with rank of a successor ordinal is in $\mathcal{N}$.
	\end{lemma}
	
	\begin{lemma}
		\label{12Apendix}\ Let $\mathcal{N}_{\alpha}$ be the non-infiniton class of $T_{\alpha}$. 
	\end{lemma}
	
	\begin{enumerate}
		\item $T=\mathcal{F}\cup\mathcal{N},\quad\mathcal{F}\cap\mathcal{N}= \varnothing$
		
		
		\item $\mathcal{N}\,=\,\,\,\smashoperator{{\displaystyle\bigcup}_{\alpha\in \mathrm{Ord}}}\,\left(T_{\alpha}-\mathcal{F}\right)\,= \,\,\,\smashoperator{{\displaystyle\bigcup}_{\alpha\in \mathrm{Ord}}}\,\mathcal{N}_{\alpha}$
		
		\item $\mathcal{F}$ contains no ordinals. $\mathcal{N}_{\alpha}$ contains all ordinals less than $\alpha$. $\mathcal{N}$ contains all ordinals.
	\end{enumerate}
	
	See theorem \ref{61} and corollary \ref{63} for detailed proofs.
	
	\begin{theorem}
	\end{theorem}
	
	\begin{enumerate}
		\item $\mathcal{N}\notin\mathcal{N},\quad\mathcal{N}\notin T,\quad T\notin T,\quad T\notin\mathcal{N}$.
		
		\item There is no vicious cycle in $\mathcal{N}$. 
		
		\item $T$ is free of Russell’s paradox.
	\end{enumerate}
	
	\begin{proof}
		(i) \ By lemma \ref{12Apendix}(iii), $\mathcal{N}$ contains all ordinals, but no $\mathcal{N}_{\alpha}$ contains all ordinals. For any $X\in\mathcal{N}$,  suppose $X\in\mathcal{N}_{\alpha}$. Then $X$ only contains the ordinals less than $\alpha$. Hence $\mathcal{N}\neq X$, i.e. $\mathcal{N}\notin\mathcal{N}$. The rest follow similarly.\medskip
		
		(ii) \ Since $\mathcal{N}$ contains all ordinals and no $T_{\alpha}$ contains all ordinals, there are no $\mathcal{M}_{k}\in T$ that $\mathcal{N}\in\mathcal{M}_{1}$, $\mathcal{M}_{1}\in\mathcal{M}_{2}$, $\cdots$, $\mathcal{M}_{n}\in\mathcal{N}$. \medskip
		
		(iii) \ In the axiom of separation, $\urcorner\left(x\in y_{1},\,y_{1}\in y_{2},\,\cdots,\,y_{n}\in x\right)$ must be considered along with $x\notin x$ because it can also lead to contradiction (p129 - 131, \cite{Quine}). Since $\mathcal{N}$ contains all the non-semi-infinitons and non-quasi-infinitons in $T$, we have ($n=0$ reduces to $x\in x$)
		\[
		x\in\mathcal{N}\:\Longleftrightarrow\:x\in T \, \wedge\,\urcorner\left(x\in y_{1}\,\wedge\, y_{1}\in y_{2}\,\wedge\cdots \wedge\, y_{n}\in x\right)\tag{1}
		\]
		
		For $n=0$, set $x=\mathcal{N}$ and $x=T$ in (1), we have
		\[
		\mathcal{N}\in\mathcal{N}\:\Longleftrightarrow\:\mathcal{N}\in T \,\wedge\,\mathcal{N}\notin\mathcal{N}\qquad\text{and}\qquad T\in\mathcal{N}\:\Longleftrightarrow\:T\in T\,\wedge\,T\notin T
		\]
		
		By (i), in both cases, the left and right side are false. So there is no contradiction.\medskip
		
		For $n>0$, set $x=\mathcal{N}$ in (1)
		\[
		\mathcal{N}\in\mathcal{N}\:\Longleftrightarrow\:\mathcal{N}\in T\,\wedge\,\urcorner\left(\mathcal{N}\in\mathcal{M}_{1} \,\wedge\,\mathcal{M}_{1}\in\mathcal{M}_{2}\,\wedge\,\cdots\,\wedge\,\mathcal{M}_{n}\in\mathcal{N}\right)
		\]
		
		Again both sides are false and there is no contradiction. Thus $T$ is free of Russell's paradox.
	\end{proof}
	
	\subsection{Conclusions}
	
	In this section we will discuss the validity of the axiom of regularity. Suppose $Z=\{\varnothing,Z\}$ is a semi-infiniton. Then $\varnothing$ is the $\in$-minimum element of $Z$. So the axiom of regularity holds for $Z$, but $Z\in Z$. This example suggests that the axiom of regularity not only can not exclude non-well-founded sets but rather holds for a (large) number of them. As a matter of fact, a well-known result that proves no set being member of itself by the axiom of regularity is actually erroneous.
	
	\begin{conclusion}
		\label{14Apendix}\ The standard theorem which uses the axiom of regularity to prove that there is no set being a member of itself is flawed.
	\end{conclusion}
	
	\begin{proof}
		\ The proof is by contradiction \cite[p54]{Suppes}. First suppose $A\in A$ and $A=\{a,b,\cdots,A\}$. Then $A\in A\cap\{A\}$ for $A\in\{A\}$. Since $A $ is the only member of $\{A\}$, by AR, $A\cap\{A\}=\varnothing$, contradiction. So we get $A\notin A$.\medskip
		
		The problem in the proof is that if $A\in A$, $\{A\}\subset A$. Then $A\cap\{A\}\,=\,\{A\}\,\neq\,\varnothing$ since $A\neq\varnothing$. Thus we can not prove $A\cap\{A\}=\varnothing$, which means that AR actually can not prove that no set can be a member of itself.\bigskip
	\end{proof}
	
	
	Conclusion \ref{14Apendix} is consistent with the fact that there are non-well-founded models of ZF. For example (exercise 2.1.7 in \cite{Chang-Keisler}), suppose $L=\{\in\}$ is the language of set theory and $L^{\prime}=L\,\cup\,\{c_i\colon i\in\omega\}$ is a new language with the distinct constants $\{c_i\colon i\in\omega\}$. Then by the compactness theorem, the theory $\Sigma=\operatorname{ZF}\,\cup\,\{c_{n+1}\in c_{n}\colon n\in\omega\}$ has a model of ZF. Another example of non-well-founded models of ZF is the ultrapower $\operatorname{Ult}_U(V)$ where $U$ is a non-principle non $\sigma$-complete ultrafilter over $\omega$. By Los's theorem, $\operatorname{Ult}_U(V)$ is a model of ZF because it is elementarily equivalent to $V$. By a well-known theorem, $\operatorname{Ult}_U(V)$ is well-founded iff $U$ is a non-principle $\sigma$-complete ultrafilter \cite{Jech}.\medskip
	
	These examples show that the axiom of regularity is consistent with infinite $\in$-sequences. In other words, it can not ban infinite $\in$-sequences and sets being member of themselves as currently believed.\medskip
	
	As a result, we conclude that the axiom of regularity is not valid even in defining the well-founded sets and so is dropped in EZF. This invalidity holds in any system that is consistent with ZF set theory. All well-founded sets are defined in definition \ref{DefWFAndNWFSetApendix}, which is stronger than the axiom of regularity as shown in the following result. 
	
	\begin{corollary}
		\label{5Apendix}\ If every branch of $S$ is finite, then there is a $x\in S$ that $x\cap S=\varnothing$. The converse is not true.
	\end{corollary}
	
	\begin{proof}
		\ Suppose $x\in\{y\colon y\in S$, $R_{V}(y)=\min\}$. Then $x\cap S=\varnothing$. Conversely, let $S=\{\varnothing,S\}$. Then $S\cap\varnothing=\varnothing$, but $S$ has an infinite branch.\bigskip
	\end{proof}

\textit{E-mail}: e353zhan@uwaterloo.ca

\end{document}